\documentclass[11pt]{amsart}
\usepackage[utf8]{inputenc}
\usepackage{amsmath,amssymb,amsthm, graphics}
\usepackage{amsthm}
\usepackage{diagmac2}
\usepackage{textcomp}
\usepackage[alphabetic]{amsrefs}
\usepackage[all,cmtip]{xy}
\usepackage{yfonts}
\usepackage{tikz}
\usepackage{fullpage}
\usepackage{stmaryrd}
\usepackage{mathrsfs}
\usepackage{caption}
\usepackage[mathscr]{euscript}

\newcommand{\cC}{\mathcal{C}}
\newcommand{\cF}{\mathcal{F}}
\newcommand{\cE}{\mathcal{W}}
\newcommand{\sD}{\mathscr{D}}

\newcommand{\sC}{\mathscr{C}}
\newcommand{\sF}{\mathscr{F}}

\newcommand{\kC}{\mathfrak{C}}
\newcommand{\kM}{\mathfrak{M}}
\newcommand{\cM}{\mathcal{M}}
\newcommand{\sA}{\mathscr{A}}
\newcommand{\sB}{\mathscr{B}}
\newcommand{\sE}{\mathscr{E}}
\newcommand{\sG}{\mathscr{G}}
\newcommand{\sH}{\mathscr{H}}

\newcommand{\sL}{\mathscr{L}}
\newcommand{\sM}{\mathscr{M}}
\newcommand{\sS}{\mathscr{S}}
\newcommand{\sX}{\mathscr{X}}
\newcommand{\sZ}{\mathscr{Z}}
\newcommand{\cL}{\mathcal{L}}

\newcommand{\cX}{\mathcal{X}}

\newcommand{\cO}{\mathcal{O}}
\newcommand{\cS}{\mathcal{S}}

\newcommand{\spec}{\operatorname{Spec}}
\newcommand{\Kg}{\mathcal{K}}

\newcommand{\sY}{\mathscr{Y}}
\newcommand{\supp}{\operatorname{Supp}}

\newcommand{\Aut}{\operatorname{Aut}}

\newcommand{\Id}{\operatorname{Id}}

\newcommand{\I}{\operatorname{I}}
\newtheorem{theorem}{Theorem}[section]
\newtheorem{Def}[theorem]{Definition}
\newtheorem{Teo}[theorem]{Theorem}

\newtheorem{Lemma}[theorem]{Lemma}
\newtheorem{Oss}[theorem]{Observation}
\newtheorem{Cor}[theorem]{Corollary}
\newtheorem{Prop}[theorem]{Proposition}
\newtheorem{Notation}[theorem]{Notation}

\newtheorem{Remark}[theorem]{Remark}

\newtheorem{Step_1}{Blow-up}

\theoremstyle{definition}
\newtheorem{EG}[theorem]{Example}
\usepackage{filecontents}

\begin{document}
\title{Moduli of Weierstrass fibrations with marked section}
\author{Giovanni Inchiostro}
\begin{abstract} We study the the moduli space of KSBA stable pairs $(X,sS+\sum a_i F_i)$, consisting of a Weierstrass fibration $X$, its section $S$, and some fibers $F_i$. We find a compactification which is a DM stack, and we describe the objects on the boundary. We show that the fibration in the definition of Weierstrass fibration extends to the boundary, and it is equidimensional when $s \ll 1$. We prove that there are wall-crossing morphisms when the weights $s$ and $a_i$ change. When $s=1$, this recovers the work of La Nave \cite{LaNave}; and a special case of the work of Ascher-Bejleri \cite{AB3}. 
\end{abstract}
\maketitle 

\section{Introduction}
Fibered surfaces have been intensively studied,
since the 1800 and the Italian school of algebraic geometry.
In this project, we will focus our attention on a particular class of fibered surfaces, namely Weierstrass fibrations.
These can be understood as surfaces $X$ with a morphism $f:X \to C$ to a smooth curve, such that $f$ admits a section $S$, and such that the fibers are genus one curves (see Definition \ref{Def:minimal:W:fib} for a precise definition).

To better understand Weierstrass fibrations, it is natural to ask for a parameter space.
The problem of constructing a moduli space of elliptic surfaces, and in particular Weierstrass fibrations, has been approached using several techniques. In \cite{miranda-moduli}, Miranda constructs a coarse moduli space using GIT, in the case where the base curve has genus 0. Seiler tackles the case where the base curve has higher genus in \cite{moulielliptic}. Finally recall that an elliptic fibration comes with the $j$-invariant map to the coarse moduli space $\overline{M}_{1,1}$. One can try to lift it to $\overline{\mathcal{M}}_{1,1}$ and use the general machinery of twisted stable maps of Abramovich and Vistoli \cite{AV1} to construct a moduli space of elliptic surfaces (see also \cite{AB2}).

The approach we follow in this project is through the minimal model program. By definition, an elliptic surface $X$ comes with a choice of a divisor, namely the section $S$. Therefore, coupling this classical
theory with the modern tools of the MMP, one can understand
the pair $(X,S)$ as a stable pair in the sense of Koll\'ar, Shepherd-Barron and Alexeev; and produce a moduli space using the MMP. Our first result in this direction is the following (see Subsection \ref{Subsection:construction:parameter:space:EI} and Proposition \ref{Prop:we:have:the:universal:curve}): 

\begin{Teo}\label{Teo:intro:there:is:a:mod:space}
There is a proper Deligne-Mumford stack $\cE_I$ which parametrizes the following stable pairs. On the interior, it parametrizes pairs $(X,sS+\sum a_i F_i)$ consisting of a Weierstrass fibration $X$ and a $\mathbb{Q}$-divisor $sS+\sum a_i F_i$ where $S$ is the section and $F_i$ are some fibers. On the boundary, it parametrizes the stable surface pairs described in Corollary \ref{Cor:description:lc:degeneration}. The subscript $I$ is an admissible weight vector (see Definition \ref{Def:weighted:W:fib}) and keeps track of $s$ and $a_i$.

Moreover, if $\sX_I \to \cE_I$ is the universal surface, there is a family of curves $\sC_I \to \cE_I$ and a morphism $\sX_I \to \sC_I$ satisfying the following condition. For every point $p$ in the interior of $\cE_I$, the morphism $(\sX_I)_p \to (\sC_I)_p$ is the fibration to a curve in the definition of Weierstrass fibration. 
\end{Teo}
The problem of understanding Weierstrass fibrations and their moduli spaces through the MMP, has been approached by several authors. La Nave in \cite{LaNave} first finds the stable limits of Weierstrass fibrations, using the twisted stable maps of Abramovich and Vistoli. Brunyate in  \cite{brunyate2015modular} addresses
the case in which the Weierstrass fibrations are elliptic K3 surfaces. In loc. cit. the author produces a proper moduli space which on the interior parametrizes elliptic K3 surfaces, with weighted section and some weighted fibers. 

Recently Ascher-Bejleri pushed the results in \cite{LaNave} even further. In \cite{AB3} they consider pairs $(X \to C,S+F_{\mathcal{A}})$
consisting of an elliptic surface $X$ with its morphism $f:X \to C$ to a smooth curve;
and a $\mathbb{Q}$-divisor $S+\sum a_i F_i \subseteq X$ consisting of the marked section
$S$ and some marked fibers $F_i$.
Given a weight vector $\mathcal{A}:=(a_1,...,a_n)$, they
construct a proper moduli space $\mathcal{E}_\mathcal{A}$ which on the interior parametrizes such objects, and they prove a theorem analogous to Theorem \ref{Teo:intro:there:is:a:mod:space} for $\mathcal{E}_\mathcal{A}$ (\cite{AB3}*{Theorems 1.1 and 1.2}). The main goal of this project is to understand how the results in loc. cit. change, when the section $S$ comes with a weight $s \le 1$, and $a_i$ are small.

In the work of La Nave, the one of Ascher-Bejleri and in this project, the main technical difficulty boils down to the existence of \emph{pseudoelliptic} components (see Definition \ref{Def:pseudoelliptic}). These are surface pairs which appear as irreducible components of stable limits of a Weierstrass fibration. One can understand them as a birational model of a Weierstrass fibration $(X,S+\sum a_i F_i)$, obtained from $(X,S+\sum a_i F_i)$ by performing some birational transformations that contract the section $S$. These surface pairs may \emph{not} admit a nonconstant morphism to a curve, causing the fibration $\sX_I \to \sC$ in Theorem \ref{Teo:intro:there:is:a:mod:space} to be \emph{not} even pure dimensional.

The main advantage of working with a section marked with $s \ll 1$ lies in the following:
\begin{Teo}[Theorem \ref{Teo:no:psudo}]\label{Teo:intro:have:flat:fibration} If $s$ is small enough compared to the weights $a_i$, the morphism $\sX_I \to \sC_I$ of Theorem \ref{Teo:intro:there:is:a:mod:space} is equidimensional (a priori, \emph{not} flat), with irreducible fibers.
\end{Teo}
In particular Theorem \ref{Teo:intro:there:is:a:mod:space} and Theorem \ref{Teo:intro:have:flat:fibration} give a compactification of the moduli space of stable Weierstrass fibrations by a proper DM stack, such that:
\begin{itemize}
\item The boundary parametrizes simpler objects: there are no pseudoelliptic components;
\item The fibers of the morphism $\sX_I \to \sC$ are \emph{irreducible} (possibly non-reduced) curves.
\end{itemize}

Now, the moduli spaces constructed in Theorem \ref{Teo:intro:there:is:a:mod:space}, depend on the weights $s$ and $a_1,...,a_n$. It is natural to ask how these moduli change when we vary $s$ and $a_i$. In fact, Ascher and Bejleri investigate a similar question for their moduli spaces $\mathcal{E}_\mathcal{A}$.
They produce wall crossing morphisms, when the weights $\mathcal{A}$ change (\cite{AB3}*{Theorem 1.5}), which generalize the ones of the Hassett spaces \cite{Hassett} to the case of elliptic surfaces. Therefore, it is natural to ask whether our moduli spaces also preserve this wall-crossing behaviour. The answer is yes.

In particular, assume that there are $0 < t\le s$ and $0 \le b_i \le a_i$ such that, for every stable Weierstrass fibration $(X,sS+\sum a_i F_i)$, the surface pair $(X,tS+\sum b_i F_i)$ is still stable.
Then the assignment $(X,sS+\sum a_i F_i) \mapsto (X,tS+\sum b_i F_i)$ induces a morphism of moduli
$$\xymatrix{
{\left\{\begin{matrix} \text{Stable Weierstrass fibrations} \\ \text{with weights }(s,a_i)
\end{matrix} \right \}} \ar[r]^-{r} & {\left\{\begin{matrix} \text{Stable Weierstrass fibrations } \\ \text{with weights }(t,b_i)
\end{matrix} \right\}} 
}
$$
But Theorem \ref{Teo:intro:there:is:a:mod:space} produces a compactification of the moduli of stable Weierstrass fibrations. Therefore it is natural to ask whether there is a morphism $R$ extending $r$ as below: 
$$\xymatrix{
{\left\{\begin{matrix} \text{Stable Weierstrass fibrations} \\ \text{with weights }(s,a_i)
\end{matrix} \right \}} \ar[r]^-{r} \ar@{_{(}->}[d] & {\left\{\begin{matrix} \text{Stable Weierstrass fibrations } \\ \text{with weights }(t,b_i)
\end{matrix} \right\}} \ar@{^{(}->}[d]\\
\overline{{\left\{\begin{matrix} \text{Stable Weierstrass fibrations} \\ \text{with weights }(s,a_i)
\end{matrix} \right \}}} \ar@{.>}[r]^-{R} & \overline{{\left\{\begin{matrix} \text{Stable Weierstrass fibrations } \\ \text{with weights }(t,b_i)
\end{matrix} \right\}}} 
}
$$
In Theorems \ref{Teo:wall:and:chamb:dec:stable:models} and \ref{Teo:wall:crossings:morphisms} we show that in fact such an $R$ exists:
\begin{Teo}\label{Teo:intro:extension:morphism}The morphism $r$ defined on $k$-points as above is algebraic, and does extend to a morphism $R_{I,I'}:\cE_I \to \cE_{I'}$. 
Moreover, these reduction morphisms induce a finite wall and chamber decomposition on the space of all admissible weights, such that if $I$ and $I'$ are in the same open chamber, then $\cE_I \cong \cE_{I'}$.
\end{Teo}
The main example we keep in mind for understanding such a wall and chamber decomposition, is the work of Hassett in \cite{Hassett}. There are at least two generalizations of \cite{Hassett}, in the case of higher dimensional varieties. One is the work of Ascher-Bejleri we discussed above (\cite{AB3}). A second one is the paper of Alexeev, on weighted hyperplane
arrangements (\cite{Alexeevhyper}).

The main difference between Theorem \ref{Teo:intro:extension:morphism} and the analogous \cite{AB3}*{Theorem 1.5} lies in its proof. In fact, in loc. cit. the authors prove their result through a vanishing theorem (\cite{AB3}*{Theorem 1.4}), to prove that the log-plurigenera commutes with base change. Using that our objects admit a degeneration to a log-canonical pair, in the case where $a_i$ are small, we provide a simplified version of \cite{AB3}*{Theorem 1.4} in Theorem \ref{Teo:cohom:vanishing}. The main advantage of Theorem \ref{Teo:cohom:vanishing} is that it does \emph{not} rely on an explicit description of the stable limits of a Weierstrass fibration, and it holds in higher dimensions. Similar results are proved by Koll\'ar in \cite{kollar2018log} and \cite{kollar2018logs}.

Finally, we provide a more explicit description of the reduction morphisms of Theorem \ref{Teo:intro:extension:morphism}.
To achieve that, we attach a combinatorial object to every surface pair parametrized by $\cE_I$, namely the \emph{refined numerical data} (Definition \ref{Def:refined:num:data}).
The main feature of such a combinatorial gadget lies in the following theorem (see Corollary \ref{Cor:image:determided:by:rnd}):
\begin{Teo}
Given a point $p: \spec(k) \to\cE_I$, its image through $R_{I,I'}$ is uniquely determined by the refined numerical data of $p$. 
\end{Teo}

One can understand the refined numerical data as a generalization, to the case of elliptic surfaces, of the dual weighted graph of the Hassett stable curves. Indeed, the reduction morphisms of the Hassett spaces (\cite{Hassett}), on $k$-points, can be explicitly understood using the dual weighted graph of a weighted stable curve.
In particular, one can use the refined numerical data to understand what birational transformations we need to perform to go from the surface pairs parametrized by $p$, to the one parametrized by $R_{I,I'}(p)$.

The paper proceeds as follows. In Section
\ref{Section:background:tsm:and:stable:pairs}, we recall the properties that we will use about twisted stable maps, and the minimal model program. In Section \ref{Section:definitions:of:the:objects:we:parametrize} we recall the background definitions about elliptic surfaces that we need for the remaining part of the paper, and we define the objects parametrized by $\cE_{I}$.
In Section \ref{Section:intermediate:fibers} 
we further study the numerical properties and the singularities for the objects parametrized by $\cE_I$.
Section \ref{Section:construction:of:the:moduli:surface:pairs} is devoted to the construction of the moduli space $\cE_I$, using the results in \cite{KP}.
In Section \ref{section:singularities:threefold}, we study the surfaces parametrized on the boundary of $\cE_I$, using the MMP and the results of the author in Appendix B to \cite{AB3}. Section
\ref{Section:QCartier:threshold} is the most technical section. First, we study the steps of the MMP one has to perform to obtain the stable limits in $\cE_I$. Then we show that there is a finite wall and chamber decomposition on the set of all possible weights, such that for $I$ and $I'$ in the same open chamber, $\cE_I$ and $\cE_{I'}$ parametrize the same objects. We begin Section \ref{Section:cohom:vanishing:and:w:crossing} by outlining the strategy we follow to produce $R_{I,I'}$, and we apply such a strategy to show that there are wall-crossing morphisms for the moduli spaces $\cE_I$.
In Section \ref{Section:universal:curve} we show that there is a universal curve $\sC_I \to \cE_I$ as in Theorem \ref{Teo:intro:there:is:a:mod:space}, and we prove Theorem \ref{Teo:intro:have:flat:fibration}.

We work over an algebraically closed field of characteristic 0. 

\begin{bf}Acknowledgements. \end{bf} I thank my advisors Dan Abramovich for his constant support and many helpful discussions. I am also grateful for insightful
conversations with Shamil Asgarli, Dori Bejleri, Justin Lacini and Luca Schaffler. I am thankful to Kenneth Ascher who carefully read a preliminary draft of this project.  Research supported in part by funds from
NSF grant DMS-1500525.
\section{Background on Twisted stable maps and stable pairs}\label{Section:background:tsm:and:stable:pairs}
This section is divided into three subsections.
In the first one, we recall the relevant definitions from \cite{AV1}, \cite{AOV} and \cite{AV2}. In the second one, we briefly discuss
the results about the
MMP that are needed in the remaining part of the paper. In the last one we focus on stable pairs.

\subsection{Twisted stable maps} In this first subsection we recall the results in \cite{AV2} that we need in the remaining part of the paper.
We begin with the definition of twisted stable maps. Recall also that for us $\operatorname{char}(k)=0$, so DM stacks are tame.
\begin{Def}
 Fix a base scheme $S$ and a DM stack $\sM$ with projective coarse moduli space $M$, and fix an ample line bundle on $M$.
 A \underline{twisted stable $n$-pointed map of genus $g$ and degree $d$ to $\sM$}, is the data of a triple $(\cC \to S ,\{\Sigma\}_{i=1}^n \to \cC, \cC \to \sM)$
 consisting of:
 \begin{itemize}
  \item A DM stack $\cC$ and a proper morphism $\cC \to S$ of relative dimension 1, such that étale locally $\cC \to S$ is a nodal curve;
   \item $n$ closed substacks $\Sigma_i \to \cC$, with coarse spaces $\sigma_i$ such that $\Sigma_i \to S$ is an étale gerbe;
  \item If $\pi:\cC \to C$ is the coarse space of $\cC$, then $\pi$ is an isomorphism over the smooth points of
  $\cC \to S$ away from $\Sigma_i$;
  \item A representable morphism $\cC \to \sM$, such that the induced morphism on coarse spaces $(C, \{\sigma_i\}_{i=1}^n) \to M$
  is a Kontsevich stable map of degree $d$, from a family of $n$-pointed genus $g$ curves;
  \end{itemize}
\end{Def}

When $S=\spec(k)$, one can understand $\cC$ as an orbifold nodal curve, with stacky structure along some smooth points (corresponding to $\Sigma_i$) and on some of the nodes.
The second bullet point ensures that the stacky structure along the smooth points of $\cC$ varies regularly.

Now, one can define a category fibered over $\cS ch/\spec(k)$, having as objects twisted stable $n$-pointed maps of genus $g$ and degree $d$ to $\sM$.
The morphisms from $(\cC \to S ,\{\Sigma\}_{i=1}^n \to \cC, \cC \to \sM)$ to
$(\cC' \to S' ,\{\Sigma'\}_{i=1}^n \to \cC', \cC' \to \sM)$ are a morphism $S \to S'$, and morphisms $\cC \to \cC'$ which induces an isomorphism
$\cC \to S \times_{S'} \cC'$, and such that the obvious
diagrams commute.
Following the notation in \cite{AV2}, we will denote this fibered category with $\mathcal{K}_{g,n}(\sM,d)$.

In \cite{AV2}*{Theorem 1.4.1} the authors, among other things, prove the following:
\begin{Teo}
 $\mathcal{K}_{g,n}(\sM,d)$ is a proper DM stack. 
\end{Teo}
\subsection{Minimal model program}
In this subsection we recall the definitions and constructions of the MMP and the moduli of stable pairs we need.
For a reference on the definitions of the singularities we will deal with, one can consult \cite{KM} and \cite{Kollarsing}.
\begin{Def}\label{Def:lc:sing}
 Let $X$ be a normal variety, let $D_i \subseteq X$ be some prime divisors and let $\Delta:=\sum a_i D_i $ be a linear combination
 with $a_i \in \mathbb{Q}_{\le 1}$. The pair $(X,\Delta)$ is \underline{log-canonical}, or lc, (resp. \underline{Kawamata-} \underline{log-terminal}, or klt)
 if $K_X+\Delta$ is $\mathbb{Q}$-Cartier (resp. $\mathbb{Q}$-Cartier and $a_i<1$) and,
 given a log-resolution $f:Y \to X$ of $(X,\sum D_i)$, with exceptional divisors $\{E_j\}_{j=1}^n$,
 for $m$ divisible enough, we can write
 $$\cO_Y(m(K_Y+\sum a_i f_*^{-1}(D_i))) \cong f^*(\cO_{X}(m(K_X+\sum a_i D_i)))\otimes \cO_Y(\sum m b_j E_j)$$
 with $b_j \ge -1$ (resp. $b_j >-1$).
\end{Def}
In what follows, we will always assume that $a_i \ge 0$.
For example, if $\Delta =0$ and $X$ is a surface,
Du Val singularities are klt, and elliptic singularities are lc but not klt.
For an example of a normal surface singularity which is not lc one can take $x^4+y^4+z^4=0$.

The standard generalization of Definition \ref{Def:lc:sing} to schemes which are not normal is the following:
\begin{Def}
 Let $X$ be a reduced $S_2$ scheme, which in codimension 1 has only nodal singularities. Let $D_i \subseteq X$ be
 some irreducible divisors, which intersect the smooth locus of $X$, and let $\Delta:=\sum a_i D_i $ be a linear combination
 with $a_i \in \mathbb{Q}_{\le 1}$. Consider $n:X^n \to X$ the normalization of $X$, let $D \subseteq X^n$ be the preimage of the double locus of $X$,
 and let $\Delta^n:=n_*^{-1}(\Delta)$.
 The pair $(X,\Delta)$ is \underline{semi-log canonical} (or slc) if:
 \begin{enumerate}
  \item $K_X+\Delta$ is $\mathbb{Q}$-Cartier, and
  \item The pair $(X^n,D+\Delta^n)$ is lc.
 \end{enumerate}
 Moreover, a slc pair $(X,\Delta)$ is stable if $K_X+\Delta$ is ample and $\Delta$ is effective.
\end{Def}
Now, assume we are given a lc pair $(X,D)$, with $\dim(X) \le 3$ and $K_X+D$ big. It is proven in \cite{KMMcK} that
$\bigoplus_m H^0(\cO_X(\lfloor m K_X+mD \rfloor))$ is a finitely generated algebra. If we define
$X^s:=\operatorname{Proj}(\bigoplus_m H^0(\cO_X(n m (K_X+D))))$ for $n$ divisible enough, there is a birational morphism
$\pi:X \dashrightarrow X^s$. Moreover, if $D^s:=\pi_*(D)$, then $(X^s,D^s)$ is a stable pair. We define $(X^s,D^s)$ to be the \underline{stable} \underline{model} of $(X,D)$.
One can understand $(X^s,D^s)$ as \emph{the} birational model of $(X,D)$ which is stable.
\subsection{Stable pairs}\label{Subsection:stable:pairs} In the previous subsection,
we recalled the definition of stable pairs, as a canonical birational model of a lc pair. 
Since such a canonical model is unique, one could try to construct a moduli space of stable pairs: we review the results in that direction that we will need.

In dimension 1, the stable pairs $(X,\Delta)$ are the Hassett stable curves (\cite{Hassett}).
In loc. cit. the author introduces a
smooth
DM stack, $\overline{\mathcal{M}}_{g,\mathcal{A}}$, which is a moduli space for stable pairs $(C,\Delta)$ where $C$ is a curve of genus $g$ and the coefficients of $\Delta$ are in $\mathcal{A}$.

For higher dimensional stable pairs, the definition of the moduli functor presents some difficulties.
Indeed, for each slc curve $(C,\sum_{i=1}^n a_i p_i)$,
the divisor $\supp(\sum a_i p_i)$ is a Cartier divisor. This may not hold in higher dimensions, and one needs to find a suitable definition for a family of divisors. 

It turns out that if the base scheme $S$ is normal, then defining a family of divisors is a more approachable problem (see \cite{kollarbook}*{Chapter 4}, in particular Theorem 4.21). In particular, there is a well defined notion of stable varieties over $S$, which is the following:
\begin{Def}[see \cite{KP}*{Definition 2.11 and 5.2}]\label{Def:stable:pairs:normal:base}
 Let $S$ be a \emph{normal} scheme, and let $\mathcal{A} \subseteq [0,1]$ a finite subset closed
 under addition. A \underline{stable variety $(X,D) \to S$} consists of a proper flat morphism $f:X \to S$ of relative dimension $n$,
 with a $\mathbb{Q}$-divisor $D \subseteq X$. Moreover, we require that:
 \begin{itemize}
  \item For every
 $s \in S$ we have that $D_s \subseteq X_s$ is a divisor with coefficients in $\mathcal{A}$;
  \item For every $s \in S$, the restriction $D_s \subseteq X_s$ is a divisor and the pair $(X_s,D_s)$ is stable, and
  \item $K_X+D$ is $\mathbb{Q}$-Cariter.
 \end{itemize}

\end{Def}

In \cite{KP} the two authors, among other things, present a particular moduli pseudo-functor of stable surface pairs (\cite{KP}*{Definition 5.6}),
and construct a proper DM stack
which represents it. We summarize the results we need about their construction (see \cite{KP}*{Section 5}).
\begin{Def}\label{Def:KP:stable:pairs}
 Let $I \subseteq [0,1]$ be a finite subset closed under addition, let $v,n,m > 0$ be three integers, and let $S$ be a scheme.
 A \underline{family of stable pairs} with coefficient set in $I$, volume $v$ and dimension $n$ over $S$,
 is the data of a flat proper morphism $f:X \to S$ of relative dimension $n$, a line bundle $\sL$ on $X$, and a map
 $\phi: \omega_{X/S}^{\otimes m} \to \sL$. This data must satisfy the following requirements:
 \begin{itemize}
  \item $\sL$ is a relatively very ample line bundle with $R^if_*(\sL)=0$ for $i>0$, and $(\sL)^n =vm^n$;
  \item For every $s \in S$, the morphism $\phi_s$ is an isomorphism at the generic points and the codimension one singular points of $X_s$;
  \item For every $s \in S$, the morphism $\phi_s$ it determines
  a divisor $D_s$ with coefficients in $I$, such that $(X_s,D_s)$ is slc and $\sL_s \cong \cO_{X_s}(m(K_{X_s}+D_s))$.
 \end{itemize}
\end{Def}
We remark that for $m$ big enough, Definition \ref{Def:KP:stable:pairs} and Definition \ref{Def:stable:pairs:normal:base} agree over a normal base. Therefore, when we have a normal base $B$, we will use Definition \ref{Def:stable:pairs:normal:base} and we will write $(X,D) \to B$ to denote a family of stable pairs over $B$. 

Now, fix a number $v >0$. Then for $m$ divisible enough, in
\cite{KP}*{Notation 5.13 and Proposition 5.14} the two authors construct a proper DM stack $\sM_{n,v,I}$
which has as objects over a reduced base $S$, families of stable pairs of dimension $n$, coefficient set in $I$, and volume $v$ over $S$.
For a description of the morphisms, see Definition 5.6 in loc. cit.
To fix the notation, we will use $\sM_{n,v,I}$ for a moduli space of stable pairs. This choice is not essential, since we will work only over seminormal
(in fact, most of the time normal) bases (see also Proposition \ref{proposition:iso:dm:stacks}).
\begin{Notation}
We will denote $\sM_{v,I}:=\sM_{2,v,I}$.
\end{Notation}

\section{Background on elliptic surfaces}\label{Section:definitions:of:the:objects:we:parametrize}
In the first subsection we recall the definitions due to La Nave \cite{LaNave} and Ascher-Bejleri \cite{AB3} that we need in the rest
of the paper. In the second one we recall some of the results in \cite{LaNave}.
\subsection{Weierstrass fibrations and elliptic surfaces} 
We start by recalling the definition of minimal Weierstrass fibration (see \cite{Miranda}*{Definition II.3.2 and Proposition III.3.2}).
\begin{Def}\label{Def:minimal:W:fib}
 A \underline{minimal Weierstrass fibration} is a normal, projective and irreducible surface $X$ with a surjective morphism $f:X \to C$ to a smooth curve $C$, and a section $\sigma:C \to X$,
 satisfying the following conditions:
 \begin{itemize} 
\item Every fiber of $f$ is irreducible, and is either a smooth genus 1 curve, or a rational curve with either a node or a cusp, and
  \item $\sigma(C)$ is contained in the smooth locus of $f$, and the singularities of $X$ are Du Val.
 \end{itemize}
\end{Def}
We remark that, if we drop the hypothesis on the singularities being Du Val, we obtain a \emph{Weierstrass fibration} (see Definition \cite{Miranda}*{Definition II.3.2}).
The singular fibers of a minimal resolution of a Weierstrass fibration were classified by Kodaira and Neron, and one can consult \cite{Miranda} and \cite{Schutt:Shioda} for a modern account (see Table \ref{tab:table3} and Remark \ref{Remark:table:sing:fibers}). 
 \begin{table}[!ht]
  \begin{center}\caption{Singular fibers.}
  \label{tab:table3}
    \begin{tabular}{|c|c|c|} 
      \hline Kodaira's fiber type & Dual graph of the fiber &Picture\\
       \hline  \raisebox{1.5ex}{$\operatorname{I}_n^*$, $\text{ }n \ge 0$} &
      \begin{tikzpicture}[scale=.4]
    \draw (-3,0) node[anchor=east]  {$D_n$};
    \foreach \x in {0}
    \draw[xshift=\x cm,thick] (\x cm,0) circle (.3cm);
    \foreach \x in {1}
    \draw[xshift=\x cm,thick] (\x cm,0) circle (.3cm);
    \foreach \x in {2}
    \draw[xshift=\x cm,thick] (\x cm,0) circle (.3cm);
    \foreach \x in {3}
    \draw[xshift=\x cm,thick] (\x cm,0) circle (.3cm);
    \draw[xshift=6 cm,thick] (30: 17 mm) circle (.3cm);
    \draw[xshift=6 cm,thick] (-30: 17 mm) circle (.3cm);
    \draw[thick] (0.3 cm,0) -- +(1.4 cm,0);
    \draw[dotted,thick] (2.3 cm,0) -- +(1.35 cm,0);
    \foreach \y in {2.15}
    \draw[xshift=\y cm,thick] (\y cm,0) -- +(1.4 cm,0);
    \draw[xshift=6 cm,thick] (30: 3 mm) -- (30: 14 mm);
    \draw[xshift=6 cm,thick] (-30: 3 mm) -- (-30: 14 mm);
    \draw[xshift=-0.4 cm,thick] (1.3 mm, 0.1cm) -- +(-30: -14 mm);
    \draw[xshift=-0.4 cm,thick] (1.3 mm, -0.1cm) -- +(30: -14 mm);
    \draw[xshift=-3.3 cm,thick,fill=black] (1.57cm, -0.85cm ) circle [radius=.3];
    \draw[xshift=-3.3 cm,thick] (1.57cm, 0.85cm )circle [radius=.3] ;
  \end{tikzpicture}&\includegraphics[scale=0.25]{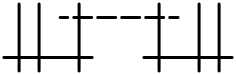}\\
      \hline  \raisebox{1.7ex}{$\operatorname{I}_n$, $\text{ }n \ge 1$} &
      \begin{tikzpicture}[scale=.4]
    \draw (-1,0) node[anchor=east]{ Cycle};
    \foreach \x in {0,...,4}
    \draw[xshift=\x cm,thick] (\x cm,0) circle (.3cm);
    \draw[dotted,thick] (0.3 cm,0) -- +(1.4 cm,0);
    \foreach \y in {1.15,...,3.15}
     \draw[xshift=\y cm,thick] (\y cm,0) -- +(1.4 cm,0);
    \draw[thick,fill=black] (4 cm,2 cm) circle [radius=.3] ;
    \draw[thick] (0.1 cm, 3mm) -- +(3.6, 1.5 cm);
    \draw[thick] (8 cm, 3mm) -- +(-3.7, 1.5 cm);
  \end{tikzpicture}& \includegraphics[scale=0.25]{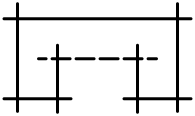}\\
      \hline  \raisebox{1ex}{$\operatorname{II}$} &
       \begin{tikzpicture}[scale=.4]
    \draw (-1,0) node[anchor=east]{$C_3$};
    \foreach \x in {0,1}
    \draw[xshift=\x cm,thick] (\x cm,0) circle (.3cm);
    \draw[thick] (0.3 cm,0) -- +(1.4 cm,0);
    \draw[thick,fill=black] (1 cm,1 cm) circle [radius=.3];
    \draw[thick] (0.1 cm, 3mm) -- +(0.6, 0.5 cm);
    \draw[thick] (2 cm, 3mm) -- +(-0.7, 0.5 cm);
  \end{tikzpicture} &\includegraphics[scale=0.25]{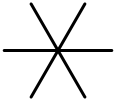}\\
      \hline  \raisebox{0.5ex}{$\operatorname{III}$} &
     \begin{tikzpicture}[scale=.4]
    \draw (-1,0) node[anchor=east]  {$C_2$};
    \draw[thick] (0 ,0) circle (.3 cm);
    \draw[thick,fill=black] (2 cm,0) circle (.3 cm);
    \draw[thick] (30: 3mm) -- +(1.5 cm, 0);
    \draw[thick] (-30: 3 mm) -- +(1.5 cm, 0);
  \end{tikzpicture} &\includegraphics[scale=0.25]{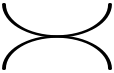}\\
      \hline  \raisebox{1.5ex}{$\operatorname{IV}$} &
      \begin{tikzpicture}[scale=.4]
    \draw (-1,0) node[anchor=east]  {Cusp};
    \foreach \x in {0}
    \draw[thick,xshift=\x cm, fill=black] (\x cm,0) circle [radius=.3] node [below] {};
  \end{tikzpicture} &\includegraphics[scale=0.25]{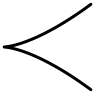}\\
      \hline  \raisebox{1.5ex}{$\operatorname{II}^*$} &
      \begin{tikzpicture}[scale=.4]
    \draw (-1,0) node[anchor=east]  {$E_8$};
    \foreach \x in {0,...,3}
    \draw[thick,xshift=\x cm] (\x cm,0) circle [radius=.3] node [below] {};
    \foreach \x in {4,5}
    \draw[thick,xshift=\x cm] (\x cm,0) circle [radius=.3];
     \foreach \x in {6}
    \draw[thick,xshift=\x cm,fill=black] (\x cm,0) circle [radius=.3];
    \foreach \y in {0,...,5}
    \draw[thick,xshift=\y cm] (\y cm,0) ++(.3 cm, 0) -- +(14 mm,0);
    \draw[thick] (4 cm,1.5 cm) circle (3 mm) ;
    \draw[thick] (4 cm, 3mm) -- +(0, 0.9 cm);
  \end{tikzpicture}&\includegraphics[scale=0.25]{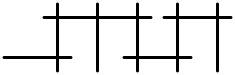}\\
  
      \hline  \raisebox{1.5ex}{$\operatorname{III}^*$} &
       \begin{tikzpicture}[scale=.4]
    \draw (-2.5,0) node[anchor=east]  {$E_7$};
    \foreach \x in {-1,...,5}
    \draw[thick,xshift=\x cm] (\x cm,0) circle (3 mm);
    \foreach \x in {5}
    \draw[thick,xshift=\x cm,fill=black] (\x cm,0) circle (3 mm);
    \foreach \y in {-1,...,4}
    \draw[thick,xshift=\y cm] (\y cm,0) ++(.3 cm, 0) -- +(14 mm,0);
    \draw[thick] (4 cm,1.5 cm) circle (3 mm);
    \draw[thick] (4 cm, 3mm) -- +(0, 0.9 cm);
  \end{tikzpicture} &\includegraphics[scale=0.25]{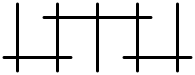}\\
  
      \hline  \raisebox{1.5ex}{$\operatorname{IV}^*$} &
       \begin{tikzpicture}[scale=.4]
    \draw (-1,0) node[anchor=east]  {$E_6$};
    \foreach \x in {0,...,4}
    \draw[thick,xshift=\x cm] (\x cm,0) circle (3 mm);
     \foreach \x in {4}
    \draw[thick,xshift=\x cm,fill=black] (\x cm,0) circle (3 mm);
    \foreach \y in {0,...,3}
    \draw[thick,xshift=\y cm] (\y cm,0) ++(.3 cm, 0) -- +(14 mm,0);
    \draw[thick] (4 cm, 1.5 cm) circle (3 mm);
    \draw[thick] (4 cm, 3mm) -- +(0, 0.9 cm);
  \end{tikzpicture} &\includegraphics[scale=0.25]{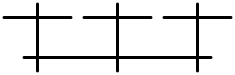}\\
      \hline 
    \end{tabular}
  \end{center}
\end{table}
\begin{Remark}\label{Remark:table:sing:fibers} The black dots in the second column of Table \ref{tab:table3} denote the components that intersect $S$. The number of irreducible components in an $\I_n$ fiber is $n$, whereas for an $\I_n^*$ fiber it is $n+5$.
\end{Remark}
We now define the Weierstrass fibrations parametrized by the interior of our moduli spaces:
\begin{Def}[see also \cite{AB3}*{Definition 4.1}]\label{Def:weighted:W:fib}
 Let $I:=(s,\vec{a},g,d)$ be a vector consisting of a rational number $0<s\le 1$,
 a vector $\vec{a} \in \mathbb{Q}^n$ with entries $0< a_i < 1$, and two natural numbers: $g$ and $d$.
 We say that $I$ is an \underline{admissible weight vector} if there is a minimal Weierstrass fibration
 $f:X \to C$ with section $S \subseteq X$ and $n$ fibers $F_1,...,F_n$ such that:
 \begin{itemize}
  \item $(X,sS+\sum a_iF_i)$ is a stable pair;
  \item The genus of $C$ is $g$, and
  \item The degree of the $j$-invariant $C \to \mathbb{P}^1$ is $d$.
 \end{itemize}
 We will call $(X,sS+\sum a_iF_i)$ a \underline{stable Weierstrass fibration with weight vector $I$}.
\end{Def}
\begin{Remark}A definition similar to Definition \ref{Def:weighted:W:fib} is given in \cite{AB3}*{Definition 4.1}. We keep $g$ and $d$ as part of the data because it is easier to argue why our moduli space is of finite type.
\end{Remark}
For every admissible weight $I$, we will construct in Section \ref{Section:construction:of:the:moduli:surface:pairs} a parameter space $\cE_I$ which on the interior (i.e. $\cE_I^\circ$) parametrizes stable Weierstrass fibrations with weight vector $I$.
\begin{Notation}
 We denote $\vec{a}F:=\sum a_i F_i$. When we write $I=(s, \vec{a},\beta)$, the entry $\beta$ represents the pair $(g,d)$. Given $I_1:=(s_1,\vec{a}_1,\beta)$ and $I_2:=(s_2,\vec{a}_2,\beta)$, we say that $I_1 \le I_2$ if $s_1 \le s_2$ and, if $a_i^{(j)}$ is the $j$-th entry of $\vec{a}_i$, then $a_1^{(j)} \le a_2^{(j)}$ for every $j$.
\end{Notation}
\begin{Lemma}\label{Lemma:W:fibration:remembers:the:fibration}Let $(X,sS+\vec{a}F)$ be a stable Weierstrass fibration. Then the morphism $X \to C$ in the definition of Weierstrass fibration is uniquely determined, unless $X$ is isomorphic to the product of two elliptic curves, and $sS+\vec{a}F$ has two irreducible components.  
\end{Lemma}
\begin{proof} Choose a fibration 
$f:X \to C$ in the definition of Weierstrass fibration. It is enough to prove that if $g:X \to C$ is another fibration, a fiber of $f$ is contracted by $g$. Because if this is the case, all the fibers will be contracted by $g$ since they are all numerically equivalent. But then the morphism $g$ factors through $f$, and both $f$ and $g$ have connected fibers. Therefore it suffices to show that one can identify a fiber of $f$ using only the surface pair $(X,sS+\vec{a}F)$.

If $\operatorname{Supp}(sS+\vec{a}F)$ has more than 2 irreducible components,
we can recognize a fiber from the combinatorics of the intersections.
If $\operatorname{Supp}(sS+\vec{a}F)$ has 2 irreducible components, from the definition of $\sL$ and from \cite{Miranda}*{Lemma II.5.6}, $S^2=-\deg(\sL) \le 0$. If the inequality is strict, then the irreducible component of $\operatorname{Supp}(sS+\vec{a}F)$ with self intersection 0 will be a fiber. If $\deg(\sL)=0$, then $X$ is isomorphic to a product from \cite{Miranda}*{Lemma III.1.4}. If the section has genus which is not 1, it is uniquely determined in $\operatorname{Supp}(sS+\vec{a}F)$, and we can identify a fiber. 
Otherwise, $X$ is a product of two elliptic curves.

Finally assume $\vec{a}F=0$, i.e. $\operatorname{Supp}(sS+\vec{a}F)$ has a single irreducible component. Then since $S^2 \le 0$, we need to have $K_X.S >0$ in order for $(X,sS)$ to be stable. But from \cite{Miranda}*{Proposition III.1.1}, $K_X \cong f^*(\sL\otimes \omega_C)$ where $\mathscr{L}$ is the fundamental line bundle. Therefore $\deg(\sL\otimes \omega_C)>0$, and a section of $H^0(m K_X)$ is supported on some fibers, for $m$ big enough. 
\end{proof}
We recall now the possible elliptic surfaces and fiber types of \cite{LaNave}, \cite{AB1} and \cite{AB3}.
Let then $\ell$ be an algebraically closed field and consider a twisted stable map $\mathcal{C} \to \overline{\mathcal{M}}_{1,1}$ over $\spec(\ell)$.
Let $(\mathcal{X},\mathcal{S}) \to \mathcal{C}$ be the
corresponding family of elliptic curves. Let $g:(X',S') \to C$ be the induced morphism between coarse moduli spaces. 
\begin{Def}[see \cite{AB3}*{Definition 3.3}]
 With the notations above, a \underline{twisted fiber} is a fiber of $g$, with its reduced structure. We call the twisted fibers which are supported on a non-reduced scheme theoretic fiber, the \underline{multiple twisted fibers}. \end{Def}
 These fibers are either DM stable genus 1 and 1-pointed curves, or a quotient of those. The ones which are not DM stable, give rise to scheme-theoretic fibers which are not reduced.

 Consider now the surface $X$ obtained from $X'$ performing the following two birational transformations.
The first one is a blow-up $\pi:Y \to X'$ of an ideal sheaf supported
 at some points $\{p_1,...,p_r\} \subseteq S'$, such that $g(p_1)$ are smooth points of $C$ (however, we allow $r=0$, i.e. $Y=X'$). We require that:
 \begin{enumerate}
  \item For every $i$, the exceptional $F_i:=\pi^{-1}(p_i)$ is irreducible, and $F_i$ is contained in the normal locus of $Y$;
  \item For every $i$, the proper transforms of the twisted fiber $g^{-1}(p_i)$ does not intersect $S$, the proper transform $S'$, and it intersects $F_i$ in a single point, and
  \item The only singular point of $Y$ along $F_i$ can be on the intersection point with the proper transforms of the twisted fiber of $g(p_i)$.
\end{enumerate}
 
 The second birational transformation is the contraction $Y \to X$ of the proper transforms $\{\pi_*^{-1}(g^{-1}(p_i))\}_{i=0}^m$ for some $0 \le m \le r$.
 Since these two birational transformations are performed along some fibers, the morphism $g:X' \to C$ induces a morphism $f:X \to C$.
\begin{Def}
 An \underline{elliptic fibration} is a pair $(X,f)$ as above.
\end{Def}
From \cite{AB1}, a minimal Weierstrass fibration $X \to C$ is an elliptic fibration.
Often we abuse notation, and we do not specify the morphism $f$. This should cause no confusion.

Therefore fibers of an elliptic fibrations have at most two irreducible components: 
\begin{Def}[see \cite{AB3}*{Definition 3.3}]Let $(X,f)$ be an elliptic fibration.
 An \underline{intermediate fiber} is a fiber $f^{-1}(p)$ which is not irreducible. The \underline{twisted component} of an intermediate fiber
 is the proper transform of a twisted fiber through the blow-up used to define $X$. We call the other component of an intermediate fiber an
 \underline{intermediate component}.
\end{Def}

\begin{center}\label{Picture:fibers}
\includegraphics[scale=0.60]{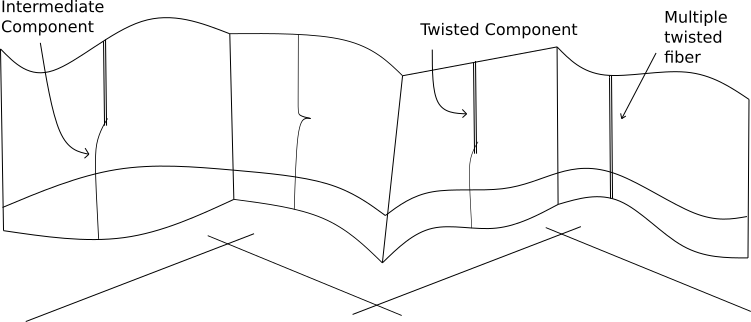}
\end{center}
\begin{Remark}
Our definition of intermediate fibers is a little bit more restrictive than the one in \cite{AB3}. However, from \cite{LaNave} and the results in Section \ref{section:singularities:threefold}, all the intermediate fibers we will find on the boundary of our moduli space satisfy our definition.
\end{Remark}
Since we deal just with slc surface pairs, we introduce the following
\begin{Def}[see also \cite{AB3}*{Definition 4.1}]\label{Def:slc:elliptic:surface}
 Let $\vec{a} \in \mathbb{Q}^n$ and $s \in \mathbb{Q}$ be such that
 $0< a_i< 1$ and let $0<s \le 1$. A \underline{slc (resp. lc, klt) elliptic surface} is a slc (resp. lc, klt) surface pair
 $(X,sS+\vec{a}F+E)$ such that there is an $f:X \to C$ which makes $(X,f)$
 an elliptic fibration. Moreover, we assume that each irreducible component of $\operatorname{Supp}(E)$ has coefficient 1 in $E$, $\supp(E)$
 is a union of some twisted fibers, all the multiple twisted fibers, and all the twisted components of the intermediate fibers. Finally, we assume that
 $\operatorname{Supp}(S)$ is the section, and $\operatorname{Supp}(\vec{a}F)$ is a union of intermediate components and irreducible fibers.
\end{Def}
Irreducible slc elliptic surfaces appear as irreducible components of surface pairs parametrized by $\cE_I$ (on the boundary). The components of $E$ come with marking 1 because the double locus will be supported on $E$. 

Now, even if we can show that a stable Weierstrass fibrations always degenerates to a slc elliptic surface (see Definition \ref{Def:tsm:limit}),
this degeneration might not (and in general will not) be stable. Namely, it is not a degeneration which is parametrized by our moduli space.
In fact, on the boundary of our moduli space, some other surfaces may appear:
\begin{Def}[see \cite{AB3}*{Definition 3.14} and \cite{LaNave}*{Definition 7.1.8}] \label{Def:pseudoelliptic}
 Let $\vec{a} \in \mathbb{Q}^n$ be such that
 $0< a_i< 1$. A slc (resp. lc) \underline{pseudoelliptic surface} is an irreducible slc (resp. lc) surface pair
 $(X,\vec{a}F+E)$ obtained from an irreducible slc elliptic surface $(Y,sS+\vec{a}F_Y+E_Y)$, contracting $S$. If $\pi:Y \to X$ is the contraction morphism, then $\pi_*(\vec{a}F_Y)=\vec{a}F$ and $\pi_*(E_Y)=E$. A \underline{pseudofiber} will be the proper transform of a fiber of $Y$.
\end{Def}
One can ask if a pseudoelliptic surface determines uniquely the elliptic surfaces it came from:
\begin{Oss}\label{Oss:section:uniquely:determined:on:psudo}
 Assume that $X$ is a pseudoelliptic surface. Once we know that some curves $F_1,...,F_n \subseteq X$, with $n$ big enough,
 are pseudofibers, then the surface pair $(Y,sS+\vec{a}F_Y+E_Y)$ in Definition \ref{Def:pseudoelliptic} is uniquely determined. Indeed, the surface $Y$ is obtained taking the stable model of a log-resolution of $(X,F_1+...+F_n)$ (see \cite{AB3}).
\end{Oss}
\subsection{The flip of La Nave}
In this subsection we recall a construction due to La Nave \cite{LaNave}.

Assume it is given a DVR $R$, with generic (resp. closed) point $\eta$ (resp. $p$).
Assume moreover that it is given a stable Weierstrass fibration $\phi:(\cX,\cS)_\eta \to \eta$ with weight vector $(1,0,\beta)$.
Since the moduli of stable pairs is proper, up to replacing $\spec(R)$ with a ramified cover, we can find a family of stable surface pairs
$(\cX^s,\cS^s) \to \spec(R)$ extending $\phi$ (the superscript $s$ stands for stable).
In \cite{LaNave} the author provides a description of $(\cX^s,\cS^s)_p$.
In particular, if $Y$ is an irreducible component of $\cX_p^s$ with double locus $E$,
it is proven that either $\cS^s \cap Y$ is a divisor, in which case $(Y,\cS^s_{|Y}+E)$ is a slc elliptic surface; or $\cS^s_{|Y}$ is not a divisor,
and $(Y,E)$ is a slc pseudoelliptic surface. Moreover, if $D \subseteq Y$ is an intermediate component of an intermediate fiber, then $D \subseteq E$.

\begin{center}
\includegraphics[scale=0.22]{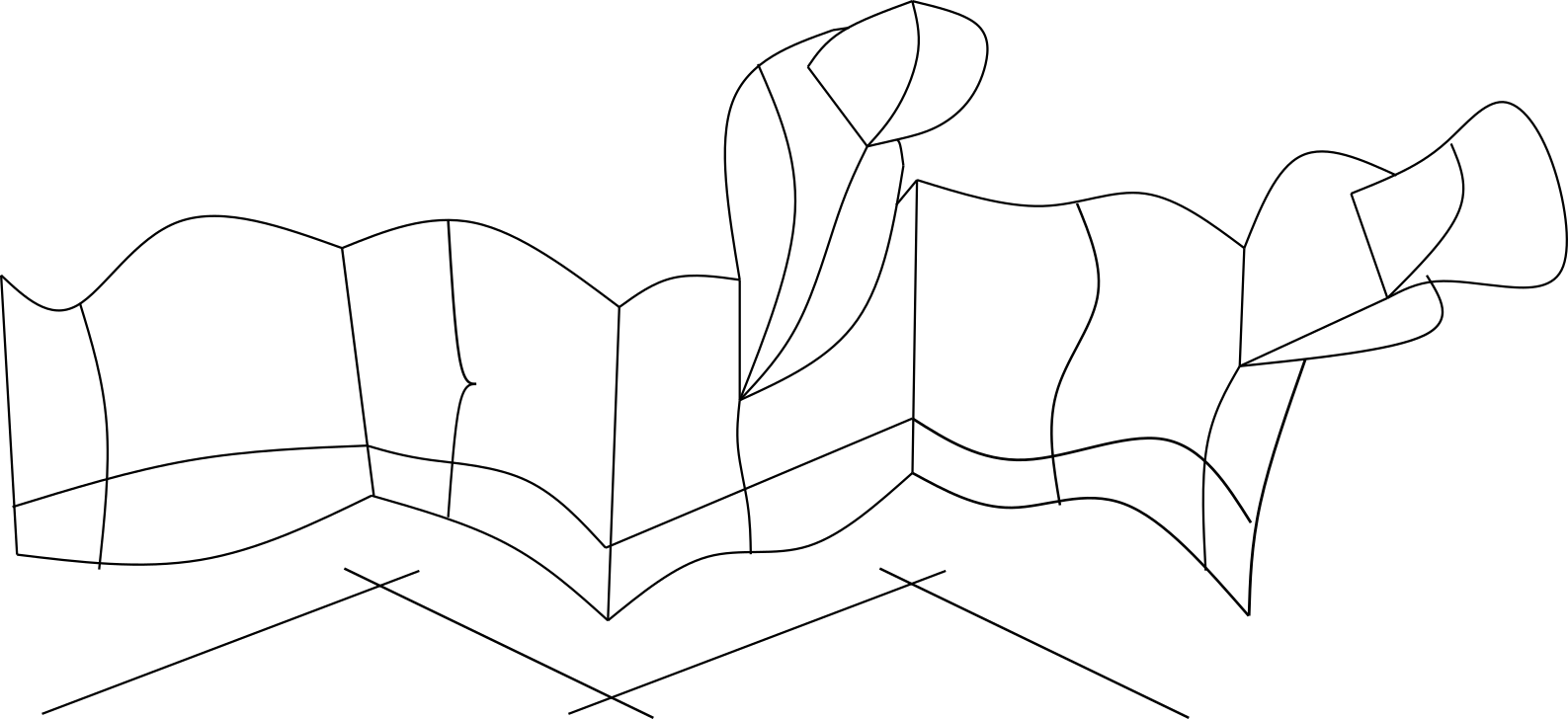}
\end{center}

The strategy used in \cite{LaNave} is the following. First, La Nave finds an auxiliary threefold pair $(\cX',\cS') \to \spec(R)$,
using \cite{AV1}. Every irreducible component $Y$ of the closed fiber $\cX_p'$ intersects $\cS'$,
comes with a map to a curve
$f_Y:Y \to C$, and is a slc elliptic fibration. Moreover, the scheme-theoretic fibers of $f_Y$, not on the double locus, are reduced.

\begin{center}
\includegraphics[scale=0.24]{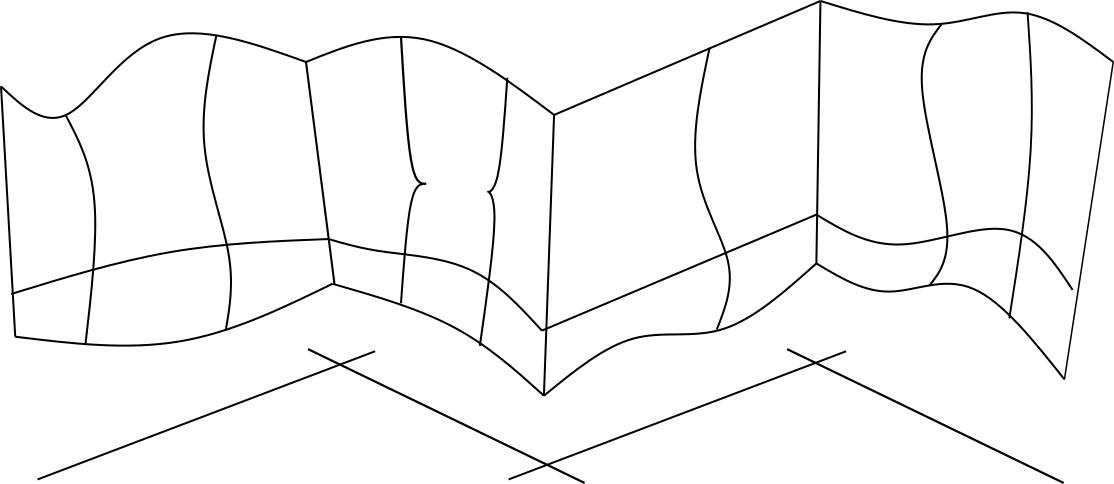}
\end{center}
Then La Nave finds the stable limit of $(\cX',\cS') \to \spec(R)$, running the MMP and through log-abundance. In particular, in \cite{LaNave} it is explicitly described a flip
that is needed to run the MMP.
\begin{Notation} We will refer to such a flip as the \underline{flip of La Nave}.\end{Notation} Since this birational
transformation plays an essential role both in what follows and in \cite{AB3}, we describe it below.
Consider $(\cX,\cS) \to \spec(R)$ a flat proper family of slc surface pairs. Assume that the generic fiber is a stable Weierstrass fibration,
and the closed fiber can be described as the closed fiber $\cX_p^s$ above. Let $C$ be an irreducible component of
$\cS_p$, assume that $(K_{\cX}+\cS).C <0$ and assume that the MMP contracts $C$ through an extremal contraction: let $f^-:\cX \to \cX^0$
be such a contraction. Let finally $X_1 \subseteq \cX_p$ be the irreducible
component of $\cX_p$ containing $C$.
Then
there is a new threefold pair $(\cX^+,\cS^+)$ with a contraction morphism $f^+:\cX^+ \to \cX^0$ such that the corresponding birational morphism
$(\cX,\cS) \dashrightarrow (\cX^+,\cS^+)$ is a flip. In this situation, La Nave shows that the flip
can be performed on a toric chart, and in \cite{LaNave}*{Theorem 7.1.2} such a flip is described explicitly.
It is shown that $X_1$ has a single fiber in the double locus of $\cX_p$, and let $X_2$ be the irreducible component of $\cX$ sharing a fiber with
$X_1$. Let $X_1^+$ (resp. $X_2^+$) be the proper transform of $X_1$ (resp. $X_2$) through $\cX \dashrightarrow \cX^+$.
It is proved that $X_1^+$ is a pseudoelliptic component, attached to $X_2^+$ along a twisted component of
an intermediate fiber, and the intermediate component
of such an intermediate fiber is the flipped curve. 

The picture below represents the behavior of the flip along $\cX_p$: 
\begin{center}\includegraphics[scale=0.32]{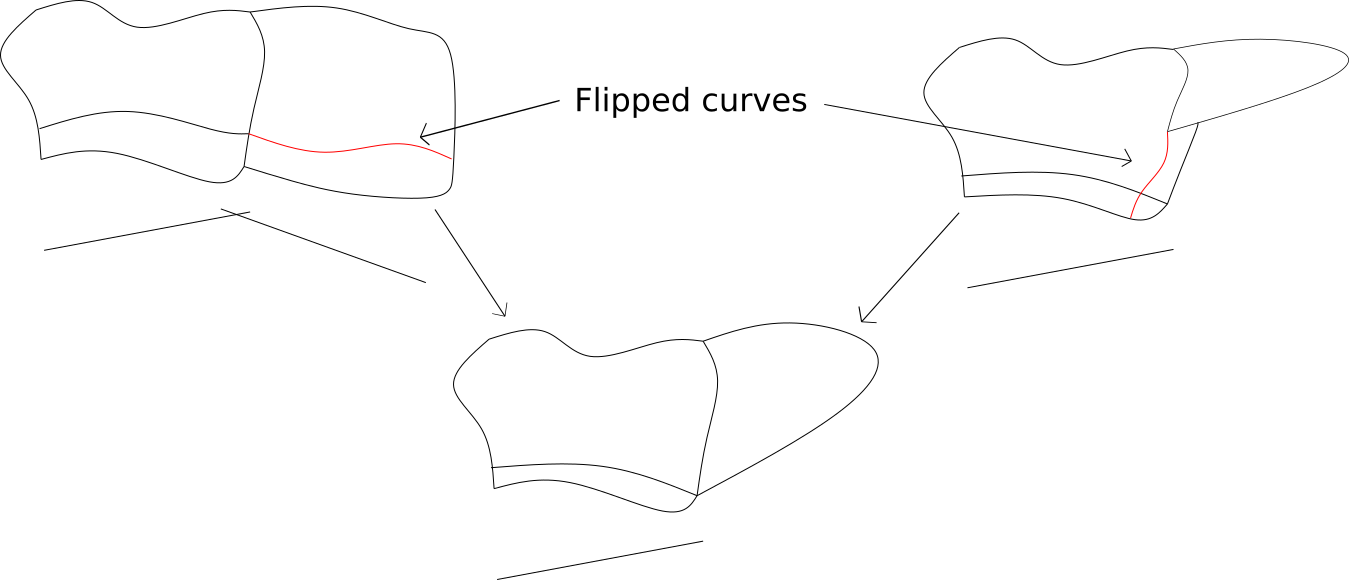}
\end{center}

\section{Stability conditions for Weierstrass fibrations and intermediate fibers}\label{Section:intermediate:fibers}
In Section \ref{Section:definitions:of:the:objects:we:parametrize}, we introduced two definitions, the one of stable Weierstrass fibration with weight data $I$,
and the one of intermediate fibers. We now study these two objects.
In the first subsection we recall the results of \cite{AB3}, to understand when a minimal Weierstrass fibration $(X,S)\to C$ is such that
$(X,sS+\vec{a}F)$ is log-canonical, for some marked fibers.
This means first understanding the singularities of $(X,sS+\vec{a}F)$, to ensure that the pair is lc. Then, the intersection pairings,
to ensure that $K_X+sS+\vec{a}F$ is ample.
In the second subsection we focus on intermediate fibers.
\subsection{Stability conditions for Weierstrass fibrations}
We start with an observation we will use several times throughout the paper:
\begin{Oss}\label{Oss:all:fib:irred:implies:pulback}
 Assume that $(X,sS+\vec{a}F+E)$ is an irreducible slc elliptic fibration, with all the fibers irreducible. Let $f:X \to C$ be the morphism to a curve.
 Then there is a $\mathbb{Q}$-divisor $D \subseteq C$ such that $K_X+\vec{a}F+E=f^*(D)$. 
\end{Oss}
In particular, if $f:X \to C$ is the morphism to a curve in the definition of slc elliptic fibration, and $M$ is an irreducible
multisection of $f$, we have
 $$(K_X+\vec{a}F+E).M=\deg(M \xrightarrow{f_{|M}}C)((K_X+\vec{a}F+E).S).$$
\begin{proof}[Proof of Observation \ref{Oss:all:fib:irred:implies:pulback}]
All the fibers are irreducible, so it is enough to show that $K_X$ is supported on some fiber components. This holds since the generic fiber has trivial canonical divisor.
\end{proof}
We now describe the conditions on the singularities and on $I$ that one has to impose on a minimal
Weierstrass fibration $X$, for the pair $(X,sS+\vec{a}F)$ to be stable.
We begin with the following lemma, the proof follows from Observation \ref{Oss:all:fib:irred:implies:pulback} and \cite{AB3}*{Lemma B.1} (see also \cite{AB1}*{Corollary 6.8}).
\begin{Lemma}\label{Lemma:if:all:fibers:irred:if:we:dont:contract:the:section:it:is:stable}
 Let $(X,sS+\vec{a}F+E) \to C$ be a slc irreducible elliptic fibration, with all the fibers irreducible (i.e. with no intermediate fibers).
 Then $K_X+sS+\vec{a}F+E$ is ample if and only if $(K_X+sS+\vec{a}F+E).S>0$.
\end{Lemma}
The main consequence of Lemma \ref{Lemma:if:all:fibers:irred:if:we:dont:contract:the:section:it:is:stable}, is Corollary \ref{Cor:stability:condition:for:minimal:W:fibration}.
The first point of the following corollary follows from \cite{AB3}*{Theorem 3.10}, whereas (2) from Lemma \ref{Lemma:if:all:fibers:irred:if:we:dont:contract:the:section:it:is:stable}.
\begin{Cor}\label{Cor:stability:condition:for:minimal:W:fibration}
A minimal Weierstrass fibration $f:X \to C$ is such that $(X,sS+\vec{a}F)$ is a stable pair if and only if it satisfies the following conditions:
\begin{enumerate}
 \item Each singular fiber of type $\operatorname{I}_n^*$ (resp. $\I \I$, $\I\I\I$, $\I\operatorname{V}$, $\I\I^*$, $\I\I\I^*$ and
 $\I\operatorname{V}^*$) is marked with weight $a \le \frac{1}{2}$ (resp. $a \le \frac{5}{6}$, $\frac{2}{3}$, $\frac{1}{2}$, $\frac{1}{6}$,
 $\frac{1}{4}$, $\frac{1}{3}$), and
 \item $(K_X+sS+\vec{a}F).S>0$.
\end{enumerate}
\end{Cor}

Finally, coupling \cite{AB3}*{Corollary 4.14} with \cite{AB3}*{Lemma B.1}, we get the following:
\begin{Lemma}\label{Lemma:conditions:when:to:contract:section}
 Let $(X,sS+\vec{a}F+E)$ be an irreducible slc elliptic fibration, with $f:X \to C$ its associated morphism.
 If $(K_X+sS+\vec{a}F+E).S<0$, then either the arithmetic genus of $C$ is 0, or it is 1. If it is 0, then the number of fibers marked with coefficient 1 is at most two and $\sum a_i \le 2$. If it is 1, then $\vec{a}F+E=0$.
\end{Lemma}
\subsection{Intermediate fibers}\label{Subsection:intermediate:fibers}
In this subsection we study intermediate fibers. We first understand the singularities of a twisted fiber, and then we focus on the intersection pairing on the intermediate fibers. These fibers can also be studied considering the stable model of a minimal log-resolution of the pair $(X,S+aF)$, consisting of a possibly non-minimal Weierstrass fibration, with a fiber $F$ with coefficient $a$. See \cite{AB3}*{Section 3} for such a point of view. 

Since the questions we will address are local over $C$, we give the following definition:
\begin{Def}\label{Def:germ}
Let $X \xrightarrow{f} C$ be an elliptic (resp. minimal Weierstrass) fibration. Given a point $x$ of $C$, let $R:=\cO_{C,x}$ and consider the morphism $\spec(R) \to C$. We call the pair
$(X\times_C \spec(R),X\times_C \spec(R) \to \spec(R))$, a \underline{germ of an elliptic} (resp. \underline{minimal Weierstrass}) \underline{fibration}. 
\end{Def}
 
From the definition of intermediate fibers, these are obtained from a twisted fiber performing a blow-up at
$p$, the intersection point between the section and the closed fiber. We begin then by focusing on the twisted fibers.
Consider $X \to \spec(R)$ the germ of an elliptic fibration, and assume that the closed fiber is twisted and \emph{singular} at $p$.
In \cite{AB1} the two authors, among other things, study the possible singularities of
$X$ at $p$. The following Lemma is implicit in \cite{AB1}:
\begin{Lemma}\label{lemma:classif:sing:tsm}
 With the notation of \cite{Kollarsing}*{3.19}, the possible singularities of $X$ at $p$ are the following:
$$\mathbb{A}^2/\frac{1}{2}(1,1); \text{ }\mathbb{A}^2/\frac{1}{3}(1,1);\text{ }\mathbb{A}^2/\frac{1}{3}(1,-1);\text{ }\mathbb{A}^2/\frac{1}{4}(1,1);
\text{ }\mathbb{A}^2/\frac{1}{4}(1,-1);\text{ }\mathbb{A}^2/\frac{1}{6}(1,1);\text{ }\mathbb{A}^2/\frac{1}{6}(1,-1)$$
\end{Lemma}
\begin{Remark}The previous Lemma can also be recovered using these three ingredients. First, that $X$ is the coarse space of a stack $\cX$. Second, that for every $p \in \cX(\spec(k))$, the group $\Aut_\cX(p)$ is cyclic of order either 1, 2, 3, 4 or 6. Third, \cite{Olsson}*{Theorem 11.3.1}.
\end{Remark}
\begin{Remark}\label{Rmk:remame:sing:tsm}
Recall that 
$\mathbb{A}^2/\frac{1}{n}(1,-1) \cong \spec(k[u,v,w]/(uv-w^n)),$
and $\mathbb{A}^2/\frac{1}{n}(1,-1)$ is an $A_{n-1}$-singularity. We will follow \cite{AB1} and call $A_{n-1}^*$ the singularity
$\mathbb{A}^2/\frac{1}{n}(1,1)$.
\end{Remark}
Consider now $X$ as above, assume that $X$ is normal (i.e. the generic fiber is not nodal) and let $Y \to X$ be a minimal resolution.
Using Kodaira's classification of the possible singular fibers of $Y$, and the minimal resolutions for each of these singularities, in \cite{AB1} the authors match the possible singularities at $p$ with the fiber of $Y$. We report the result in
Table \ref{tab:table2}. Finally one can treat also the case in which the generic fiber is nodal. In fact, using that the automorphism group of a nodal genus 1 one-pointed curve is $\mathbb{Z}/2\mathbb{Z}$, the singularity at $p$ is an $A_1$ singularity.
 \begin{table}[!ht]
  \begin{center}\caption{}
  \label{tab:table2}
    \begin{tabular}{|c|l|c|c|c|c|c|c|}
      \hline
       \text{Singular fiber of }Y &$\operatorname{I}_n^*$  & $\operatorname{II}$ & $\operatorname{III}$ & $\operatorname{IV}$ &$\operatorname{II}^*$ & $\operatorname{III}^*$ & $\operatorname{IV}^*$\\
      \hline
      \text{Singularity at }$p$&$A_1$ & $A_5^*$ & $A_3^*$ & $A_2^*$ & $A_5$ & $A_3$ & $A_2$ \\
      \hline
      
    \end{tabular}
  \end{center}
\end{table}

Now we focus on intermediate fibers. To obtain such a fiber, we proceed in two steps. First we resolve the singularity at $p$, to get $Z \to X$.
The exceptional divisor will be a chain
of $\mathbb{P}^1$. Then we can contract all the exceptionals that do not meet the proper transform of $S$. The type of intermediate fiber we get
depends on the resolution $Z$ we choose.
\begin{Def}\label{Def:minimal:intermediate}
 With the above notation, a \underline{minimal intermediate fiber} is an intermediate fiber obtained as above, from a minimal resolution $Z \to X$
 of $p$ (by hypothesis $X$ is singular at $p$). 
\end{Def}
We remark that our definition agrees with the one in \cite{AB3}*{Definition 3.3}.
In the case of minimal intermediate fibers, the contraction of the chain of $\mathbb{P}^1$ mentioned above is algebraic by \cite{AB1}. 
In loc. cit. the two authors also prove that,
if $Z$ is a germ of an elliptic fibration with a minimal intermediate fiber, there is a morphism $Z \to Y$ that contracts the twisted component.
\begin{Def}
 A germ of an elliptic fibration $Y \to \spec(R)$ has a \underline{minimal cusp} if there is $Z \to \spec(R)$ a germ of an elliptic
 fibration, with a minimal intermediate fiber, and a contraction morphism $Z \to Y$ that contracts the twisted component. 
\end{Def} From \cite{AB1}, if $Y \to \spec(R)$ has a minimal cusp, the closed fiber of $Y$ is a cusp.
Therefore, if $Y$ is normal and $T \to Y$ is a minimal resolution, the closed fiber of $T \to \spec(R)$ is one of the
following fiber types: $\operatorname{I}_n^*$, $\I \I$, $\I\I\I$, $\I\operatorname{V}$, $\I\I^*$, $\I\I\I^*$ or
 $\I\operatorname{V}^*$ (see \cite{Miranda}). The germ of a minimal Weierstrass fibration which has a cusp as closed fiber, has a minimal cusp (see \cite{Miranda}*{Proposition III.3.2}). On the other hand, the germ of a normal elliptic fibration which as closed fiber has either a minimal cusp or a DM stable 1-pointed curve, is the germ of a minimal Weierstrass fibration.
  
\begin{bf}Non-minimal intermediate fibers:\end{bf} We now focus on intermediate fibers which are not minimal.
We start by giving an example:
\theoremstyle{definition}
\begin{EG}\label{Example:smooth:int:fiber}
Consider $X \to \spec(R)$ the germ of an elliptic fibration, and assume that
the closed fiber is a DM stable genus 1 curve (observe in particular that it is a twisted fiber, and $X$ is smooth along $S$).
Then we can blow-up the intersection point of the section and the closed fiber. The resulting fiber will
be intermediate, but not a minimal intermediate fiber.\end{EG}

Let $A+E$ an intermediate fiber of $X \to \spec(R)$, where $E$ is the twisted component and let $A$ the intermediate one. Let $q:=E \cap A$.
By definition, $X$ is smooth along $A \smallsetminus E$, and there is a morphism $X \to Z$ that contracts $A$.

Consider $Z' \to Z$ the minimal resolution of a neighbourhood of $S \cap Z$, where $S$ is the section. Consider
$X' \to X$ a minimal resolution of $X$ around
$q$. By the minimality of $Z' \to Z$, there is a morphism $X' \to Z'$, which is a composition of
blow-ups of smooth points:
$$X':=X^{(r)} \to X^{(r-1)} \to ... \to X^{(1)}=:Z'$$
By the minimality of $X'$ and since $X' \cong X$ in a neighbourhood of the section, the morphism
$X^{(m)} \to X^{(m-1)}$ is the blow-up of the point of intersection of the closed fiber of $X^{(m-1)} \to \spec(R)$,
and the proper transform of the section. Therefore:
\begin{Lemma}\label{Lemma:dual:graph:intermediate}
Let $E_X$ (resp. $E_Z$) be the proper transform of the twisted component in $X'$ (resp. $Z'$), let $G_X$ (resp. $G_Z$)
exceptional locus of $X' \to X$ (resp. $Z' \to Z$), and let $A_X$ the proper transform of the
intermediate component in $X'$. Let $\Gamma_{X'/X}$ (resp. $\Gamma_{Z'/Z}$) be the dual graph of the closed fiber $E_X + G_X+A_X$
(resp. $E_Z+ G_Z$).
Then $\Gamma_{X'/X}$ is obtained
 from $\Gamma_{Z'/Z}$ adding a chain to the edge of $\Gamma_{Z'/Z}$ that corresponds to the component intersecting $S$.
\end{Lemma}
An example is illustrated below, where if $p:X' \to Z'$ is the map described above,
$F_Z':=p_*^{-1}(F_Z)$; and the black edge denotes the component that intersects the section.
\begin{center}$\Gamma_{Z'/Z}:$
\begin{tikzpicture}[scale=.4]
    \draw (-1,0) node[anchor=east]{ };
    \foreach \x in {-1,...,1}
    \draw[xshift=\x cm,thick] (\x cm,0) circle (.3cm);
    \foreach \x in {-1}
    \draw[xshift=\x cm,thick,] (\x cm,0) circle (.3cm) node [above] {$E_{Z}$};
    \foreach \x in {2}
    \draw[xshift=\x cm,thick, fill=black] (\x cm,0) circle (.3cm) node [above] {$F_{Z}$};
    \draw[dotted,thick] (-1.6 cm,0) -- +(1.4 cm,0);
    \foreach \y in {0.15,...,1.15}
    \draw[xshift=\y cm,thick, fill=black] (\y cm,0) -- +(1.4 cm,0);
  \end{tikzpicture} $\text{ }$ $\text{ }$ $\text{ }$ $\text{ }$ $\text{ }$ $\Gamma_{X'/X}:$
  \begin{tikzpicture}[scale=.4]
 \draw (-1,0) node[anchor=east]{ };
    \foreach \x in {-1,...,1}
    \draw[xshift=\x cm,thick] (\x cm,0) circle (.3cm);
    \foreach \x in {-1}
    \draw[xshift=\x cm,thick] (\x cm,0) circle (.3cm) node [above] {$E_X$};
    \foreach \x in {2}
    \draw[xshift=\x cm,thick] (\x cm,0) circle (.3cm) node [above] {$F_Z'$};
    \draw[dotted,thick] (-1.6 cm,0) -- +(1.4 cm,0);
    \foreach \y in {0.15,...,3.15}
    \draw[xshift=\y cm,thick, fill=black] (\y cm,0) -- +(1.4 cm,0);
    \foreach \y in {4.15,...,4.15}
    \draw[xshift=\y cm,dotted, thick, fill=black] (\y cm,0) -- +(1.4 cm,0);
    \foreach \x in {3,...,4}
    \draw[xshift=\x cm,thick] (\x cm,0) circle (.3cm);
      \foreach \x in {5}
    \draw[xshift=\x cm,thick, fill=black] (\x cm,0) circle (.3cm) node [above] {$A_X$};
  \end{tikzpicture}
    \end{center}
    Observe that if $Z$ is already smooth (as in Example \ref{Example:smooth:int:fiber}), then $F_Z=E_Z$.
\begin{Oss} From Lemma \ref{lemma:classif:sing:tsm}, the exceptional $G_Z$ is a chain of $\mathbb{P}^1$. We can understand the
 minimal intermediate fibers as those such that $G_Z \neq \emptyset$, but $\Gamma_{X'/X} = \Gamma_{Z'/Z}$. 
\end{Oss}
\begin{Lemma}\label{Lemma:can:obtain:every:graph:from:intermediate}
With the notation of Lemma \ref{Lemma:dual:graph:intermediate}, every graph obtained from
$\Gamma_{Z'/Z}$ adding a chain to $F_Z$, can be obtained as above from an intermediate fiber.
\end{Lemma}
 \begin{proof} The main ingredient are the results in \cite{Artinsing1} and
\cite{Artinsing2}.

We can work inductively adding one edge after the other.
Assume then that $f:X \to \spec(R)$ has an intermediate fiber $E+A$, and let $Z$ be the surface obtained contracting the intermediate component $A$.
Assume that, with the notation above, $\Gamma_{X'/X}$ has $m$ edges.
On $X$, we can perform a blow-up at
the intersection point $x$
of $A$ and $S$, to get $\phi:B_xX \to X$. The resulting surface $B_xX \to \spec(R)$ will be such that the closed fiber of $B_xX \to \spec(R)$ has three irreducible components:
$E':=\phi_*^{-1}(E)$, $A':=\phi^{-1}(A)$ and $F$ the exceptional. Since $Z$ has rational singularities, using \cite{Artinsing1}*{Proposition 1}
and \cite{Artinsing2}*{Theorem 2.3} to a minimal resolution of $B_xX$, we see that we can contract $A'$ on $B_xX$, to
get $\psi:B_xX \to Y$. Then
$Y \to \spec(R)$ has an intermediate fiber, and the dual graph of the exceptional divisor of a minimal resolution $Y' \to Y$ has $m+1$ edges. \end{proof}
Consider now two birational germs of elliptic surfaces $X$ and $Y$, assume they are obtained as in the proof of Lemma \ref{Lemma:can:obtain:every:graph:from:intermediate}. Namely, let $A$ (resp. $E$) be the intermediate (resp. twisted) component of the intermediate fiber of $X \to \spec(R)$. Then $Y$ is obtained from $X$ performing a blow-up $\phi$ at the intersection point of $A$ and the section, and $\psi$ is the contraction of $\phi^{-1}_*(A)$.
$$\xymatrix{ & B_xX \ar[ld]_{\phi} \ar[rd]^{\psi} & \\ X & &Y}$$

We can then compare the intersection form on $X$ and on $Y$: let $A_Y$ (resp. $E_Y$) be the intermediate (resp. twisted) component of
the closed fiber of $g:Y \to \spec(R)$. Let $q$ be the closed point of $\spec(R)$. First observe that if $f^{-1}(q)=A+mE$, then
$g^{-1}(q)=A_Y+mE_Y$ for the same $m$. Moreover, $A_Y(A_Y+mE_Y)=A(A+mE)=0$ since $A_Y+mE_Y$ and $A+mE$ are fibers.
Then to understand the intersection form on $Y$ it suffices to compute $A_Y^2$.

Since $\phi$ is a blow-up
we have $\phi^*(A)=A'+F$, moreover $F^2=-1$ and $A'.F=1$. Therefore $$(A')^2=\phi^*(A).A' - F.A'=A.\phi_*(A') -1=A^2-1$$
On the other hand, there is a constant $\alpha$ such that $\psi^*(A_Y)=\alpha A' + F$, since $F$ is the proper transform of
of $A_Y$ through $\psi$. Then we have
$$0=\psi_*(A').A_Y=A'.\psi^*(A_Y)=\alpha (A')^2+1 \text{ }\text{ }\text{ therefore }\text{ }\text{ }\alpha=\frac{1}{-(A')^2}$$
So we have $$A_Y^2=\psi_*(F).A_Y=F.\psi^*(A_Y)=\alpha -1=\frac{1}{-(A')^2}-1=\frac{A^2}{1-A^2}$$
Recall now that $A^2<0$ and $A_Y^2<0$ since these are exceptional curves. Therefore
$$A_Y^2=\frac{A^2}{1-A^2}>A^2$$
The main consequence of these computations is the following proposition:
\begin{Prop}\label{Prop:recognize:the:intermediate:from:int:pairing}
 Assume that are given $X_1 \to \spec(R)$ and $X_2 \to \spec(R)$, two germs of elliptic fibrations, with closed fibers that are intermediate. Let
 $A_i$ be the intermediate component of $X_i$, and let
 $p_i:X_i \to Z_i$ the contraction of $A_i$. If $Z_1 \cong Z_2$ and $A_1^2=A_2^2$, then $X_1 \cong X_2$.
\end{Prop}
Finally, we remark that one can compute the intersection form on the germ of an elliptic fibration with a \emph{minimal} intermediate fiber, see \cite{AB3}*{Table 2}. 
\section{Construction of the moduli space}\label{Section:construction:of:the:moduli:surface:pairs}
The goal of this section is to construct a parameter space for surface pairs that are degenerations of
stable Weierstrass fibrations with weight vector $I:=(s,\vec{a},\beta)$.
To explain our strategy more precisely, we need to introduce some notation.

Consider $\Kg$, the \emph{normalization} of an atlas of $\bigcup_{n \le m \le d}\mathcal{K}_{g,m}(\overline{\mathcal{M}}_{1,1},d)$.
Over $\Kg$, we have the universal curve $\sC' \to \Kg$, and the universal morphism $\sC'\to \overline{\mathcal{M}}_{1,1}$ that induces a
family of elliptic curves
$(\sX',\sS') \to \sC' \to \Kg$.
Let $\sX$ (resp. $\sS$) be the coarse moduli space of $\sX'$ (resp. $\sS'$).
 
This section is divided into three subsections. In the first one we study the singularities of $\sX$ along $\sS$.
In the second subsection
we construct $(\sY,s\sS+\vec{a}\sF) \to \Kg^\circ$ (see Notation \ref{Notation:def:SI}),
a family of lc stable elliptic surfaces, with weight data $I$.
This family will be obtained replacing the multiple twisted fibers of $\sX \to \Kg$ with minimal cusps, restricting the resulting family to the locus $\Kg^\circ \hookrightarrow \Kg$ parameterizing normal surfaces, and marking $n$ fibers with the entries of $\vec{a}$. This family is such that every stable Weierstrass fibration with weight vector $(s,(a_1,...,a_n),\beta)$ will appear as a fiber of $(\sY,s\sS+\vec{a}\sF) \to \Kg^\circ$, and every member of this family will be a stable Weierstrass fibration with weight vector $(s,(a_1,...,a_n),\beta)$. Observe that the number of minimal cusps in such a Weierstrass fibration is less than $d$ (see \cite{Miranda}): this is the reason for taking the union $\bigcup_{n \le m \le d}$. 

Once $(\sY,s\sS+\vec{a}\sF) \to \Kg^\circ$ is constructed, it induces a morphism
$\Psi:\Kg^\circ \to \mathcal{M}$, where $\mathcal{M}$ is an appropriate moduli space of stable surface
pairs.
In the last subsection we introduce the morphism $\Psi$ and we define $\mathcal{W}_{I}$ (resp. $\cE_I^{sn}$) to be the normalization (resp. seminormalization) of the closure of the image of $\Psi$. 
\subsection{Singularities of $\sX$ along $\sS$}\label{subsection:singularities:and:simult:blowup}

Let $(\sC,\{\sigma_i\}_{i})$ the coarse moduli space of $(\sC',\{\Sigma_i\}_{i})$ and similarly $\sS$ the one of $\sS'$.
Since we are over an algebraically closed field of characteristic 0, taking the coarse moduli space of a DM stack commutes with base
change (\cite{AV2}*{Lemma 2.3.3}). Therefore $h:\sC \to \Kg$ is a family of nodal genus $g$ curves, with distinct points,
whereas $(\sX,\sS) \to \Kg$ is a family of elliptic fibrations with twisted fibers:
$$(\sX, \sS) \xrightarrow{g} (\sC,\{\sigma_i\}_i) \xrightarrow{h} \Kg.$$
\begin{Notation} Since $\sigma_i$ are sections of $h$ and $\sS$ is a section of $g$, we can take the composition to get a section of $h \circ g$. We will denote with $s_i$
such a section.\end{Notation}

Observe that the singularities of $\sX_q$ along $p:=s_i(q)$ are classified in Lemma \ref{lemma:classif:sing:tsm}. We want to replace
some twisted fibers with minimal cusps, so first
we explicitly produce a blow-up of $p$:

\begin{Lemma}\label{Lemma:sing:duval:fibers}
With the notations above, assume that the singularities along $p$ are either
$A_2$, $A_3$, $A_4$ or $A_6$.
Then if $\pi:B_p(\sX_q) \to \sX_q$ is the blow-up of $p$, the proper transform of the section intersects the exceptional divisor in a smooth point of
$B_p(\sX_q)$. Moreover, $\pi^*(K_{\sX_q})=K_{B_p(\sX_q)}$; a minimal resolution $Z \to \sX_q$ of $p$ factors through $ B_p(\sX_q) \to \sX_q$, and the singularities of $B_p(\sX_q)$ along the exceptional of $B_p(\sX_q) \to \sX_q$ are Du Val. \end{Lemma}
\begin{proof}
The statement is local around $s_i(q)$, and since $(\Sigma_i)_q$ is contained in the smooth locus of $(\sC')_q$, we can assume that
$(\sC')_q$ (and thus $(\sC)_q$) is smooth.

Consider then $(\cC,\Sigma_i) \to \overline{\mathcal{M}}_{1,1}$, a smooth twisted stable curve over $\spec(k)$, and let $(\cX,\mathcal{S})\to \cC$ be the corresponding
family of stable curves.
Now, $\cX\to \cC$ has a section, so we have morphisms $\Sigma_i \to \cC \to \cX$ which are closed embeddings. Moreover, the section $\cC \to \cX$
is contained in the smooth locus of $\cX$, therefore the morphism $\cC \to \cX$ étale locally looks like
$ \spec(k[x])\xrightarrow{\alpha} \spec(k[x,y])$,
where the $\alpha$ is a closed embedding corresponding to the ideal $(y)$.

Up to passing to an étale neighbourhood of $p$, we can replace $\cX$ with $[U/\operatorname{Aut}(p)]$ (\cite{Olsson}*{Theorem 11.3.1}).
But any automorphism
in $\operatorname{Aut}(p)$ sends the section to itself. Since $\Aut(p)$ is cyclic, if $m_p$ is the ideal sheaf of $p$ in $U$,
we can take a basis of $m_p/m_p^2$ given by eigenvalues, and we can choose
a generator of the ideal sheaf of the section in $\cO_{U,p}$ to be one of those eigenvalues.

Therefore, étale locally around $p$, $\mathcal{X}$ looks like
$[\spec(k[x,y])/\operatorname{Aut}(p)]$,
where a generator $g \in \operatorname{Aut}(p)$ sends $y \mapsto ay$ and $x \mapsto bx$ for some roots of unity $a,b$ (recall that $\operatorname{Aut}(p)$ is
a cyclic group). Since $\cX \to \cC$ is representable and has a section, $a$ is a generator of $\operatorname{Aut}(p)$. But $\cC \to \overline{\mathcal{M}}_{1,1}$ is representable, which means that any non-trivial automorphism of $\cC$ comes from a non-trivial automorphism of $\cX \to \cC$: also $b$ generates $\operatorname{Aut}(p)$. Then $ab=1$,
since the singularity is $A_n$ and not $A_n^*$. Therefore:
$$\operatorname{Coarse}([\spec(k[x,y])/\operatorname{Aut}(p)]) \cong \spec(k[u,v,w]/uv-w^{n})\text{ where } n:=|\operatorname{Aut}(p)|$$ and the
quotient map sends $u \mapsto x^n,v\mapsto y^n$ and $w \mapsto xy$.
Therefore $(y)$, the ideal sheaf of the section, maps to $(v,w)$.

Then it is enough to explicitly perform the blow-up of $p$, which can be performed étale locally, and check that the proper
transform of the section intersects the exceptional in a smooth point of the surface. This can also be checked étale locally.
These are the three charts of the blow-up:
$$\spec(k[u_1,v_1,w]/(u_1v_1-w^{n-2})) \text{ and the ideal }(v,w) \text{ becomes } (w),$$
$$\spec(k[u_1,v,w_1]/(u_1-w_1^nv^{n-2})) \text{ and the ideal }(v,w) \text{ becomes } (v),$$
$$\spec(k[u,v_1,w_1]/(v_1-w_1^{n}u^{n-2})) \text{ and the ideal }(v,w) \text{ becomes } (v_1u,w_1u)=(w_1u).$$
Thus the proper transform of the section intersects only the last chart, which is smooth.

One can check that $\pi^*(K_{\sX_q})=K_{B_p(\sX_q)}$ (see \cite{KM}*{Section 4.2}).
Moreover, by the classification of Du Val singularities (\cite{KM}*{Theorem 4.20}) we see that the only singularities along the exceptional are
$A_{n-3}$ singularities,
which are Du Val. Finally, recall that to obtain a minimal resolution of an $A_m$ singularity, we can keep blowing the singular point
(with its reduced structure). Since $B_p(\sX_q) \to \sX_q$ is a blow-up of a Du Val singular point with its reduced structure, a minimal resolution
$Z \to \sX_q$ factors through $B_p(\sX_q) \to \sX_q$.
\end{proof}
\begin{Oss}\label{Oss:contract:intermediate:duval}
 With the notation of Lemma \ref{Lemma:sing:duval:fibers}, one can produce a contraction morphism $B_p(\sX_q) \to Y$ that contracts the fiber components that do not meet the section.
 Moreover, since the proper transform of the section intersects the exceptional divisor of $B_p(\sX_q)$ in the smooth locus,
 we can understand the resulting surface $Y$ as follows. It is the surface obtained from $\sX_q$ replacing the twisted fiber through $p$ with a
 minimal cusp.
\end{Oss}
The following lemma can be proved in a similar way, see \cite{Kollarresol}*{Section 2.4, page 86}.
\begin{Lemma}\label{Lemma:sing:nonduval:fibers}
 Let $q \in \Kg(\spec(k))$ and $p:=s_i(q)$, such that $(\sX)_q$ has a singularity at $p$ which is either
 $A_3^*$, $A_4^*$ or $A_6^*$. 
 Then if we take the blow up $Y \to (\sX')_q$ of $p$, the morphism $\operatorname{Coarse}(Y) \to \sX_q$ is a minimal resolution of $\sX_q$ around $p$.
\end{Lemma}
Observe that we are performing the blow up on the stack $\sX'_q$, not on its coarse space.

 \begin{Oss}\label{Oss:contract:intermediate:nonduval}
 With the notation of Lemma \ref{Lemma:sing:nonduval:fibers}, let $\widetilde{Y}$ be the coarse space of $Y$. Then
 the exceptional divisor of $\widetilde{Y} \to \sX_q$ has a single irreducible component (see \cite{Kollarresol}). We can understand $\widetilde{Y}$
 as obtained from $\sX_q$ replacing the twisted fiber through $p$ with a \emph{minimal} intermediate fiber. In particular, there is a morphism
 that contracts the twisted component of this minimal intermediate fiber, producing a minimal cusp.
\end{Oss}
 Finally, the following lemma describes how these singularities behave in our family of surfaces.
 \begin{Lemma}\label{Lemma:open:and:closed:sing}
 For every $q \in \Kg(\spec(k))$ there is a neighbourhood $V$ of $q$ satisfying the following condition.
 For every $t \in V(\spec(k))$, there is an étale neighbourhood of $s_i(t) \in (\sX)_t$ which is isomorphic to an
 étale neighbourhood of $s_i(q) \in (\sX)_q$.
In other words, the type of singularities along $s_i$ are locally constant.
 \end{Lemma}
  \begin{proof}
  Assume is given a smooth surface $X$ and a fixed point $x \in X$ for the action of a cyclic group $H$. Then the
  type of singularity of the image of $x$ through $X \to X/H$ depends only on the action of $H$ on the tangent space $T_xX$ of $X$ at $x$. 
  
  Up to replacing $\sX'$ with an étale neighbourhood of $p:=s_i(q)$, we can assume that $\sX' \cong [W/G]$ where $G=\operatorname{Aut}(p)$.
  The closed embedding $\Sigma_i \to \sX'$ corresponds to a closed subset $Z\subseteq W$. Since $\Sigma_i$ is a gerbe banded by $G$,
  every $z \in Z$ is $G$-invariant. Moreover, since $\Sigma_i$ is contained in the smooth locus of $\sX'$, $(T_{\sX'/\Kg})_{|\Sigma_i}$
  is a vector bundle of rank 2. Then $(T_{W/\Kg})_{|Z}$ is a vector bundle of rank 2 with an action of $G$. Up to shrinking $Z$, we can assume
  $(T_{W/\Kg})_{|Z} \cong Z \times \mathbb{A}^2_{\rho}$ where $\rho:G\times Z \to \operatorname{GL}_2\times Z$ is a homomorphism
  of group objects over $Z$. We can further assume that $Z$ is connected, and since $G$ is finite, $\rho$ is the pull back through $Z \to \spec(k)$ of a homomorphism $G \to \operatorname{GL}_2$. In other words, the action of $G$ on $(T_{W/\Kg})_{|Z}$ is locally constant.
 \end{proof}
 \begin{Cor}\label{Cor:sing:open:and:closed}
   For every $n$, the set $\{q \in \Kg: s_i(q)$ is a singularity of type $A_n^*\}$ is open and closed.
 \end{Cor}
 \begin{Cor}\label{Corollary:same:dimension:completion}
 For every $q \in \Kg(\spec(k))$ there is a neighbourhood $V$ of $q$ such that, for every $r$ and every $t \in V$,
 $\dim_k(\cO_{\sX_q}/m^r_{s_i(q)})=\dim_k(\cO_{\sX_t}/m^r_{s_i(t)})$
 \end{Cor}
 \begin{proof}
  The henselization of $\cO_{\sX_q,s_i(q)}$ is isomorphic to the one of $\cO_{\sX_t,s_i(t)}$ from Lemma \ref{Lemma:open:and:closed:sing}.
 \end{proof}
The following proposition allows us to perform the blow-ups of Lemma \ref{Lemma:sing:duval:fibers} and \ref{Lemma:sing:nonduval:fibers} simultaneously.
\begin{Prop}\label{Proposition:blowup:commutes}
  Let $f:X \to B$ be a proper morphism of schemes, and let $s:B \to X$ be a section of $f$.
 Assume that $B$ is reduced, and for every $b_1,b_2 \in B$ and every integer $r$,
 $$\dim_k(\cO_{X_{b_1}}/m^r_{s(b_1)})=\dim_k(\cO_{X_{b_2}}/m^r_{s(b_2)}).$$
 
 Then for every $p \in B$, if $B_{S}(X)$ is the blow-up of $X$ along $s(B)$, and
 $B_{S_p}(X_p)$ is the blow-up of $X_p$ along $s(p)$, the following diagram is fibered:
 $$\xymatrix{B_{S_p}(X_p) \ar[r] \ar[d] & B_{S}(X)\ar[d] \\ \spec(k(p)) \ar[r] & B}$$
\end{Prop}
\begin{proof}
 The statement is local so we can assume that $X=\spec(A),B=\spec(R)$. Let $m_p$ be the ideal of $R$ corresponding to $p$; $n_p:=m_pA$ the ideal of
 $X_p$; $I_p$ the ideal of $s(p) \in X_p$ and let $I$ be the ideal of $S:=s(B)$. Now,
 $B_{S_p}(X_p)=\operatorname{Proj}(\bigoplus _{r \ge 0} I_p^r)$ whereas $B_S(X) \times_B \spec(k(p)) \cong B_S(X) \times_X X_p\cong \operatorname{Proj}(\bigoplus _{
 r\ge 0} I^r
 \otimes_A A/n_p)$.
 Thus to prove the thesis it suffices to show that
 $$I^r\otimes_A A/n_p \cong I_p^r \text{ } \text{ for every } p \text{ and }r.$$

 Consider the following exact sequence:
 $$0 \to I^r \to A \to A/I^r \to 0. \text{ } \text{ }\text{ }\text{ }\text{ }\text{ } (*)$$
 Tensoring it with $A/n_p$ we get
 $ I^r \otimes_A A/n_p  \to A/n_p \to A/(I^r+n_p) \to 0$.
 Now, $A/(I^r+n_p) \cong (A/n_p)/(I_p)^r$. Therefore to prove the desired result it suffices to show that $I^r \otimes_A A/n_p  \to A/n_p$ is injective.
 But as $R$-modules, $I^r \otimes_A A/n_p \cong I^r \otimes_A(A \otimes _R R/m_p) \cong I^r \otimes_R R/m_p$ and $A/n_p \cong A \otimes_R R/m_p$.
 Moreover the map $I^r \otimes_A A/n_p  \to A/n_p$ is induced tensoring with $R/m_p$ the sequence $(*)$. Therefore if we can show that $A/I^r$ is a
 flat $R$-module, we have the thesis.
 
 But $A/I^r$ is \emph{finite} over $R$, i.e. the corresponding sheaf of $R$-modules is coherent. Therefore since $B$ is reduced,
 to show flatness it is enough
 to show that the map $B \to \mathbb{N}$, $p \mapsto \dim_k(A/(I^r+n_p))$ is locally constant, which holds by hypothesis.
\end{proof}
\subsection{Construction of $(\sY,s\sS+\vec{a}\sF) $}\label{subsection:construction:family:of:surfaces} 
 This subsection is divided into two parts. First, we replace all the multiple twisted fibers of $\sX \to \sC$, through two
 blow-ups, to get $B \to \sC \to \Kg$. Then we contract some fiber components. This procedure replaces all the multiple twisted fibers with minimal cusps.
  
 \begin{Step_1}
  We perform the blow-ups of Lemma \ref{Lemma:sing:nonduval:fibers} simultaneously.
 \end{Step_1}

 First we find the closed subset we have to blow-up.
 For each $1 \le i \le d$, let $U_i \hookrightarrow \Kg$ be the closed subset such that $(\sX)_t$ has a singularity of type $A_m^*$ on $s_i(t)$,
 for some $m$ (if there are no $A_m^*$ singularities, $U_i=\emptyset$).
 From Corollary \ref{Cor:sing:open:and:closed},
 $U_i \to \Kg$ is a closed embedding, so also $(\Sigma_i)_{|U_i} \to \Sigma_i$ is a closed embedding. Since $\Sigma_i \to \sC'$ and
 $\sC' \cong \sS' \to \sX'$ are closed embeddings, also the composition $(\Sigma_i)_{|U_i} \to \sX'$ is a closed embedding. 
 Let $\mathscr{Z}_i$ be the closed substack that corresponds to $(\Sigma_i)_{|U_i} \to \sX'$,
 and let $\mathscr{Z}:=\bigcup _{i=1}^n \mathscr{Z}_i$. We can understand $\mathscr{Z}$
 as a substack of $\sX'$ whose coarse space is the set of points $p\in \sX$ such that $(\sX)_{(h\circ g)(p)}$ has an $A_m^*$ singularity at $p$.
  \begin{Notation}
  We need to give a name to the fibers of $g:\sX' \to \sC'$ that contain $\sZ$: let $\sE':=g^{-1}(g(\sZ))$. Notice that $\sE'$ is a Cartier divisor.
 \end{Notation}
Let $\sB:=B_{\sZ}(\sX')$ be the blow-up of $\sX'$ along $\sZ$. For every $t \in \Kg(\spec(k))$, étale locally,
$\sZ_t$ are an union of closed points in the smooth locus of $\sX'_t$, therefore from Proposition \ref{Proposition:blowup:commutes},
for every $t \in \Kg(\spec(k))$ we have that
$$(\sB)_{t} \cong B_{(\sZ)_{t}}(\sX'_t).$$
\begin{Notation} Let $\sA^*$ be the exceptional divisor of $\sB \to \sX'$, and let $\sE^*$ (resp. $\sS_B$) be the proper transform of $\sE'$ (resp. $\sS'$).
 Let $\widetilde{\sB}$ (resp. $\widetilde{\sS_{\sB}}$, $\widetilde{\sA^*}$ and $\widetilde{\sE^*}$)
be the coarse moduli space of $\sB$
(resp. $\sS_{\sB}$, $\sA^*$ and $\sE^*$). 
\end{Notation}
\begin{Step_1}
 We perform the blow-ups of Lemma \ref{Lemma:sing:duval:fibers} simultaneously.
\end{Step_1}

As before, for each $1 \le i \le d$, let $V_i \hookrightarrow \Kg$ be the closed subset such that $(\widetilde{\sB})_t$ has a singularity of type $A_m$ on $s_i(t)$,
for some $m$. Then $s_i(Z_i)$ is a closed subscheme, and let $Z:=\bigcup_{i=1}^ns_i(Z_i)$. Then let $B:=B_{Z}(\widetilde{\sB})$
be the blow-up of $Z$ in $\widetilde{\sB}$.
From Corollary
\ref{Corollary:same:dimension:completion} and Proposition \ref{Proposition:blowup:commutes}, we have:
$$(B_{Z}(\widetilde{\sB}))_{t} \cong B_{(Z)_{t}}(\widetilde{\sB}_t).$$
Observe that $B$ comes with a map $B \to \sX$, thus we still have a morphism $g_B:B \to \sC$. Moreover, on
$\sC$ there is the divisor $D$ given by $\bigcup_i \sigma_i(V_i)$.
\begin{Notation}
On $B$, we have the following divisors: $\sS_{B}$ (resp. $\sA^*_B$, $\sE_B^*$), the strict transform of $\widetilde{\sS_{\sB}}$ (resp.
$\widetilde{\sA^*_B}$, $\widetilde{\sE_B^*}$),
and $\sF_B:=g_B^{-1}(D)$. 
\end{Notation} 
We can understand $\sF_B$(resp. $\sA^*_B$) as a family of  fibers (resp. intermediate components of minimal intermediate fibers),
where the associated twisted fiber
has an $A_n$ (resp. $A_n^*$) singularity.
Thus we have constructed a family of surface pairs
$$(B,s\sS_B+\sF_B+\sA^*_B+\sE_B^*) \to \Kg.$$

\begin{bf}Contraction morphism $B \to \sY'$.\end{bf} We need to contract the extra components produced by the blow-up. We begin with an observation:
\begin{Oss}\label{Oss:before:the:blow:down:Qcartier}
 The divisors $K_{B / \Kg}$, $\sS_B$, $\sA_B^*$, $\sE^*_B$ and $\sF_B$ are $\mathbb{Q}$-Cartier.
\end{Oss}
\begin{proof}We prove just the case of $K_{B / \Kg}$, the other cases are similar.
The divisor $K_{\sB / \Kg}$ is Cartier since $\sB \to \Kg$ is a family of Gorenstein surfaces.
Thus $K_{\widetilde{B}/\Kg}$ is $\mathbb{Q}$-Cartier, since $\sB$ is a Deligne-Mumford stack of finite type and $\widetilde{\sB}$ is its coarse space.
Finally from \cite{AB1}*{Theorem 6.1}, if
$p:B \to \widetilde{\sB}$ is the blow-up of $Z$, $p^*(K_{\widetilde{\sB}/\Kg}) \cong K_{B/\Kg}$.\end{proof}

Now the strategy is the following. For each point $\spec(k) \to \Kg$, the surface $B_p$ has a contraction morphism
$B_p \to X$ that replaces the non-irreducible fibers with minimal cusps
(see Observation \ref{Oss:contract:intermediate:duval} and \ref{Oss:contract:intermediate:nonduval}).
Our goal is to perform these contractions simultaneously. We need to
find a line bundle on
$B$ which is base point free, and such that the morphism
$B \to \sY'$ it induces contracts the fiber components that do not intersect the section. 

From Lemma \ref{Lemma:open:and:closed:sing} and Table \ref{tab:table2},
there is a divisor $\sE^{\I\I}_B$ (resp. $\sE^{\I\I\I}_B$ and $\sE^{\I\operatorname{V}}_B$)
supported on $\sE^*_B$ such that, for every $p \in \Kg$,
the minimal resolution of the fiber through $(\sE^{\I\I}_B)_p$ (resp. $(\sE^{\I\I\I}_B)_p$ and $(\sE^{\I\operatorname{V}}_B)_p$) is a
$\I\I$ (resp. $\I \I \I$, $\I \operatorname{V}$) fiber. One can define similarly $\sA^{\I\I}$, $\sA^{\I\I\I}$ and $\sA^{\I \operatorname{V}}$, supported on $\sA^*_B$.
Consider then
$\sD:=\sS+\frac{5}{6} \sA^{\I\I}+ \frac{3}{4} \sA^{\I\I\I}+\frac{2}{3}\sA^{\I \operatorname{V}}+\sE^{\I\I}_B+\sE^{\I\I\I}_B+\sE^{\I\operatorname{V}}_B$.

\begin{Prop}\label{Prop:contracting:E:does:not:contribute:to:cohom}
Let $s \in \Kg(\spec(k))$, let $X:=B_s$, let $D:=\sD_s$ and let $f:B_s\to C:=\sC_s$ be the morphism induced by $B \to \sC$.
Then $(X,D)$ is slc, $K_X+D$ is $f$-nef, and for $m$ divisible enough, we have $R^if_*(\cO_{X}(m( K_X+D)))=0$ for $i>0$. Moreover,
the stable model of $(X,D)$ over $C$ is obtained from $(X,D)$ contracting the fiber components that do not intersect the section.
\end{Prop}
See also \cite{AB1}*{Lemma 4.5, 4.6 and 4.7} for similar results.
\begin{proof} From \cite{AB2}*{Proposition 4.3} the pair $(\sX_s,\sS_s)$ is slc. But $X$ is obtained performing some blow-ups of $\sX_s$. In particular, where $X$ and $\sX_s$ are isomorphic, $(X,S)$ is slc. Namely, $X$ has slc singularities away from $D$ and the exceptional locus of of $X \to \sX_s$ along the fibers $(\sF_B)_s$.
Since Du Val singularities are lc, and the singularities on the exceptional locus of $X\to \sX_s$ along the fibers $(\sF_B)_s$
are Du Val from Lemma \ref{Lemma:sing:duval:fibers},
the pair $(X,D)$ is lc away from $D$. Using \cite{AB1} and since the marked fibers are of the types
of \cite{AB1}*{Lemma 4.5, 4.6 and 4.7}, the pair $(X,D)$ is lc also along $D$, so it is slc. We now check that $K_X+D$ is $f$-nef.

For any point $q \in X\smallsetminus D$, from Observation \ref{Oss:all:fib:irred:implies:pulback} and
Lemma \ref{Lemma:sing:duval:fibers} there is a neighbourhood $U$ of $f(q)$ such that
$(K_X)_{|f^{-1}(U)} \cong \cO_{f^{-1}(U)}$. So $K_X+D$ is $f$-nef along $f^{-1}(U)$. From
\cite{AB1}*{Lemma 4.5, 4.6 and 4.7}, the divisor $K_X+D$ is $f$-nef also along $D$.

Let then
$(X',D':=p_*(D))$ be the stable model of $(X,D)$ over $C$, with contraction morphism $p:X \to X'$.
Since $K_X+D$ is $f$-nef, the morphism $p$ is given by log-abundance. We need to understand which fiber components
it contracts, i.e. for which fiber components $F$, we have $(K_X+D).F=0$.
As in the previous paragraph, for any point $q \in X\smallsetminus D$ there is a neighbourhood $U$ of $f(q)$ such that
$(K_X)_{|f^{-1}(U)} \cong \cO_{f^{-1}(U)}$. So any fiber component not contained in $D$ that does
not intersect $S$ gets contracted by $p$. Moreover, from
\cite{AB1}*{Lemma 4.5, 4.6 and 4.7}, $p$ contracts the twisted components of the intermediate fibers of $D$: all the new fibers of $X'$ are minimal cusps.

Since $K_X+D$ is nef, we have that
$K_X+D=p^*(K_{X'}+D')$. From its definition, $X'$ comes with a morphism $g:X' \to C$ such that
 $K_{X'}+D'$ is $g$-ample.
 Let then $m$ be divisible enough such that both $m(K_{X'}+D')$ and $m(K_X+D)$
 are Cartier,
 and $R^ig_*(m(K_{X'}+D'))=0$ for $i>0$. Let $L_X:=\cO_X(m(K_X+D))$ and
 $L_{X'}:=\cO_{X'}(m(K_{X'}+D'))$.
 Then $R^ip_*(L_X)=R^ip_*(p^*L_{X'})=L_{X'}R^ip_*(\cO_{X})$. If we can prove that $R^ip_*(\cO_{X})=0$ for $i>0$, the Leray spectral sequence will give the desired vanishing.
 Now, the positive dimensional log-canonical centers of $X$ are supported along the double locus, so $-K_X$ is log-big since every $p$-exceptional curve is not supported on the double locus. Moreover, from \cite{AB1}*{Theorem 6.1} and since the intersection pairing is negatively definite along the exceptional curves, $-K_X$ is $p$-nef along $D$. From Lemma \ref{Lemma:sing:duval:fibers}, $-K_X$ is trivial along $(\sF_B)_s$ so it is $p$-nef everywhere. Therefore from \cite{Fuj-vanishing}*{Theorem 1.10} we have $R^ip_*(\cO_{X})=0$ for $i>0$.
 \end{proof}
 Let $\Psi:B \to \sC$ be the morphism
induced by $\sX \to \sC$. Then
from Proposition \ref{Prop:contracting:E:does:not:contribute:to:cohom} and from cohomology and base change, for $\ell$ divisible enough,
$\operatorname{Proj}(\bigoplus _{m=1}^\infty \Psi_*\cO_{B}( m\ell(K_{B/\Kg}+ \sD) ))$
commutes with base change, giving a family of elliptic fibrations $g_{\sY'}^{\text{ }}:\sY' \to \sC$. Let $\pi:B \to \sY'$ be the corresponding morphism,
and let $\sS':=\pi_*(\sS_B)$.

For every $p \in \Kg(\spec(k))$, the fibers of $\sY'_p$ not contained in the double locus are either DM stable, or minimal cusps.
So if $\sY_p'$ is normal, it is a minimal Weierstrass fibration.
Moreover, for every $n \le m \le d$, the universal twisted curve
over $\Kg_{g,m}(\overline{\mathcal{M}}_{1,1},d)$ has $m \ge n$ marked stacky points. Then if we choose the first $n$ of these points, 
from the definition of $\Kg$, the family of curves $\sC \to \Kg$ has $n$ distinguished sections $\sigma_i$. We denote $\sF'_i:=g_{\sY'}^{-1}(\sigma_i(\sC))$ and $\vec{a}\sF'$ the $\mathbb{Q}$-divisor $\sum a_i \sF_i'$.
\begin{Oss}\label{Oss:description:that:the:blow:ups:replace:twisted:with:mcusps}
We can understand the family of surfaces $\sY'$ as obtained from $\sX$ replacing any multiple twisted fiber of $\sX_p$ with a minimal cusp, for every $p$.
\end{Oss}
We give now a definition that will use later:
\begin{Def}\label{Def:bounded:family:tsm:limits}We call $(\sY',s\sS'+\vec{a}\sF') \to \Kg$ the \underline{bounded family of twisted stable maps limits} with weight vector $I$. 
\end{Def}

A priori, there might be some $p \in \Kg(\spec(k))$ such that $(\sY',s\sS'+\vec{a}\sF')_p$ is not slc. However, from Proposition \ref{Prop:contracting:E:does:not:contribute:to:cohom}, the only points on which it fails to be slc are along the divisor $\supp(\sS'_p + \vec{a} \sF_p)$; and from Corollary
\ref{Cor:stability:condition:for:minimal:W:fibration} the type of singular fibers is locally constant. So from Lemma
\ref{Lemma:open:and:closed:sing} there is an open embedding $\Kg_{slc} \to \Kg$ such that
for every $p \in \Kg_{slc}(\spec(k))$, the surface pair $(\sY',s\sS'+\vec{a}\sF')_p$ is
slc. We will abuse notation, and still denote with $\sY'$ (resp. $\sS'$ and $\sF'$) the family $\sY' \times_\Kg \Kg_{slc}$ (resp. $\sS' \times_\Kg \Kg_{slc}$ and $\sF' \times_\Kg \Kg_{slc}$).
Notice that $\Kg_{slc}$ will depend on $I$, and it is not empty since
$I$ is admissible. Moreover, $\sY' \to \Kg_{slc}$ is a flat proper family of surfaces, so there is an open embedding $\Kg^\circ \hookrightarrow \Kg_{slc}$ such that all the fibers of
$\sY'\times_{\Kg_{slc}}\Kg^\circ
\to \Kg^\circ$ are normal \cite{EGAIV}*{Theorem 12.2.4}.
\begin{Notation}\label{Def:sY}
 We define $(\sY,s\sS+\vec{a}\sF) \to \Kg^\circ$ to be $(\sY',s\sS'+\vec{a}\sF')_{|\Kg^\circ} \to \Kg^\circ$. 
\end{Notation}
Observe that the surface pairs appearing as fibers of $(\sY,s\sS+\vec{a}\sF) \to \Kg^\circ$ are all the possible stable Weierstrass fibrations with weight vector $I$.
\subsection{Construction of the parameter space $\cE_I$}\label{Subsection:construction:parameter:space:EI}
Once we define our parameter spaces $\cE_I$, we want to relate them through wall-crossing morphisms, reducing the weights. In order to reduce the weights on an irreducible component of $\sF$, it is convenient to have such an irreducible component as part of the data. But since in the formalism of \cite{KP} (see Subsection \ref{Subsection:stable:pairs}) the divisor $D$ does \emph{not} come with an ordering on the irreducible components, we need to find an ad hoc solution to keep track of them. 
Since for every $p \in \Kg^\circ(\spec(k))$ the point $\sF_p \cap \sS_p$ is a smooth point, for every irreducible component in $\sF_p$ we have a section $\tau_i:\Kg^\circ \to \sS$ which sends $p \mapsto (\sS \cap \sF_i)_p$. We will use $\tau_i$ to keep track of the irreducible components of $\sF$.

Observe that the
log-canonical divisor $K_{\sY/\Kg^\circ}+s\sS+\vec{a}\sF$ is $\mathbb{Q}$-Cartier, so the volume of the surface pairs in the family
$(\sY,s\sS+\vec{a}\sF) \to \Kg^\circ$ is constant on the connected components of $\Kg^\circ$. Since $\Kg$ is of finite type,
there are finitely many of possible volumes.
Let then $v_1,...,v_r$ be such volumes. Let 
$\overline{\I}:=\{b \in (0,1]: b=  \delta_{0}s+\sum a_i \delta_{i} $ for every $(\delta_{i})_{i=0}^n \in \{0,1\}^{n+1}\}$ (i.e. the possible numbers in $(0,1]$ obtained adding some of the $a_i$'s and $s$). Let $\mathcal{M}':=\bigcup_{i=1}^r \mathscr{M}_{v_i,\overline{\I}}$ be
the moduli of stable surface pairs of volumes $v_i$ for $1\le i \le r$ and
coefficient set $\overline{\I}$ (see Subsection \ref{Subsection:stable:pairs} for the definition of $\mathscr{M}_{n,\overline{\I}}$). Since $\sM_{v_i,\overline{\I}}$ is of finite type for every $i$, also $\mathcal{M}'$ is of finite type.

Therefore (recall that, from the beginning of Section \ref{Section:construction:of:the:moduli:surface:pairs}, $\Kg$ is normal) we have a morphism
$\Psi':\Kg^\circ \to \mathcal{M}'$
induced by the family $(\sY,s\sS+\vec{a}\sF) \to \Kg^\circ$.
Moreover, let $\mathcal{X} \to \mathcal{M}'$ be the universal family of surfaces. If we denote with
$\sH:=\sH om_{\mathcal{M}'}(\mathcal{M}',\mathcal{X})$ the Hom-stack of \cite{Homstack}, it parametrizes sections of $\mathcal{X} \to \mathcal{M}'$. Therefore, let $\mathcal{M}:=\mathcal{M} \times (\sH)^n$. The sections $\tau_i$ induce a morphism $\Psi: \Kg^\circ \to \mathcal{M}$.
 We define $\mathcal{W}_I$ (resp. $\cE_I^{sn}$) to be the normalization
 (resp. seminormalization) of the closure of the image of $\Psi$. 
 
 More precisely, take an atlas $U \to \mathcal{M}$,
and an atlas $V$ of $\Kg\times_{\mathcal{M}}U$. Then
there is a morphism
 $\psi:V \to U$ such that the following diagram commutes:
 $$\xymatrix{V \ar[r]^{\psi} \ar[d] & U\ar[d] \\ \Kg \ar[r]_{\Psi} & \mathcal{M}}$$
Let $R:=U \times_{\mathcal{M}}U$,
with the two projections $p_1,p_2:R\rightrightarrows U$. Now, $p_1^{-1}(\psi(V))$ and $p_2^{-1}(\psi(V))$ have the same $k$-points. Since
$p_1$ and $p_2$ are smooth, they are open, therefore $p_1^{-1}(\overline{\psi(V)})=\overline{p_1^{-1}(\psi(V))}$ and similarly for $p_2$. So
if we put the reduced structure on $\overline{\psi(V)}$, 
$$p_1^{-1}(\overline{\psi(V)})=
\overline{p_1^{-1}(\psi(V))}=\overline{p_2^{-1}(\psi(V))}=p_2^{-1}(
\overline{\psi(V)}).$$
Let us denote with $U_\mathcal{W}:=\overline{\psi(V)}$ and
$R_\mathcal{W}:=p_1^{-1}(\overline{\psi(V)})$, with their reduced structure.
Then if $\pi_i$ is the restriction $(p_i)_{|p_1^{-1}(\overline{\psi(V)})}$, the two arrows
$\pi_1,\pi_2:R_\cE \rightrightarrows U_\cE$ give a groupoid structure induced
by the one of $R\rightrightarrows U$. This defines a closed substack
$\widetilde{\mathcal{W}_{I}} \subseteq \mathcal{M}$.
\begin{Def}\label{Def:EI:and:EIsn}
 Let $\mathcal{W}_I$ be the normalization of $\widetilde{\mathcal{W}_I}$, and let $\cE_I^{sn}$ be its seminormalization.
 We will call $\mathcal{W}_I$ the \underline{moduli space of elliptic surfaces with
 weight data $I$}.
\end{Def}
Similarly, we can define $\cE_I'$ to be the normalization of the closure of the image of $\Psi'$. 
\begin{Notation}
  Given $I=(s,\vec{a},\beta)$ an admissible weight vector, we will denote with $\sX_I \to \cE_I$ (resp. $\sX_I^{sn} \to \cE_I^{sn}$) the universal surface.
\end{Notation}
 From \cite{Elkik}*{Th\'eor\'eme 4} and since
 $\sX_I \to \cE_I$  and $\sX_I^{sn} \to \cE_I^{sn}$ are universally closed, there are open substacks of $\cE_I$ and $\cE_I^{sn}$ which parametrize normal surfaces with rational singularities. But Du Val singularities are Gorenstein rational singularities (\cite{KM}*{Corollary 5.24}), and Gorenstein singularities are open. Therefore there is an open substack which parametrizes surfaces with only Du Val singularities. Moreover, having $n$ distinct marked fibers is an open condition as well. So there are open substacks $\cE_I^\circ \subseteq \cE_I$ and $(\cE_I^\circ)^{sn} \subseteq \cE_I^{sn}$ parameterizing stable Weierstrass fibrations with weight vector $I$.
 
\begin{Oss}\label{Oss:specialization:of:the:interior:of:the:moduli}
Since $\Kg$ is \emph{of finite type}, also $V$ is of finite type. Since $\mathcal{M}$ is locally of finite type,
using Chevalley's theorem we see that any point in $\overline{\psi(V)}$ is the specialization of a point in $\psi(V)$. Therefore
$\cE_I^\circ$ is dense in $\cE_I$.
\end{Oss}
 \begin{Oss}\label{Oss:definition:psi:KtoW}Recall that $\Kg$ is normal. Then the morphism $\Psi:\Kg \to \cM$, which induces $\Kg \to \widetilde{\cE_I}$, factors through the normalization $\cE_I \to \widetilde{\cE_I}$ (\cite{AB3}*{Lemma A.5}) giving $\Kg \to \cE_I$.
 \end{Oss}
\begin{Oss}Since $\mathcal{M}'$ is proper, also $\cE_I'$ is proper. Moreover, also $\cE_I$ and $\cE_I^{sn}$ are of finite type. This can be checked using that $\cE_I'$ is of finite type and that the Hom-scheme is a disjoint union of schemes of finite type (see \cite{arbarello-cornalba-griffiths}), so the moduli problems represented by $\cE_I$ and $\cE_I^{sn}$ are bounded. Finally, since the universal family of surfaces $\mathcal{X} \to \mathcal{M}$ is proper we can check, using the valuative criterion, that $\cE_I$ and $\cE_I^{sn}$ are proper.
\end{Oss}
Since the seminormalization is functorial (see \cite{Kollarsing}*{10.16}) the stack $\cE_I^{sn}$, when restricted to seminormal schemes, represents the following pseudofunctor. For $B$ seminormal, the objects of $\cE_I^{sn}(B)$ are the families $(X\to B,\omega_{X/B}^{\otimes m} \to \sL)$ as in Subsection \ref{Subsection:stable:pairs}, with $n$ sections $\{\tau_i\}_{i=1}^n$ of $X \to B$. Moreover, for every $b \in B$, the pair $(X_b,D_b)$ can be obtained as the closed fiber of a family of stable pairs $(X,D) \to \spec(R)$ over a DVR $R$, where the generic fiber is a minimal Weierstrass fibration with weight vector $I$. Finally, let $\eta$ be the generic point of $\spec(R)$. Then if $S_\eta$ is the section and $\vec{a}F_\eta:=D_\eta-S_\eta$, the point $\tau_i(b)$ is the limit of a point of intersection of $S_\eta$ and $\vec{a}F_\eta$.  

Now, assume that the surface parametrized by $\cE_I^\circ$ are \underline{not}
the product of two elliptic curves, with the divisor being the section $S$ and a fiber $F$. Then from Lemma \ref{Lemma:W:fibration:remembers:the:fibration}, the section $S$ is uniquely determined by $(X,D)$. Moreover, there is a morphism $\chi:\sY \to \sX_I$ induced by $\Kg \to \cE_I$, and a morphism $\chi^{sn}:\sY \to \sX_I^{sn}$. Proceeding as before we can introduce the following
\begin{Notation}\label{Notation:def:SI} We will denote with $\sS_I$ (resp. $\sS_I^{sn}$) the closure of $\chi(\sS)$ in $\sX_I$ (resp. $\chi^{sn}(\sS)$ in $\sX_I^{sn}$). Notice that $\sS_I$ (resp. $\sS_I^{sn}$) is a closed substack of the support of $\operatorname{Coker}(\phi:\omega_{\sX_I/\cE_I}^{\otimes m} \to \cL)$ (resp. $\operatorname{Coker}(\phi^{sn}:\omega_{\sX_I^{sn}/\cE_I^{sn}}^{\otimes m} \to \cL^{sn})$) where $\phi$ and $\phi^{sn}$ are obtained from Definition \ref{Def:KP:stable:pairs}.\end{Notation}
Similarly, for every $j$, let $\sF_{j,\sY}
  \hookrightarrow \sY$ be the irreducible component of $\sF$ with coefficient $a_j$. 
  \begin{Notation}We denote with $(\sF_{j})_I$ be closure of $\chi(\sF_{j,\sY})$ in $\sX_I$; and let $\sF_I :=\bigcup_j (\sF_j)_I$.
  \end{Notation}
  The $(\sF_j)_I$ are distinguished since we introduced the sections $\tau_j$.
\begin{Remark}
From now on, we will restrict ourselves to the case in which $\cE_I^\circ$ does \underline{not} parametrize surface pairs which are the product of two elliptic curves, with a single marked fiber.
\end{Remark}

\section{One parameter degenerations of Weierstrass fibrations}\label{section:singularities:threefold}
The goal of this section is to understand the boundary of $\cE_I$
finding the stable limits of a Weierstrass fibration. The case $\vec{a}=0$ and $s=1$ is studied in \cite{LaNave}, in \cite{AB3} is treated the case $s=1$ and $\vec{a}$ arbitrary. We want to understand what happens if $s$ is allowed to vary as well.

To fix the notation,
let $R$ be a DVR, and let $(X,sS+\vec{a}F) \to \spec(R)$ be a surface pair, induced by a morphism
$\spec(R) \to \cE_I$ for some weight vector $I$. Let
$\eta$ be the generic point of $\spec(R)$, $p$ the closed one, and assume that $\eta \mapsto \cE_I^\circ$.
\begin{Def}
 We will call a family of elliptic surfaces $(X,sS+\vec{a}F) \to \spec(R)$ as above, a \underline{stable degeneration}. We will call a threefold pair $(X,D) \to \spec(R)$ a \underline{degeneration} if there is an effective $\mathbb{Q}$-divisor $D'$ such that, $(X,D+D')$ is a stable degeneration, and $\supp(D')$ is the closure of some fibers on the generic fiber.
\end{Def}
The example we have in mind for a degeneration is obtained from a stable degeneration decreasing the weights on the fibers.

Our first goal is to understand the threefold $(X,sS+\vec{a}F)$. Since we are already
provided with a birational modification of $(X,sS+\vec{a}F)$, namely its associated tsm limit
$(X',sS'+\vec{a}F')$ (see Subsection \ref{subsection:definition:tsm:limit}),
this will be achieved taking the stable model of $(X',sS'+\vec{a}F')$. 

\subsection{Twisted stable maps-limits}\label{subsection:definition:tsm:limit}
In this subsection we construct a modification of a degeneration $(X,sS+\vec{a}F) \to \spec(R)$ such that,
up to replacing $\spec(R)$ with a ramified cover of it, is birational to $X$ (see also \cite{LaNave}*{Lemma 4.2.1}).
Let $\eta$ be the generic point of $\spec(R)$ and let $p$ be the closed one. By definition,
$(X,sS+\vec{a}F)_\eta$ comes with a morphism to a curve $g:X_\eta \to C_\eta$ such that it is a minimal Weierstrass fibration.
Therefore, there is an open subset $U \subset C_\eta$ such that $(g^{-1}(U),(S_\eta)_{|g^{-1}(U)}) \to U$ is a family of DM stable 
elliptic curves. Since $\overline{\mathcal{M}}_{1,1}$ is a proper DM stack, up to replacing $C_\eta$ with a suitable root-stack
$\sC_\eta' \to C_\eta$, the morphism $U \to \overline{\mathcal{M}}_{1,1}$ extends to a morphism $\sC_\eta' \to \overline{\mathcal{M}}_{1,1}$.
We can assume it to be representable, up to replacing $\sC_\eta'$ with the relative coarse moduli space of $\sC_\eta' \to \overline{\mathcal{M}}_{1,1}$.
Observe also that $(\sC_\eta')_{|U}=(C_\eta)_{|U}$.

Let $(\Sigma_i)_\eta \subseteq \sC_\eta'$ be the cosed substack which corresponds to points with non-trivial stabilizers.
Then $(\sC_\eta', \{(\Sigma_i)_\eta \}) \to \overline{\mathcal{M}}_{1,1}$
is a twisted stable map. But since $\mathcal{K}_{g,n}(\overline{\mathcal{M}}_{1,1},d)$ is proper, up to replacing $\spec(R)$ with
a ramified cover of it, we can extend the twisted stable map to get a family of twisted stable curves $(\sC', \{\Sigma_i\}) \to \spec(R)$, with
a morphism
$(\sC', \{\Sigma_i\}) \to \overline{\mathcal{M}}_{1,1}$. The latter corresponds to a family of genus one DM stable curves with a section:
$$(\sX', \sS') \to (\sC', \{\Sigma_i\}) \to \spec(R)$$
Observe that $\sX'\times_{\sC'}U \cong X\times_C U$.

We now take the coarse spaces, $(\sX, \sS) \to (\sC, \{\sigma_i\}) \to \spec(R)$, and proceed as in Subsection
\ref{subsection:construction:family:of:surfaces} to replace all the multiple twisted fibers with minimal cusps (see also Observation \ref{Oss:description:that:the:blow:ups:replace:twisted:with:mcusps}).
Let $X'$ be the resulting threefold.
From the construction there is an isomorphism $\xi:X_\eta \to X'_\eta$, thus let $S':=\xi_*(S)$ and $\vec{a}F':=\xi_*(\vec{a}F)$.
We denote with $(X',sS'+\vec{a}F')$ the resulting surface pair.
Observe that we have a morphism $X' \to C'\to \spec(R)$ where $C' \to \spec(R)$ is a family of nodal curves.
\begin{Def}\label{Def:tsm:limit}
 With the same notation as before,
 we call the family $(X',sS'+\vec{a}F') \to \spec(R)$ the \underline{twisted stable maps-limit} (or tsm limit) associated to
 $(X,sS+\vec{a}F)$. 
\end{Def}
Observe in particular that $(X'_\eta,sS'_\eta+\vec{a}F'_\eta) \cong (X_\eta,sS_\eta+\vec{a}F_\eta)$.
\begin{Remark}
 A priori the tsm limit depends on the ramified cover of $\spec(R)$ we choose. We will ignore this subtlety since it will not cause any issue in what follows.
\end{Remark}
\begin{Oss}\label{Oss:cosed:fiber:tsm:limit:is:lc}
Observe that from the construction of $(X',sS'+\vec{a}F')$ of Subsection \ref{subsection:construction:family:of:surfaces}, the closed fiber $(X'_p,sS_p'+\vec{a}F'_p)$ is slc. Moreover, proceeding as in Subsection \ref{subsection:construction:family:of:surfaces}, one can show that $S'$ and each irreducible component of $F'$ is
 $\mathbb{Q}$-Cartier.\end{Oss}
Thus we constructed a birational modification $(X',sS'+\vec{a}F')$ of $(X,sS+\vec{a}F)$. All the irreducible components of $X'_p$
are lc elliptic fibrations, with all the fibers irreducible. 

\subsection{Stable reduction}
The goal of this subsection is to study a stable degeneration, taking the stable model of $(X',sS' + \vec{a}F')$. Our main result is the following:
\begin{Teo}\label{Teo:descripton:bir:trans:to:get:lcdeg} Let $(X,sS+\vec{a}F)$ be a stable degeneration, let $p$ (resp. $\eta$) be the closed (resp. generic) point of $\spec(R)$, and let $(X',sS'+\vec{a}F')$ be its tsm limit.
 Then there is a $\mathbb{Q}$-divisor $G^{(1)}$, with each irreducible component of $\supp(G^{(1)})$ which is $\mathbb{Q}$-Cartier, such that
 $(X^{(1)},D^{(1)}):=(X',sS'+\vec{a}F'+G^{(1)})$ is a stable degeneration.
 Moreover, we can obtain $(X,sS+\vec{a}F)$ from $(X^{(1)},D^{(1)})$ performing a series of birational transformations
 $$\xymatrix{\ar @{} [r]
 (X^{(1)},D^{(1)}) \ar@{.>}^{f^{(1)}}[r] &(X^{(2)},D^{(2)}) \ar@{.>}^-{f^{(2)}}[r]&\text{...} \ar@{.>}^-{f^{(m-2)}}[r]
 &(X^{(m-1)},D^{(m-1)}) \ar@{.>}^-{f^{(m-1)}}[r] & (X,sS+\vec{a}F)}$$
 satisfying the following conditions:
 \begin{itemize}
 \item $(X^{(i)},D^{(i)})$ is a stable pair; 
  \item  $D^{(i)}$ is obtained from $f^{(i-1)}_*(D^{(i-1)})$ reducing the weights on
$(f^{(i-1)} \circ... \circ f^{(1)})_*(G^{(1)})$;
  \item If $i<m-1$, the rational morphism $f^{(i)}$ is produced through some steps of the MMP which are either a divisorial contraction of some irreducible components of $X^{(i)}_p$, or a flip of La Nave,
  or the composition of such a divisorial contraction and a flip of La Nave;
  \item $f^{(m-1)}$ is a morphism, and is
 the contraction of some (possibly none)
 irreducible components of $X^{(m-1)}_p$ and $S^{(m-1)}_p$, where $S^{(1)}:=S'$ and $S^{(i+1)}:=f^{(i)}_*(S^{(i)})$, and
 \item If $C$ is an irreducible component of $S'$ that gets contracted through $f^{(m-1)}\circ f^{(m-2)} \circ ... \circ f^{(1)}$, then
 $C \cong \mathbb{P}^1$.
 \end{itemize}
\end{Teo}
We will produce $f^{(i)}$ through some steps of the MMP,
whereas $f^{(m-1)}$ will be obtained through log-abundance. 
For the proof we mainly follow the strategy in \cite{AB3}. The
main ingredient will be Theorem \ref{Teo:generalization:of:appendix}, 
which
is a slight generalization of \cite{AB3}*{Theorem B.10}. The proof is the same, just notice that in \cite{AB3}*{Appendix B},
we never used that $s=1$.
\begin{Teo}\label{Teo:generalization:of:appendix}
   Assume that $(X,sS+\vec{a}F) \to \spec(R)$ is a stable degeneration, over a DVR $R$. Assume that
 we can write $\vec{a}F=\vec{b}F+G$ where $G$ is an effective $\mathbb{Q}$-Cartier $\mathbb{Q}$-divisor.
 Assume finally that $(X,sS+\vec{b}F+\beta G)$ is stable for every rational $\beta_0 < \beta \le 1$, but $K_X+S+\vec{b}F+\beta_0G$ is nef.
  Then the codimension two exceptional locus arising from taking the stable model of $(X, sS + \vec{b}F+\beta_0 G)$ will be a union of components of
  the section of the closed fiber.
  
  Moreover, if $\supp(G)$ is irreducible and $\beta_0>0$, there is an $\epsilon$ small enough such that the stable model of $(X,sS+\vec{b}F+
  (\beta_0-\epsilon)G)$ can be obtained from $(X,sS+\vec{b}F+
  (\beta_0-\epsilon)G)$ performing some divisorial contractions and at most one flip of La Nave.
\end{Teo}
\begin{proof}[Proof of Theorem \ref{Teo:descripton:bir:trans:to:get:lcdeg}]
We apply Theorem \ref{Teo:generalization:of:appendix}. First, we find a $\mathbb{Q}$-Cartier
$\mathbb{Q}$-divisor $G'$, which makes $(X',sS'+\vec{a}F'+G')$ a stable pair. 
We choose $G'$ as follows.
Let $\{Y^i\}_i$ be the irreducible components of $X'_p$. Then
$Y^i$ is an elliptic fibration with all the fibers irreducible.
We choose some fibers on it, say $F_1^i,...,F_{m_i}^i$, such that
they do not intersect the double locus $Y^i\smallsetminus \overline{(X'_p\smallsetminus Y^i)}$. Let $\{F_1,...,F_m\} =\bigcup _{i,j} \{F^i_{m_j}\}$
be union of the $F_j^i$.

Let $C' \to \spec(R)$ be the family of nodal curves associated to $X'$, and let $h':X' \to C'$ be the associated morphism. Then the fibers
$F_i$ map to some closed points $q_1,...,q_m$ of $ C'_p$, supported on the smooth locus of $C'_p$. Up to replacing $\spec(R)$ with a covering
of it, we can assume that there are closed points $x_1,...,x_m$ of $C'_\eta$, such that $\overline{\{x_i\}} \cap X'_p= q_i$.
Then $\overline{\{x_i\}}$ are Cartier divisors. Up to adding some fibers $F_i$, we can choose
$G^{(1)}:= \sum b_i (h')^*(\overline{\{x_i\}})$ such that $(S' \cap Y^i).G^{(1)}=3$, and $b_i$ are small and positive, such that
$(X',sS'+\vec{a}F'+G^{(1)})$ is lc.

\begin{center}
\includegraphics[scale=0.30]{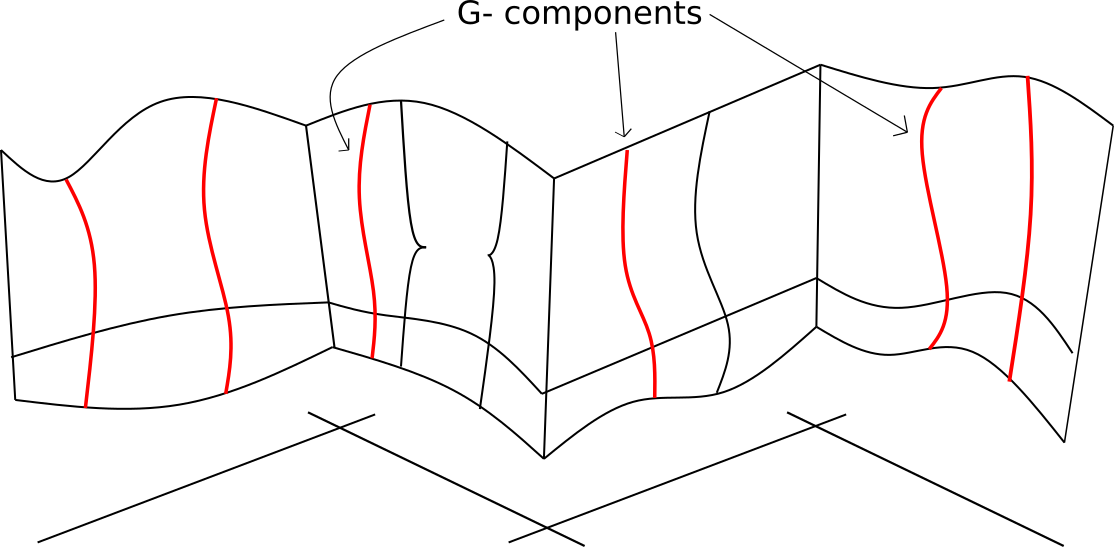}
\end{center}
Now we show that $(X',sS'+\vec{a}F'+G^{(1)})$ is a stable pair.
We need to check that $K_{X'}+sS'+\vec{a}F'+G^{(1)}$ is ample when restricted to the the generic fiber, and to every irreducible component of the closed fiber.
This follows from Lemma \ref{Lemma:conditions:when:to:contract:section}, since we added the fibers $G^{(1)}$. 

Now we can apply Theorem \ref{Teo:generalization:of:appendix}, and decrease the weights on $G^{(1)}$ one at the time.
This produces a series of birational transformations 
$$\xymatrix{\ar @{} [r]
 (X^{(1)},D^{(1)}) \ar@{.>}^{f^{(1)}}[r] &(X^{(2)},D^{(1)}) \ar@{.>}^-{f^{(2)}}[r]&\text{...} \ar@{.>}^-{f^{(m-2)}}[r] &
 (X^{(m-1)},D^{(m-1)})}$$
Since we want all the $f^{(i)}$ to be steps of the MMP, to guarantee that our divisors remain $\mathbb{Q}$-Cartier, we need to avoid any small contraction.
So we decrease the weights until they are all small, but positive, rational numbers.
 Then for every $i$, from Theorem \ref{Teo:generalization:of:appendix} the birational transformation
 $f^{(i)}$ is either a divisorial contraction of an irreducible component of the special fiber, or a flip of La Nave, or the composition
 of both. 
 Moreover, $D^{(i)}$ is an effective divisor supported on $\supp((f^{(i)} \circ ... \circ f^{(1)})_*(sS'+\vec{a}F'+G'))$,
 and $G^{(i)}:=D^{(i)}-(f^{(i)} \circ ... \circ f^{(1)})_*(sS'+\vec{a}F')$ is effective.
 Each $f^{(i)}$ is a composition of steps of the MMP, so using \cite{Fuj-contraction}*{Theorem 16.4, (3)} and proceeding as in \cite{KM}*{Proposition 3.36, 3.37}
 one can show that $G^{(i)}$ are $\mathbb{Q}$-Cartier for every $i$.
 Then we can proceed reducing the weights until $(f^{(m-2)} \circ ... \circ f^{(1)})_*(sS'+\vec{a}F')$ is nef.
 At this point Theorem \ref{Teo:generalization:of:appendix} applies again.
 
 We are left with showing the last bullet point. Since $f^{(i)}$ does not contract curves which are positive for $K_{X^{(i)}}+D^{(i)}$, and since we are not contracting $S_\eta$, the last bullet point follows from Lemma \ref{Lemma:conditions:when:to:contract:section}.
\end{proof}
\begin{Cor}\label{Cor:description:lc:degeneration}
 Let $(X,sS+\vec{a}F) \to \spec(R)$ be a stable degeneration. Then:
 \begin{enumerate} \item There is a flat family of nodal curves $C \to \spec(R)$ and a map $h:X \to C$ such that $h_{|S}:S \to C$ is an isomorphism; \item The irreducible components of $(X,sS+\vec{a}F)_p$ are either pseudoelliptic surfaces (Definition \ref{Def:pseudoelliptic}) which map to a point through $h$, or elliptic surfaces (Definition \ref{Def:slc:elliptic:surface}), and
  \item  The double locus of $(X,sS+\vec{a}F)_p$ is supported on some twisted fibers or pseudofibers, and on the twisted components of every intermediate fiber and pseudofiber.
 \end{enumerate}
\end{Cor}
Notice that Corollary \ref{Cor:description:lc:degeneration} gives a description of the possible surface pairs of $\cE_I (\spec(k))$:
\begin{center}
\includegraphics[scale=0.23]{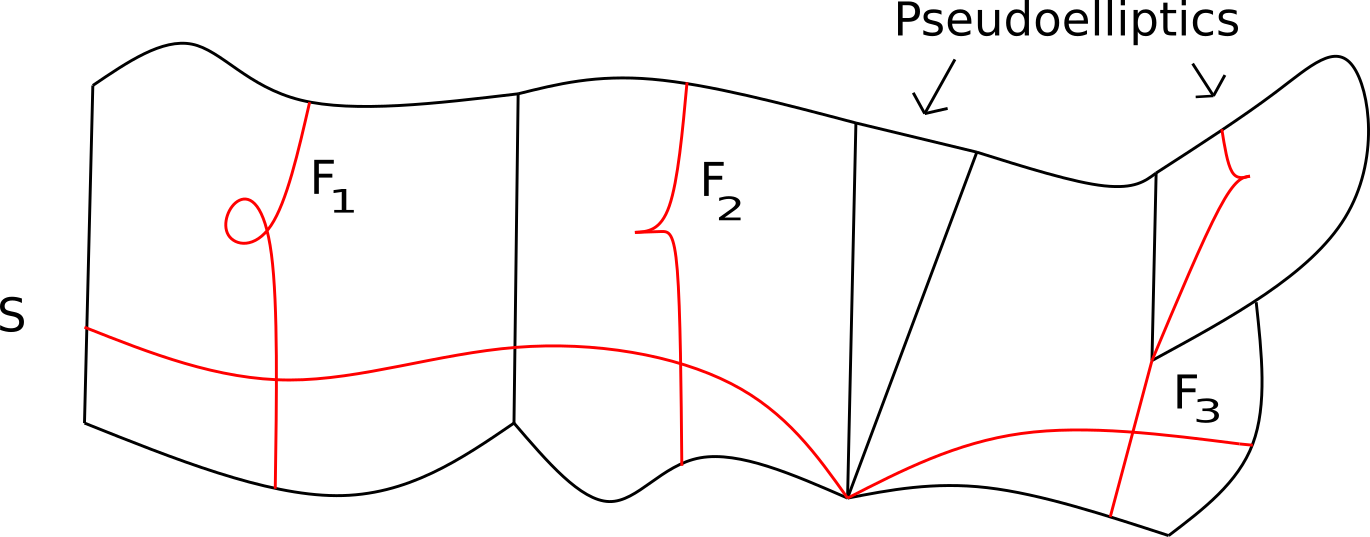}
\end{center}
\begin{proof}
 With the notation of Theorem \ref{Teo:descripton:bir:trans:to:get:lcdeg}, we show by induction that (1), (2) and (3) hold
 for each $(X^{(i)},D^{(i)})$. 
 For $(X',sS'+\vec{a}F')$ the points (1), (2) and (3) hold. Assume they hold for $(X^{(i-1)},D^{(i-1)})$, and let
 $h^{(i-1)}_{|S^{(i)}}:S^{(i-1)} \to C^{(i-1)}$ be the isomorphism of (1).

Observe that the third and the fourth bullet points of Theorem \ref{Teo:descripton:bir:trans:to:get:lcdeg} provide us with the possible choices for the transformation $f^{(i-1)}:X^{(i-1)} \dashrightarrow X^{(i)}$. 
 Then from the definitions of pseudoelliptic and elliptic surfaces and the description of the flip of La Nave, (2) and (3) hold for $(X^{(i)},D^{(i)})$. Moreover, if $f^{(i-1)}$ is the contraction of a pseudoelliptic component, then also (1) holds. We are left with checking that (1) holds even if $f^{(i-1)}$ is not the contraction of a pseudoelliptic component.
 
 From Theorem \ref{Teo:descripton:bir:trans:to:get:lcdeg},
the irreducible component $C$ of $S^{(i-1)}_p$ that gets contracted through $X^{(i-1)} \dashrightarrow X^{(i)}$ is
 isomorphic to $\mathbb{P}^1$. Then let $Y$ be the elliptic component containing $C$.
From Lemma \ref{Lemma:conditions:when:to:contract:section}, and since $X^{(i-1)} \dashrightarrow X^{(i)}$ contracts only non-positive curves, the number of fibers of $Y$ contained in the double locus of $X^{(i-1)}_p$ is at most 2. Therefore $C$ has at most 2 points in the double locus of  $S^{(i-1)}_p$. By inductive hypothesis, the same holds for $D:=h^{(i-1)}(C)$ and $C^{(i-1)}_p$. But $C^{(i-1)} \to \spec(R)$ is a family of nodal curves, with the generic fiber smooth. So there is a contraction morphism $C^{(i-1)} \to C^{(i)}$ that contracts $D$.
This produces a new (flat) family of curves $C^{(i)} \to \spec(R)$.

Now, since $X^{(i)}$ is normal, it is clear (using \cite{Giansiracusa-Gillam}*{Theorem 7.3}) that there is a morphism $X^{(i)} \to C^{(i)}$, which induces $S^{(i)} \to C^{(i)}$; and the latter is an homeomorphism. Moreover, we know that $S^{(i)}_\eta \to C^{(i)}_\eta$ is an isomorphism. Then from \cite{Giansiracusa-Gillam}*{Theorem 7.3},
the rational map $C^{(i)} \dashrightarrow S^{(i)}$ extends to a morphism $\sigma:C^{(i)} \to S^{(i)}$. But from the Zariski main theorem, the composition $C^{(i)} \to S^{(i)} \to C^{(i)}$ is an isomorphism. Therefore $\cO_{S^{(i)}} \to \sigma_* \cO_{C^{(i)}}$ is surjective, so $\sigma$ is a closed embedding. Since the divisor $S^{(i)}$ is the closure of $S^{(i)}_\eta$, we get that $S^{(i)} \to C^{(i)}$ is an isomorphism.
\end{proof}
\begin{Cor}\label{Cor:SI:is:a:fam:of:nodal:curves}The morphism $\sS_I \to \cE_I$ is a family of nodal curves.
\end{Cor}
\begin{proof}Recall that $\sS_I$ is defined as the closure of the image of $\chi(\sS)$ in $\sX_I$ (see Notation \ref{Notation:def:SI}). In particular, whenever we take a stable degeneration $(X,sS+\vec{a}F) \to \spec(R)$ induced by a morphism $\spec(R) \to \cE_I$, the pull-back $\sS_I \times_{\cE_I} \spec(R)$ is \emph{supported} over $S$, since both are the closure of $\sS_I \times_{\cE_I} \spec(k(\eta))$ where $\eta$ is the generic point of $\spec(R)$. To prove the desired result, it suffices to show that
$\sS_I \times_{\cE_I} \spec(R)$ agrees with $S$, i.e. we need to show that it is reduced. 

Consider
$B \to \cE_I$ an atlas which is a scheme, and let $\sX_B:= \sX_I \times_{\cE_I} B$. 
Up to shrinking $B$ we can assume that $\sX_B \to B$ is projective. Fix an embedding
$\sX_B \hookrightarrow \mathbb{P}^N_B$, and let $\sS_B:=\sS_I \times_{\cE_I}B$.
From Corollary \ref{Cor:description:lc:degeneration}, for each $b \in B$, the scheme $(\sS_B)_b$ is supported on a nodal curve, and (again from Corollary \ref{Cor:description:lc:degeneration}) for every stable degeneration $(Y,sS+\vec{a}F) \to \spec(R)$, the divisor $S$ is a \emph{flat} family of (reduced) nodal curves. Furthermore, if we assume such a stable degeneration comes from a morphism $\spec(R) \to B$, then $S=(\sS_B \times_B \spec(R))^{red}$ and if $\eta$ is the generic point of $\spec(R)$, then $S_\eta=(\sS_B)_\eta$. If we show that $\sS_B$ is flat, then $S=\sS_B \times_B \spec(R)$ from the uniqueness of the flat limit, and the latter has no embedded points. We use the results of \cite{kollarbook}*{Chapter 4}.

Since $B$ is \emph{normal} and from \cite{kollarbook}*{Theorem 4.26 and 4.2}, $\sS_B \to B$ is a \emph{family of generically Cartier divisors}, and for every stable degeneration $(Y,sS+\vec{a}F) \to \spec(R)$, the morphism $S \to \spec(R)$ is flat. But then the Hilbert polynomial of $S_\eta$ agrees with the one of $S_p$.
Then the flatness of $\sS_B \to B$ follows from \cite{kollarbook}*{Proposition 4.34}.
\end{proof}

\section{Stable reduction algorithm and $\mathbb{Q}$-Cartier chambers}\label{Section:QCartier:threshold}
This section is mainly devoted at showing that it is possible to divide the set of all admissible weights into finitely many chambers satisfying the following condition.
For every $I_1:=(s_1,\vec{a}_1,\beta)$, $I_2:=(s_2,\vec{a}_2,\beta)$ in the same open chamber,
let $(X',s_1S'+\vec{a}_1F')$ be the tsm limit of $(X,s_1S+\vec{a}_1F)  \to \spec(R)$. We show that the stable
model of $(X',s_2S'+\vec{a}_2F')$ is $(X,s_2S+\vec{a}_2F)$. Observe that this is a \emph{necessary} condition for having
a \emph{finite} wall and chamber decomposition.
To achieve this, we study the steps of stable reduction we perform on $(X',sS'+\vec{a}F')$.
The results of this section (especially Theorem \ref{Teo:stable:model:depends:only:on:num:data}) will be the key ingredients to prove Theorem \ref{Teo:intro:have:flat:fibration}  (Theorem \ref{Teo:no:psudo}).
\begin{Def}\label{Def:numerical:data}Let $R$ be a DVR and let $p$ (resp. $\eta$) be the closed (resp. generic) point of $\spec(R)$. Let $C_i'$ be the irreducible components of $S_p'$. The \underline{numerical data} associated to a tsm limit $(X',sS'+\vec{a}F') \to \spec(R)$ is the data of:
 \begin{itemize}
  \item The dual weighted graph of $(S_p',\vec{b}F'_{|S'_p})$ for every rational vector $\vec{b}$;
  \item For every $i$, the intersection numbers $(K_{X'}).C_i'$ and $S'.C_i'$.
 \end{itemize}
\end{Def}
Observe that $(\vec{a}F').C'_i$ is determined by the first bullet point, since $F'$ intersects $C_i$ transversally in the smooth locus. Observe also that
the numerical data of a tsm limit depends only on its special fiber, so we give the following definition:
\begin{Def}\label{Def:ref:num:data:special:fiber}
 Let $(X',sS'+\vec{a}F') \to \spec(R)$ be a tsm limit, and let $p$ be the closed point of $\spec(R)$.
 We define the numerical
 data of $(X',sS'+\vec{a}F')_p$ to be the one of $(X',sS'+\vec{a}F')$.
\end{Def}

Now, let $(X',sS'+\vec{a}F')$ be a tsm limit. Let $(X,sS+\vec{a}F)$ be the
stable model of $(X',sS'+\vec{a}F')$. Using Theorem \ref{Teo:descripton:bir:trans:to:get:lcdeg}, we can
factor $X' \dashrightarrow X$ into a sequence of explicit birational transformations
$X' \dashrightarrow X^{(2)} \dashrightarrow ... \dashrightarrow X^{(m)}=X$.
A priori, the number $m$ and the order of these birational transformations is not unique. We show that we can choose these birational transformations
using only the numerical data. Namely,
we show that, once
we choose $m$ and such an order for $X' \dashrightarrow X$, for any tsm limit $(X'',sS''+\vec{a}F'')$ with the same numerical data of
$(X',sS'+\vec{a}F')$, we can assume that stable reduction is performed
applying $m$ birational transformations of Theorem \ref{Teo:descripton:bir:trans:to:get:lcdeg}, in the same order (see Theorem \ref{Teo:stable:model:depends:only:on:num:data}) .
Thus, if we show that the kind of birational transformations one has to perform on $X'$ to get the stable model of
$(X',s_1S'+\vec{a}_1F')$ are the same as those to get the stable model of $(X',s_2S'+\vec{a}_2F')$, the same conclusion will hold for any tsm limit $X''$
with the same numerical data.
Therefore we show that after a non-canonical choice (namely such an ordering for the stable reduction
$X' \dashrightarrow X$), the steps of stable reduction for any other
$(X'',sS''+\vec{a}F'')$ as above are uniquely determined.

Next, we observe that we can stratify an atlas of $\bigcup _{n \le m \le d}\mathcal{K}_{g,m}(\overline{\mathcal{M}}_{1,1},d)$ into finitely many strata $\mathcal{Z}_i$,
such that any two tsm limits that limit to a point in $\mathcal{Z}_i$, have the same numerical data (Proposition \ref{Prop:stratification:exists}).
Then it is enough to show that such a chamber-decomposition exists for a \emph{fixed} $(X',sS'+\vec{a}F')$, which
follows from studying the birational transformations in stable reduction.
\subsection{Numerical data and stable reduction}\label{Subsection:alg:nature:stable:reduction}
Let $(X',sS'+\vec{a}F')$ be a tsm limit over $\spec(R)$, and let $X' \to C'$
the corresponding morphism to a family of nodal curves. This subsection is aimed at proving Theorem \ref{Teo:stable:model:depends:only:on:num:data}. In particular, we show that the steps of stable reduction on $(X',sS'+\vec{a}F')$
can be chosen only using its numerical data.

The main idea is the following. Theorem \ref{Teo:descripton:bir:trans:to:get:lcdeg} describes the possible birational transformations
one has to perform on $X'$ to get $X$.
In particular, in order to have either a flip or a small contraction, we need to contract a component of the section of the special fiber.
Then we can control when flips happen, checking when a section-component is a negative curve. Similarly, we can check when a divisorial contraction
happens checking when the log-canonical divisor, when restricted to an irreducible component of the special fiber, has self intersection 0.
Our
goal is to show that all this
can be checked using the numerical data of $(X',sS'+\vec{a}F')$.

We begin with a definition that generalizes Definition \ref{Def:numerical:data}:
\begin{Def}\label{Def:refined:num:data}
Let $R$ be a DVR and let $p$ (resp. $\eta$) be the closed (resp. generic) point of $\spec(R)$.
 Given a degeneration $(Y,sS+\vec{a}F) \to \spec(R)$, we say that it has a \underline{refined numerical data} if each irreducible
 component of $\supp(sS+\vec{a}F)$ is $\mathbb{Q}$-Cartier. In this case, its refined numerical data consists of:
 \begin{itemize}
  \item The dual weighted graph of $(S_p,\vec{c}F_{|S_p})$ for every rational vector $\vec{c}$;
  \item For every $C_i$ irreducible component of $S_p$,
  the intersection numbers $(K_{Y}).C_i$ and $S.C_i$;
  \item For every irreducible component $Z$ of $Y_p$ and every rational vector $\vec{c}$, the intersection numbers
  $((K_{Y}+sS+\vec{c}F)_{|Z})^2$.
 \end{itemize}
\end{Def}

Notice that the refined numerical data of $(Y,sS+\vec{a}F)$ does not depend on $s$ and $\vec{a}$. 
 Observe also that
the refined numerical data of a degeneration depends only on its special fiber:
\begin{Def}\label{Def:num:data:special:fiber}
 Let $(X,sS+\vec{a}F) \to \spec(R)$ be a degeneration, let $p$ be the closed point of $\spec(R)$.
 We define the refined numerical
 data of $(X,sS+\vec{a}F)_p$ to be the one of $(X,sS+\vec{a}F)$.
\end{Def}
\begin{Lemma}\label{Lemma:numerical:data:gives:rnum:data} Let $(X,sS+\vec{a}F)$ be a tsm limit. Then the refined numerical data of $(X,sS+\vec{a}F)$
is determined by its numerical data. \end{Lemma}
\begin{proof} Let $C_i$ be the irreducible components of $S_p$.
For every $i$ and every weight vector $\vec{b}$, the following intersection pairings are part of the data:
 $(K_{X}).C_i$,
  $(S).C_i$, and
  $(\vec{b}F).C_i$. We need to show that these determine the third bullet point in Definition \ref{Def:refined:num:data}.

Let $Y_i$ be the irreducible components of $X_p$.
Since $\pi_i:Y_i \to C_i$ is an
elliptic fibration with all the fibers irreducible, from Observation \ref{Oss:all:fib:irred:implies:pulback} there are $\mathbb{Q}$-Cartier $\mathbb{Q}$-divisors $D_1,D_2 \subseteq C_i$ such that
$\pi_i^*(D_1)=(K_{X})_{|Y_i}$ and $\pi^*(D_2)=\vec{b}F_{|Y_i}$.
Therefore we have $(\pi_i^*(D_1+D_2)+sC_i)^2=2s(\pi_i^*(D_1+D_2)).S +s^2(S)^2=2s(((K_{X})_{|Y_i}).C_i+(\vec{b}F).C_i)+s^2S^2$.\end{proof}
We study how the refined numerical data changes after a transformation of Theorem \ref{Teo:descripton:bir:trans:to:get:lcdeg}:
\begin{Lemma}\label{Lemma:recovering:refined:numerical:data}
 Let $(X,D):=(X,sS+\vec{a}F) \to \spec(R)$ be a degeneration, with each irreducible component of $\supp(D)$
 that is $\mathbb{Q}$-Cartier.
 Assume that $(Y,D_Y)$ is obtained from $(X,D)$ through a step of the MMP
which is either a divisorial contraction of an irreducible component of the special fiber, or La Nave's flip.
Then the refined numerical data of $(Y,D_Y)$
is determined by the one of $(X,D)$, and the type of birational transformation
$\pi:X\dashrightarrow Y$.
\end{Lemma}
With the notation of Theorem \ref{Teo:descripton:bir:trans:to:get:lcdeg}, we will use Lemma \ref{Lemma:recovering:refined:numerical:data} on $(X^{(i)},(f^{(i)})^{-1}_*D^{(i+1)})$ for $i<m-1$.\\
\emph{Proof}.
Using \cite{Fuj-contraction}*{Theorem 16.4, (3)} and proceeding as in \cite{KM}*{Proposition 3.36, 3.37},
 one can show that every irreducible component of $\supp(D_Y)$
 is $\mathbb{Q}$-Cartier. Let $S_Y:=\pi_*(S)$ and $\vec{a}F_Y:=\pi_*(\vec{a}F)$. 
 
 Since we know $\pi$, we know if it either contracts an irreducible component of $S_p$ or not. If
 it does, let $C_j$ be such a component.
  Since the dual weighted graph of $(S_p,\vec{a}F_p)$ and the edge corresponding to $C_j$ are part of the data, we know the dual
  weighted graph of $((S_Y)_p,(\vec{a}F_Y)_p)$. Thus we only need to show that we can recover the second and third bullet points of the definition of refined
  numerical data (Definition \ref{Def:refined:num:data}).
  
  Let then $M,N \subseteq X$ be two $\mathbb{Q}$-Cartier $\mathbb{Q}$-divisors such that
  $\supp(  M)$ and $\supp( N )$ are flat over $\spec(R)$.
  Let $M',N' \subseteq Y$ be their proper
  transform. Proceeding as above, we can show that $M'$ and $N'$ are $\mathbb{Q}$-Cartier.  Let $Z_j$ be the irreducible components of $X_p$, and let $Z_j':=\pi_*(Z_j)$. For every $Z_j$ (resp. $Z_j'$), let
  $M_j:=M_{|Z_j}$ (resp. $M'_j:=M'_{|Z'_j}$) and $N_j:=N_{|Z_j}$ (resp. $N'_j:=N'_{|Z'_j}$).
 
From the explicit description of $\pi$, there are $Z_1$ and $Z_2$
  such that $\pi$ is an isomorphism on $X \smallsetminus \{Z_1,Z_2\}$.
  These two irreducible components are, a priori, not uniquely determined. However, they always exist.  
  Then we can compute:
\begin{align*}
(M',N'_1)&=(M',N'_p - (N'_p - N'_1))
=(M',N'_p)-(M',N'_p - N'_1)=\\
&=(M'_\eta,N'_\eta)-(M',N'_p - N'_1-N'_2)-(M',N'_2)=\\&
=(M_\eta,N_\eta)-(M,N_p - N_1-N_2)-(M',N'_2)=\\
&=(M,N_p)-(M,N_p - N_1-N_2)-(M',N'_2)=(M,N_1+N_2)-(M',N'_2)
\end{align*}
where the fourth equality follows since $X \smallsetminus \{Z_1,Z_2\} \cong Y \smallsetminus \{Z_1',Z_2'\}$.

Now, we choose $Z_2$
such that $(S_Y)_{|Z'_2}=0$. Then
if we replace $N$ with $S$ and $M$ with either $K_{X}$, $\vec{a}F$ or $S$; we have  $(M',N'_2)=0$.
Therefore we can recover the second bullet point of Definition
\ref{Def:refined:num:data}.

For the third bullet point, we replace $M$ and $N$ with $K_{X}+ D$. If the contraction is divisorial, then $Z'_2=0$ and again
we can use the equalities above right away. Otherwise, $Z'_2$ is a pseudoelliptic
component and we need to show that we can recover
$((K_{Y}+ D_Y)_{|Z'_2})^2$ from the refined numerical data of $(X,D)$. Observe
now that since $X_p$ and $Y_p$ are nodal in codimension 1, we can compute that
$(K_{X})_{|Z_2}=K_{Z_2}+E \text{ and }(K_{Y})_{|Z'_2}=K_{Z'_2}+E'$,
where $E'$ and $E$ are is supported on the double locus. The following observation finishes the proof of Lemma \ref{Lemma:recovering:refined:numerical:data}:
\begin{Oss}\label{Oss:we:can:recover:selfintersection:lc:on:psudo}
 Let $(Z',\vec{a}F'+E')$ be a pseudoelliptic surface, with one twisted pseudofiber $E'$,
 and assume $Z'$ is obtained from $(Z,sS+\vec{a}F+E)$ contracting the section. Let
 $L_{Z'}$ be the lc divisor of $(Z',\vec{a}F'+E')$ and $L_Z$ the one of $(Z,sS+\vec{a}F+E)$.
 Then if  $c:=\frac{(K_Z+\vec{a}F+E,S)}{-(S,S)}$, we have $$(L_{Z'})^2=(L_Z)^2+2(c-b)(K_{Z}+sS+\vec{a}F+E).S+(c-b)(S)^2.$$
 
 In particular, we can determine $(L_{Z'})^2$ from some intersection pairings on $Z$.
\end{Oss}
\begin{proof}[Proof of Observation \ref{Oss:we:can:recover:selfintersection:lc:on:psudo}]
Let $p:Z \to Z'$ be the contraction of $S$. Then there is a $c$ such that
$$K_{Z}+cS+\vec{a}F+E=p^*(K_{Z'}+\vec{a}F'+E').$$
We can compute $c=\frac{(K_Z+\vec{a}F+E,S)}{-(S,S)}$. Then
\begin{align*}
 &(L_{Z'})^2=(K_{Z}+cS+\vec{a}F+E)^2=(K_{Z}+sS+\vec{a}F+E+(c-s)S)^2=\\
 &=(K_{Z}+sS+\vec{a}F+E)^2+2(c-s)(K_{Z}+sS+\vec{a}F+E).S+(c-s)^2(S)^2.
\end{align*}\end{proof}
We are finally ready to prove the following theorem:
\begin{Teo}\label{Teo:stable:model:depends:only:on:num:data}
 Let $(X',sS'+\vec{a}F')$ and $(X'',sS'+\vec{a}F'')$ be lc pairs which are either:
 \begin{enumerate}
 \item Two tsm limits with the same numerical data, or
 \item Two degenerations having two effective $\mathbb{Q}$-divisors $\vec{b}G'$ and $\vec{b}G''$ such that $(X',sS'+\vec{a}F'+\vec{b}G')$ and $(X'',sS''+\vec{a}F''+\vec{b}G'')$ are stable degenerations with the same refined numerical data.
 \end{enumerate}

 Assume that, to take the stable model of $(X',sS'+\vec{a}F')$, we perform $r$ birational transformations of Theorem \ref{Teo:descripton:bir:trans:to:get:lcdeg},
 $$X'=:X^{(1)}\dashrightarrow X^{(2)}\dashrightarrow...\dashrightarrow X^{(r-1)}\dashrightarrow X^{(r)}=X$$
 
 Then we can take the stable model of $(X'',sS'+\vec{a}F'')$ performing $r$ birational transformations, 
 $$X''=:Z^{(1)}\dashrightarrow Z^{(2)}\dashrightarrow...\dashrightarrow Z^{(r-1)}\dashrightarrow Z^{(r)}=Y$$
 and we can assume that $X^{(i)} \dashrightarrow X^{(i+1)}$ is a La Nave's flip (resp. divisorial contraction of an elliptic component, divisorial contraction of a pseudoelliptic component, small contraction) if and only if
 $Z^{(i)} \dashrightarrow Z^{(i+1)}$ is a La Nave's flip (resp. divisorial contraction of an elliptic component, divisorial contraction of a pseudoelliptic component, small contraction). 
\end{Teo}
Observe that for $(2)$, from the definition of refined numerical data, each irreducible component of $\supp(sS'+\vec{a}F'+\vec{b}G')$ (resp. $\supp(sS''+\vec{a}F''+\vec{b}G')$) is $\mathbb{Q}$-Cartier. Moreover, the refined numerical data of $(X',sS'+\vec{a}F')$ is the same as the one of $(X'',sS''+\vec{a}F'')$.
\begin{proof}We first reduce (1) to proving (2), and then we prove (2).

(1): 
From Theorem \ref{Teo:descripton:bir:trans:to:get:lcdeg}, up to replacing $\spec(R)$ with an \'etale cover of it,
we can add to $sS'+\vec{a}F'$ a $\mathbb{Q}$-divisor $G'$ such that all its irreducible components
$\mathbb{Q}$-Cartier and $(X',sS'+\vec{a}F'+G')$ is a stable degeneration. Moreover, let $\{Y_j'\}_j$ be the irreducible components of $X'_p$, and let $C_j':=Y_j' \cap S'$. To make $G'$ more canonical, we can choose it as follows:
\begin{itemize}
\item $G_p' \cap Y'_j$ is supported on some non-multiple fibers away from the double locus;
\item The weights on each irreducible component of $G'$ are $\frac{1}{12}$, and 
\item For each $Y'_j$
the we require that $((G'\cap Y'_j),C_j')=3$. 
\end{itemize}
In particular, this determines the numerical data of $(X',sS'+\vec{a}F'+G')$.
The choice $\frac{1}{12}$ is not essential, however we need to make sure that our pair is slc, see Corollary \ref{Cor:stability:condition:for:minimal:W:fibration}.
We define in a similar way $G''$, and observe that $(X',sS'+\vec{a}F'+G')$ and $(X'',sS''+\vec{a}F''+G'')$ have the same numerical data.
 Then from Lemma \ref{Lemma:numerical:data:gives:rnum:data}, this determines uniquely the \emph{refined} numerical data of $(X',sS'+\vec{a}F'+G')$ and $(X'',sS''+\vec{a}F''+G'')$.
But now (1) follows from (2).

(2): Let $(X^{(1)},sS^{(1)}+\vec{a}F^{(1)}+G^{(1)}):=(X',sS'+\vec{a}F+\vec{b}G)$, let $D^{(1)}:= sS^{(1)}+\vec{a}F^{(1)}+G^{(1)}$.

We proceed as in Section \ref{section:singularities:threefold}, lowering the weights on $G^{(1)}$. This produces
a sequence of birational transformations
as in Theorem \ref{Teo:descripton:bir:trans:to:get:lcdeg}:
 $$\xymatrix{\ar @{} [r]
 (X^{(1)},D^{(1)}) \ar@{.>}^{f^{(1)}}[r] &(X^{(2)},D^{(2)}) \ar@{.>}^-{f^{(2)}}[r]&\text{...} \ar@{.>}^-{f^{(m-2)}}[r]
 &(X^{(m-1)},D^{(m-1)}) \ar@{.>}^-{f^{(m-1)}}[r] & (X,sS+\vec{a}F)}$$
 where $(X^{(i)},D^{(i)})$ is a stable degeneration, and $(f^{(i)} \circ... \circ f^{(1)})_*( sS^{(1)}+\vec{a}F^{(1)}+G^{(1)})-D^{(i)}$ and
 $-(f^{(i)} \circ... \circ f^{(1)})_*( sS^{(1)}+\vec{a}F^{(1)})+D^{(i)}$ are effective (i.e. we reduce the weights on $G^{(i)}$).

 By Lemma \ref{Lemma:recovering:refined:numerical:data}, if we know the birational transformation $f^{(i)}$ and the refined numerical data of
 $(X^{(i)},D^{(i)})$, we know the refined numerical data of $(X^{(i+1)},D^{(i+1)})$ when it exists. Then
 it suffices to show that we can choose $f^{(i)}$ and $(X^{(i+1)},D^{(i+1)})$ using only the refined numerical data of $(X^{(i)},D^{(i)})$. 
 
 Let then $G^{(i)}:=D^{(i)}-(f^{(i)} \circ... \circ f^{(1)})_*( sS^{(1)}+\vec{a}F^{(1)})$ and let $L_{X^{(i)}}$ be the log-canonical divisor of
$(X^{(i)},D^{(i)})$. If
 for every $0 \le t \le 1$ both $(L_{X^{(i)}}-tG^{(i)},C_j^{(i)})$ and
$((L_{X^{(i)}}-tG^{(i)})_{|Y^{(i)}_j})^2$ remain non-negative for every $j$, from Theorem
 \ref{Teo:generalization:of:appendix} the divisor $K_{X^{(i)}}+D^{(i)}-G^{(i)}$ is nef.
 Then to get $(X,sS+\vec{a}F)$ we need to use log-abundance on
 $(X^{(i)},D^{(i)}-G^{(i)})$. But from Theorem \ref{Teo:generalization:of:appendix} \emph{we have a set of candidates for the possible non-positive curves}. We know that the exceptional locus is a union of irreducible components of $X^{(i)}_p$,
and irreducible components of $S^{(i)}_p$. To find the first ones we compute $((L_{X^{(i)}}-G^{(i)})_{|Y^{(i)}_j})^2$,
to find the others we compute $(L_{X^{(i)}}-G^{(i)},C_j^{(i)})$, for every $j$.

Otherwise $K_{X^{(i)}}+D^{(i)}-G^{(i)}$ is not nef. Then, up to reducing the weights on $G^{(i)}$ keeping $(L_{X^{(i)}},C_j^{(i)})$ and
$((L_{X^{(i)}})_{|Y^{(i)}_j})^2$ positive for every $j$, we can choose
an irreducible component $G_1^{(i)}$ of $\supp(G^{(i)})$ such that for a $0 < t <1$, either
$(L_{X^{(i)}}-tG_1^{(i)},C_j^{(i)})= 0$ or 
$((L_{X^{(i)}}-tG_1^{(i)})_{|Y^{(i)}_j})^2 = 0$.

 Then for $\epsilon$ small enough, to make $K_{X^{(i)}}+D^{(i)}-(t-\epsilon)G^{(i)}$ nef we need to perform a step of the MMP.
 From Theorem \ref{Teo:generalization:of:appendix} this is either a divisorial contraction or a flip of La Nave.
 We can control the divisorial
contractions computing $((L_{X^{(i)}}-(t+\epsilon)G^{(i)}_1)_{|Y^{(i)}_j})^2$. We can control La Nave's flips since we have a finite set of candidates
for isolated negative curves.
Namely, it is enough to compute $(L_{X^{(i)}}-(t+\epsilon)G^{(i)}_1,C_j^{(i)})$ for every $j$.
This produces the new threefold pair $(X^{(i+1)},D^{(i+1)})$ with the morphism
$f^{(i)}:(X^{(i)},D^{(i)}-(t+\epsilon)G_1^{(i)}) \dashrightarrow (X^{(i+1)},D^{(i+1)})$. \end{proof}

\subsection{Wall and chamber decomposition and $\mathbb{Q}$-Cartier walls}
This subsection is mainly devoted at proving Theorem \ref{Teo:wall:and:chamb:dec:stable:models}: we prove that we can divide the set of all admissible weight vectors into finitely many chambers,
where the stable models do not change. 

Recall that in Section \ref{Section:construction:of:the:moduli:surface:pairs} we picked $\Kg \to \bigcup_{n \le m \le d} \mathcal{K}_{g,m}(\overline{\mathcal{M}}_{1,1},d)$, the normalization of an atlas, and we defined
a family of surface pairs $(\sY',s\sS'+\vec{a}\sF') \to \Kg$: the bounded family of tsm limits (see Definition \ref{Def:bounded:family:tsm:limits}). With this notation, we have the following
\begin{Prop}\label{Prop:stratification:exists}There is a scheme of finite type $\sZ$ with a surjective (quasi-finite) morphism $\iota:\sZ \to \Kg$ satisfying the following. Given a connected component $\sZ_i$ of $\sZ$ and two points $p_1,p_2 \in \sZ_i$, the pairs $(\sY'_{\iota(p_1)},s\sS_{\iota(p_1)}'+\vec{a}\sF_{\iota(p_1)}')$ and 
$(\sY_{\iota(p_2)}',s\sS_{\iota(p_2)}'+\vec{a}\sF_{\iota(p_2)}') $ have the same numerical data.
\end{Prop}
In particular, there are finitely many numerical data for tsm limits coming from $\Kg$.
\begin{proof}
Over $\Kg$ we have the following objects:
\begin{enumerate}
 \item A family of nodal weighted curves $(\sS',\vec{a}\sF_{|\sS'}') \to \Kg$;
 \item A family of divisors $(K_{\sY'/\Kg})_{|\sS'}$ and
 \item A family of divisors $s\sS'_{|\sS'}$.
\end{enumerate}
Consider first a stratification $\sZ_1 \to \Kg$, with $\sZ_1$ of finite type, such that two points $p_1,p_2$ are in
the same connected component of $\sZ_1$ if and only if the curves $(\sS',\vec{a}\sF_{|\sS'}')\times_{\Kg}\spec(p_1)$ and $(\sS',\vec{a}\sF_{|\sS'}') \times_{\Kg}
\spec(p_2)$ have the same dual weighted graph. Take then a stratification $\sZ_2 \to \sZ_1$ such that the family of
curves $\sS'\times_U \sZ_2$ is equinormalizable (see \cite{Kollarequinorm}). Let $\sS^n$ be the family of curves that
simultaneously normalizes $\sS'\times_{\Kg} \sZ_2$, and let $\psi: \sS^n \to \sS'$ be the induced morphism. Then on
$\sS^n$ we have the divisors $D_1:=\psi^*((K_{\sY'/\Kg})_{|\sS'})$ and $D_2:=\psi^*\sS'_{|\sS'}$. Thus now $\sS^n\to \sZ_2$ is a flat family of
possibly not connected smooth curves.

A priori, even if we take a connected component $T$ of $\sS^n$, the corresponding morphism $T \to
\sZ_2$ will not be a family of connected curves, so we cannot yet distinguish the connected components of the fibers of $\sS^n \to \sZ_2$ using the
geometry of $\sS^n$. But from \cite{FGA}*{Corollary 8.2.18},
up to taking an étale cover of $\sZ_2$, we can assume that for every
connected component $W$ of $\sZ_2$ and
 for every
$q \in W$, there is a bijection between the connected components of $\sS^n_q$ and those of $\sS^n \times_{\sZ_2} W$.
Namely, up to replacing $\sZ_2$ with an étale cover of it, we can assume that each connected component $\sS_j^n$ of $\sS^n$ gives
a family of connected curves $\sS^n_j \to \sZ_2$. Then
using Riemann-Roch for curves on each connected component of $\sS^n$, and the theorems on cohomology and base change,
we see that for every connected component $\sS^n_j$ of $\sS^n$,
the maps $\sZ_2 \to \mathbb{Z}$ that send $z \mapsto \deg((D_1)_{|(\sS^n_j)_z}))$ and $z \mapsto \deg((D_2)_{|(\sS^n_j)_z}))$
are locally constant. 
Therefore we can find the desired morphism $\sZ \to \sZ_2 \to \Kg$.
\end{proof}
The main consequence of Proposition \ref{Prop:stratification:exists} is the following theorem:

\begin{Teo}\label{Teo:wall:and:chamb:dec:stable:models} There is a finite wall and chamber decomposition for the set of all admissible weighs, satisfying the following conditions.
Let $I':=(s',\vec{a}',\beta)$ and $I'':=(s'',\vec{a}'',\beta)$ be two vectors in the same chamber,
and let $(X',s'S'+\vec{a}'F')$ be a tsm limit with stable model $(Y,s'S+\vec{a}'F)$. Then 
the stable model of $(X',s''S'+\vec{a}''F')$ is $(Y,s''S+\vec{a}''F)$.
\end{Teo}
\begin{proof}
From how the stable limit is constructed, there is such a finite wall and chamber decomposition for a \underline{fixed}
tsm limit $(X',sS'+\vec{a}F')$. Indeed, from Theorem \ref{Teo:descripton:bir:trans:to:get:lcdeg},
there are finitely many possibilities for the possible special fibers of the stable model of
$(X',sS'+\vec{a}F')$, when $(s,\vec{a})$ varies.
Therefore there are finitely many surface pairs $(Y_i,D_i)$ satisfying the following condition. Up to changing the coefficients of the components of $D_i$,
any irreducible component of the special fiber of the stable model of $(X',tS'+\vec{b}A')$, when $a$ and $\vec{b}$ vary, is one of the $(Y_i,D_i)$.
Let
$L_{Y_i}$ be the lc divisor of $(Y_i,D_i)$. Since the coefficients of  $D_i$ depend on $(s,\vec{a})$, so does $L_{Y_i}$.
Let $C_i$ be the section component of $Y_i$ (we put $C_i=0$ if $Y_i$ is a pseudoelliptic component).
From Theorem \ref{Teo:descripton:bir:trans:to:get:lcdeg},
the stable model of $(X',sS'+\vec{a}F')$ changes when either we contract one of the $Y_i$ or one of the $C_i$.
Then we can declare our walls to be given by the equations $(L_{Y_i})^2=0$ and $(L_{Y_i},C_i)=0$, when these
are not identically 0. In particular there are finitely many of them.

Similarly, if we take a finite set of tsm limits $\{(X_i',sS_i'+\vec{a}F_i')_i\}$,
intersecting the wall and chambers decompositions given by each
$(X_i',sS_i'+\vec{a}F_i')$ gives a wall and chamber decomposition that holds for every $(X_i',sS_i'+\vec{a}F_i')$. We reduce
to this situation using Theorem \ref{Teo:stable:model:depends:only:on:num:data} and Proposition \ref{Prop:stratification:exists}.

From Proposition \ref{Prop:stratification:exists}, we can find a morphism
$\sZ \to \Kg$ of finite type and surjective, such that $\sZ$ has connected components $\{\sZ_i\}_{i=1}^m$, and for each $i$ and each $q_1, q_2 \in \sZ_i$,
the numerical data of $(\sX_{q_1},s'\sS'_{q_1}+\vec{a}\sF'_{q_1})$ is the same as the one of
$(\sX_{q_2},s'\sS'_{q_2}+\vec{a}\sF'_{q_2})$. 

Let then $\{p_1,...,p_m\}$ be closed points of $\sZ$, corresponding to $(Y_i,sS_i+\vec{a}F_i)$, such that $p_i \in \sZ_i$.
Let $\{(X_i,sS_i+\vec{a}F_i) \to \spec(R_i)\}$
be $m$ tsm limits such that $(X_i,sS_i+\vec{a}F_i)_p \cong (Y_i,sS_i+\vec{a}F_i)$. Then since $m$ is finite, there is a finite wall and chamber decomposition
for $\{(X_i,sS_i+\vec{a}F_i)\}_{i=1}^m$, and from Theorem \ref{Teo:stable:model:depends:only:on:num:data}
such a wall and chamber decomposition will work for any tsm limit.
\end{proof}

 Now, since the wall and chamber decomposition of Theorem \ref{Teo:wall:and:chamb:dec:stable:models} is finite,
 for every $(s,a_1,...,a_m,\beta)$ and for every $i$, there are finitely many walls
 that the line segments $$(1-t)(s,a_1,...,a_n,\beta)+t(s,a_1,...,a_{i-1},0,a_{i+1},...,a_n,\beta) \text{ with }0< t < 1$$ cross. Let $\{t_1^{(i)},...,t_{m_i}^{(i)}\}$ be such that the walls are at $t=t_j^{(i)}$ for $1 \le j \le m_i$ (if there are no walls, we define $m_i=1$ and $t^{(i)}_1:=\infty$).
 A similar conclusion holds
 for $$(1-t)(s,a_1,...,a_n,\beta)+t(0,a_1,...,a_n,\beta) \text{ with }0< t < 1$$
 and let $\{t_{1}^{(s)},...,t_{q_s}^{(s)}\}$ be such that these walls are at $t=t_i^{(s)}$ for $1 \le i \le q_s$.
\begin{Def}
 With the notation above, we
 define the \underline{$\mathbb{Q}$-Cartier threshold} for the weight data $I$ to be $w(I):=\operatorname{min}_{i,j,\ell}(t_j^{(i)},t_\ell^{(s)})$.\end{Def}
 The $\mathbb{Q}$-Cartier threshold is a positive number, which is at most the "distance" of $I$ from any wall we meet, decreasing any weight.
 Observe that
 $w(I)>0$ for every $I$.

\begin{Cor}\label{Corollary:Qcartier:threshold:guarantees:Qcartier}
 Let $I=(s,\vec{a},\beta)$ be an admissible weight vector.
 For every $0<\epsilon<w(I)$ and for every $j$, let
 $I':=(s-\epsilon,\vec{a},\beta)$ and $I_j:=(s,\vec{a}=(a_1,...,a_{j-1},a_j-\epsilon,a_{j+1},...,a_m),\beta)$.
 Then the universal divisors $\sS_{I'}$ and $(\sF_j)_{I_j}$ are $\mathbb{Q}$-Cartier for every $j$.
\end{Cor}
This is the main point where we use that $\cE_I$ is \emph{normal}, instead working with $\cE_I^{sn}$. First notice that to use the two Definition \ref{Def:stable:pairs:normal:base} and \ref{Def:KP:stable:pairs} interchangeably, we need to stick with working over normal bases. Moreover, to apply
\cite{kollarbook}*{Theorem 4.36} we need $\sS_{I'}$ to be a family of \emph{generically Cartier divisors}, which follows from normal bases from \cite{kollarbook}*{Theorem 4.26 and Theorem 4.2}.
\begin{proof}
We prove the case of $I'$, the
other cases can be proved in the same way.

We need to show that if $B \to \cE_{I'}$ is an étale atlas, where $B$ is a scheme, if $(X,(s-\epsilon) S+\vec{a}F) \to B$ is the
corresponding family of surface pairs, then $S$ is a $\mathbb{Q}$-Cartier divisor. From
\cite{kollarbook}*{Theorem 4.36} we can replace $X$ with a DVR $R$, and we can further assume that the generic point of $\spec(R)$ maps to
$\cE_I^\circ$. But then from Theorem \ref{Teo:wall:and:chamb:dec:stable:models} and from the definition of $w(I)$,
also $(X,(s-\epsilon') S+\vec{a}F)$ is a stable pair, for any $0 <\epsilon'<w(I)$. Thus both
$K_{X}+(s-\epsilon') S+\vec{a}F$ and $K_X+(s-\epsilon) S+\vec{a}F$ are $\mathbb{Q}$-Cartier, which implies that $S$ is $\mathbb{Q}$-Cartier.
\end{proof}From Theorem \ref{Teo:wall:and:chamb:dec:stable:models}, if two weight vectors $I$ and $I'$ are in the same chamber, the spaces $\cE_I$ and $\cE_{I'}$ parametrize the same surface pairs. Therefore it is reasonable to expect the following Corollary (see also Proposition \ref{Prop:iso:se:riduco:s} and \cite{AB3}*{Theorem 1.2}).
\begin{Cor}Assume that $I$ and $I'$ are in the same open chamber. Then $\cE_I \cong \cE_I'$.
\end{Cor}
\begin{proof}Let $I:=(s,(a_1,...,a_n),\beta)$ and let $I':=(t,(b_1,...,b_n),\beta)$. Since we are in an open chamber, up to changing a coefficient at the time, we can assume that the divisor $\sD:=t\sS_I+\vec{b}\sF_I-(s\sS_I+\vec{a}\sF_I)$ is effective. From Corollary \ref{Corollary:Qcartier:threshold:guarantees:Qcartier}, $\sD$ is $\mathbb{Q}$-Cartier. Composing the morphism $\omega_{\sX_I/\cE_I}^{\otimes m} \to \sL$ (see Definition \ref{Def:KP:stable:pairs}) with the inclusion $\sL \to \sL \otimes \cO_{\sX_I}(m\sD)$ for $m$ divisible enough, gives a morphism $\phi:\omega_{\sX_I/\cE_I}^{\otimes m} \to \sL \otimes \cO_{\sX_I}(m\sD)$. From Theorem \ref{Teo:wall:and:chamb:dec:stable:models}, the family $(\sX_I, \phi) \to \cE_I$ is a family of stable pairs (Definition \ref{Def:KP:stable:pairs}). Then it induces a morphism $\cE_I \to \cE_{I'}^{sn}$, and from \cite{AB3}*{Lemma A.5 (3)} a morphism $\cE_I \to \cE_{I'}$.
This morphism restricts to an isomorphism $\cE_I^{\circ} \to \cE_{I'}^{\circ}$, it is quasi-finite from Theorem \ref{Teo:wall:and:chamb:dec:stable:models}, representable and proper. Then it is an isomorphism from Proposition \ref{proposition:iso:dm:stacks}.
\end{proof}
\section{Cohomology vanishing and wall-crossing morphisms}\label{Section:cohom:vanishing:and:w:crossing}
We begin by outlining the strategy we follow for proving that there are wall-crossing morphisms. We emphasize what are the main ideas, and how they are guaranteed in our case.

\begin{bf}The set-up:\end{bf} For every admissible weight vector $I$ we have two seminormal (in our case, normal) moduli spaces for stable surface pairs, namely $\cE_I^\circ$ and $\cE_I$; with a dense open embedding $\cE_I^\circ \to \cE_I$. 

When we decrease the weights on the divisor to go from $I$ to $I'$, we have a reduction morphism $r_{I,I'}:\cE_I^\circ \to \cE_{I'}^\circ$. Assume that $I$ parametrizes surface pairs $(X,sS+\sum a_i F_i)$,
$I'$ parametrizes surface pairs $(X,tS+\sum b_i F_i)$, and $\pi:\sX_I \to \cE_I$ is the universal family of surfaces. For $d$ divisible enough, the morphism $r_{I,I'}$ is induced by $\operatorname{Proj}(\bigoplus_n \pi_*(\cO_{\sX^\circ_I}( nd(K_{\sX^\circ_I/\cE^\circ_I}+t\sS^\circ_I+ \vec{b}\sF^\circ_I)))).$
These morphisms give a finite wall-and-chamber decomposition for the interior of $\cE_I$, i.e. for the moduli spaces $\cE_I^\circ$.
Our goal is to extend $r_{I,I'}$ to get $R_{I,I'}:\cE_I \to \cE_{I'}$, as in the introduction.

\begin{bf}Step 1: We check a necessary condition.\end{bf} In the previous section we proved a necessary condition for having a \emph{finite} wall and chamber decomposition. Namely, we showed that for every $I$, there is a positive number (the $\mathbb{Q}$-Cartier threshold $w(I)$) satisfying the following. Take any two admissible vectors $I_1$ and $I_2$ different from $I$ and obtained from $I$ reducing the coefficient on a marked divisor by less than $w(I)$. Then we can obtain the surfaces of $\cE_{I_2}$ from those parametrized by $\cE_{I_1}$, simply by adjusting the coefficients on the marked divisor.

\begin{bf}Step 2: We check a $\mathbb{Q}$-Cartier condition. \end{bf}We ensure that, if we are in an open chamber, the divisor we want to reduce the weights of is $\mathbb{Q}$-Cartier (see Corollary \ref{Corollary:Qcartier:threshold:guarantees:Qcartier}).

\begin{bf}Step 3: From an open chamber, we reach a wall.\end{bf} We show, by a cohomology vanishing, that if we reduce the weights until when the log-canonical divisor is no longer ample but is still nef, the log-plurigenera commutes with base change (Theorem \ref{Teo:cohom:vanishing}, see also
\cite{kollar2018logs} and \cite{kollar2018log}). This gives a morphism from an open chamber to a wall.

\begin{bf}Step 4: From a wall, we reach an open chamber decreasing the weights.\end{bf} A priori, once we reach a wall, we cannot simply reduce the weights on the divisor to get a reduction morphism. In fact, the divisor we would like to reduce the weights of might not be $\mathbb{Q}$-Cartier: we need to proceed differently. In this case, we show that Proposition \ref{proposition:iso:dm:stacks} applies.
\\

We now prove the cohomology vanishing mentioned in Step 3 above. See also
\cite{kollar2018logs} and \cite{kollar2018log} for similar results.
\begin{Teo}\label{Teo:cohom:vanishing}
 Let $\spec(R)$ be a DVR, with generic (resp. closed) point $\eta$ (resp. $p$). Let $(X,D) \to \spec(R)$ be a morphism, with $(X,D)$ lc and
 $(X_p,D_p)$ slc. If $(K_X+D)$ is nef and $(K_X+D)_{|X_\eta}$ is log-big, then for $m$ divisible enough, $H^i(\cO_{X_p}(m(K_{X_p}+D_p)))=0$ for $i>0$.
\end{Teo}
\begin{proof}
From \cite{kollarbook}*{Proposition 2.13}, the lc centers of $(X,D)$ intersect the generic fiber. Then $K_X+D$ is nef and log-big, so
from \cite{Fuj-vanishing}*{Theorem 1.10} (see also \cite{Kollarsing}*{Theorem 10.37}) we have $R^if_*(\cO_X(m(K_X+D)))=0$
for every $m$ divisible enough and for $i>0$. But since $(K_X)_{p}=K_{X_p}$,
from cohomology and base change, also
$H^i(\cO_{X_p}(m(K_{X_p}+D_p)))=0$.
\end{proof}

We will use the following two propositions for the case in which the divisor we would like to reduce the weighs of is not $\mathbb{Q}$-Cartier (see Step 4).

\begin{Prop}\label{proposition:iso:dm:stacks}
 Let $f:\sX_1 \to \sX_2$ be a representable, proper morphism of seminormal DM stacks (of finite type over an algebraically closed field $k$
 of characteristic 0).
 Assume that the morphism of sets $|\sX_1(\spec(k))| \to |\sX_2(\spec(k))|$ has finite non-empty fibers. Assume one of the following:
 \begin{enumerate}
  \item $\sX_1(\spec(k)) \to \sX_2(\spec(k))$ is an equivalence of groupoids, or
  \item $\sX_1$ and $\sX_2$ are normal, there
  is an open dense substack $U_2 \to \sX_2$ such that $U_1:=\sX_1 \times_{\sX_2} U_2 \to U_2$
  is an isomorphism, and $U_1$ is dense in $\sX_1$.
 \end{enumerate}

 Then $f$ is an isomorphism. 
\end{Prop}
One can understand Proposition \ref{proposition:iso:dm:stacks} as an analogue of the Zariski main theorem for representable morphisms, see \cite{LMB}*{Theorem 16.5} and \cite{Picard}*{Theorem A.5} for similar results. 
\begin{proof}
 Let $V_2 \to \sX_2$ be an \'etale atlas which is a scheme, let $V_1:=\sX_1 \times_{\sX_2}V_2$ and let $\psi:V_1 \to V_2$ be the second projection.
 Since $f$ is representable, $V_1$ is an algebraic space.
 Since $f$ is proper, also $\psi$ is proper, then from \cite{Olsson}*{Theorem 7.2.10} we see that $V_2$ is a scheme.
 It is enough to show that $\psi$ is an isomorphism.
 
  \underline{Assuming (1):} For every morphism $\spec(k) \to V_2$, observe that $\spec(k) \times_{V_2} V_1 \cong \spec(k) \times_{\sX_2}\sX_1$.
  From the definition of fibered product of fibered categories (\cite{Olsson}), for every
 morphism $\spec(k) \to V_2$, there is an isomorphism $\spec(k) \cong \spec(k)\times_{V_2}V_1$.
 So now the situation is the following. We have a proper quasi-finite morphism
 $\psi:V_1 \to V_2$ between two seminormal schemes (of finite type over $k$, with $k=\overline{k}$ and of characteristic 0),
 and we know that $\psi$ is bijective on $k$-points. We want to show that $\psi$ is an isomorphism.
 
 First notice that $\psi$ is finite (so in particular affine), since it is proper and quasi-finite.
 Since $\psi$ is proper, it is closed. But a closed bijective morphism between two topological spaces is an homeomorphism, so
 $V_1$ and $V_2$ are homeomorphic.
 Therefore we have a proper morphism, which is an
 homeomorphism, between two seminormal schemes of finite type over an algebraically closed field of characteristic 0: it is an isomorphism.
 
 \underline{Assuming (2):} First we show that $\psi$ is finite. Consider a point $p:\spec(k)\to \sX_1$, and let $q:=f(p)$. From the definition of fibred product of categories fibred in groupoids (\cite{Olsson}*{Section 3.4.9.}), we have an inclusion of sets 
 $$\spec(k) \times_{\sX_2}\sX_1\subseteq \{(a,\sigma): a \in \sX_1(\spec(k)) \text{ such that }f(a)\cong q; \text{ }\sigma \in \operatorname{Hom}_{\sX_2(\spec(k))}(f(a),q)\}.$$
 Since $|\sX_2(\spec(k))| \to |\sX_1(\spec(k))|$ has finite fibers and since the objects of $\sX_i(\spec(k))$ have finite automophisms, $\spec(k) \times_{\sX_2}\sX_1$ is finite. Then notice that
 $\spec(k)\times_{V_2}V_1 \cong \spec(k) \times_{\sX_2}\sX_1$, so the morphism $\psi$ is quasi-finite. Since it is proper, it is finite.
 
 Consider $Z \to V_2$ a connected component, let $T:=Z \times_{V_2}V_1$ and let $g:T \to Z$ be the corresponding map. To prove the desired result
 is enough to show that $g$ is an isomorphism.
 
 We show first that $T$ has a single irreducible component. Since it is normal, it is enough to show that it is connected.
 Since $U_i$ is dense in $\sX_i$, for every connected component $T_i$ of $T$, the open subset $U_1 \times_{\sX_1}T_i$ is
 non-empty in $T_i$. So $U_1 \times_{\sX_1}T_i$ is dense in $T_i$, and in particular there is a bijection
 between the connected components of $U_1\times_{\sX_1}T$ and those of $T$. The same reasoning applies to $Z$, so the open subset $Z \times_{\sX_2}U_2$ is
 dense in $Z$.
 Since $U_1 \xrightarrow{f_{|U_1}} U_2$ is an isomorphism, also its pull-back
 $U_1 \times_{\sX_1} T \to U_2\times_{\sX_2} Z$ is an isomorphism. But then:
 \begin{align*}
  1=\# \text{(connected components of }Z\times_{\sX_2}U_2)&=\# \text{(connected components of }U_1 \times_{\sX_1} T)=\\
  &=\# \text{(connected components of }T)
 \end{align*}
Then $g$ is a birational finite morphism, and $T$ and $V$ are normal varieties: $g$ is an isomorphism.
 \end{proof}
 
 The main application of Proposition \ref{proposition:iso:dm:stacks} is the following proposition (see also Step 4 above):
\begin{Prop}\label{Prop:iso:se:riduco:s}
 Let $I:=(s,\vec{a}=(a_1,...,a_m),\beta)$ be an admissible weight vector. Then for every $0<\epsilon<w(I)$ and for every $j$,
 $I':=(s-\epsilon,\vec{a},\beta)$ and $I_j:=(s,\vec{a}=(a_1,...,a_{j-1},a_j-\epsilon,a_{j+1},...,a_m),\beta)$
 are such that $\mathcal{W}_{I} \cong \cE_{I'}\cong \cE_{I_j}$.
\end{Prop}
Before proceeding with the proof of Proposition \ref{Prop:iso:se:riduco:s}, we remark the following
\begin{Oss}\label{Oss:automorphism}
Let $(X,D)$ be a stable slc surface. Let $D'$ be an effective $\mathbb{Q}$-Cartier $\mathbb{Q}$-divisor with $\supp(D') \subseteq \supp(D)$, such that $K_{X}+D+ D'$ is nef, and $(X,D+D')$ is slc. Let $(Y,D_Y)$ be the stable model of $(X,D+D')$, and let $p:X \to Y$ the morphism induced by taking the stable model. Then $p$ does not contract any irreducible component of $X$. In particular, let $\phi_X$ be automorphism  of $(X,D)$, let $\phi_Y$ be an automorphism of $(Y,D)$, and assume that $p \circ \phi_X=\phi_Y \circ p$. Then $\phi_X=\Id$ implies $\phi_Y=\Id$.  
\end{Oss}
\begin{proof}
To show that $p$ does not contract any irreducible component of $Y'$, take an ample hyperplane section $H$, not contained in $\supp(D)$, of each irreducible component of $X$. Then $0<(K_X+D).H\le (K_X+D+D').H$: the divisor $H$ does not get contracted.

For the part on the automorphisms, observe first that $p$ is an isomorphism on an open dense subset $U$ of $X$, and $p(U)$ is dense in $Y$. Then if $\phi_X=\Id$ there is an open dense subset (namely $p(U)$) where $\phi_Y$ and $\Id$ agree. Therefore $\phi_Y=\Id$.\end{proof}
\begin{proof}[Proof of Proposition \ref{Prop:iso:se:riduco:s}] We first tackle the case of $I'$.
We construct a morphism $\Phi:\cE_{I'} \to \cE_I$, and using Proposition \ref{proposition:iso:dm:stacks} we show that it is an isomorphism.
To produce such a morphism, we use the universal property of the moduli space constructed in \cite{KP}. In particular, we
construct $\sY \to \cE_{I'}$, a family of slc surfaces, and a relatively very ample line bundle $\sL$ over $\sY$, with a morphism
$\omega_{\sY/\cE_{I'}}^{\otimes r} \to \sL$ satisfying the assumptions of Definition \ref{Def:KP:stable:pairs}.

\begin{bf}Step 1: construction of $\Phi$.\end{bf}
 We start by constructing the family of surfaces $\sY \to \cE_{I'}$. To make this step less notation-heavy, we drop the subscript $I'$ on $\sX_{I'}$, $\sS_{I'}$ and $\sF_{I'}$. This should cause no confusion. Let $\epsilon<w(I)$. From Corollary \ref{Corollary:Qcartier:threshold:guarantees:Qcartier}, the divisor
 $\sS$ is $\mathbb{Q}$-Cartier.
 Let $\pi:(\sX,(s-\epsilon)\sS+\vec{a}\sF) \to \cE_{I'}$ be the universal family, and consider the $\mathbb{Q}$-Cartier divisor
 $\sD:=s\sS+\vec{a}\sF$.
 
 Let $R$ be a DVR, with generic point $\eta$ and closed one $p$, and let $(X,(s-\epsilon)S+\vec{a}F)$ be a stable degeneration over $\spec(R)$.
 By definition of $w(I)$, for every $\epsilon$ small enough the pair
 $(X,(s-\epsilon)S+\vec{a}F)$ is lc, so also for $\epsilon=0$ it is lc. Moreover, since
 for every $w(I)>\epsilon>0$, the pair $(X,(s-\epsilon)S+\vec{a}F)$ is a stable pair, and since the nef cone is closed, $K_X+sS+\vec{a}F$ is nef.
 Let then $(X,sS+\vec{a}F):=(X,D)$.
 Observe that 
 $(K_X+D)_{|X_{\eta}}$ is ample, so $(K_X+D)_{|X_{\eta}}$ is log-big.
 
 Therefore the hypothesis of Theorem \ref{Teo:cohom:vanishing} apply, and $H^1(m(K_{X_p}+D_{X_p}))=0$ for $m$ divisible enough.
 Thus for $m$ divisible enough $\operatorname{Proj}(\bigoplus_{n \in \mathbb{N}}\pi_*(\cO_{\sX}(mn(K_{\sX/\cE_{I'}}+\sD)))$
 commutes with base change, and gives a family of surfaces $\xi:\sY \to \cE_{I'}$ with a map $f:\sX \to \sY$. If we take a morphism $\spec(R) \to \cE_{I'}$ which sends the generic point to $\cE_{I'}^\circ$, we can pull back $f$ to get $f_R:\sX_R \to \sY_R$. Then the morphism $f_R$ is obtained taking the stable model of $(\sX_R,s\sS_R+\vec{a}\sF_R)$. In particular, since $\eta \mapsto \cE_{I'}^\circ$ and from Observation \ref{Oss:automorphism}, the exceptional locus of $f_R$ has codimension at least 2.
 
 We construct now $\omega_{\sY/\cE_{I'}}^{\otimes r} \to \sL$. Let $m$ be such that $\sG:=\cO_{\sX_{I'}}(m(K_{\sX/\cE_{I'}}+\sD))$ is Cartier. By the definition of the moduli psudofunctor of
 \cite{KP} (see also Section \ref{Section:background:tsm:and:stable:pairs}), there is a morphism
 $\omega_{\sX/\cE_{I'}}^{\otimes a} \to \cO_{\sX}(a(K_{\sX/\cE_{I'}}+(s-\epsilon)\sS+\vec{a}\sF))$ for $a$ divisible enough. This induces $\omega_{\sX/\cE_{I'}}^{\otimes am} \to \cO_{\sX}(am(K_{\sX/\cE_{I'}}+(s-\epsilon)\sS+\vec{a}\sF))$, and composing it with the inclusion $\cO_{\sX}(am(K_{\sX/\cE_{I'}}+(s-\epsilon)\sS+\vec{a}\sF)) \to \sG^{\otimes a}$ gives $\omega_{\sX/\cE_{I'}}^{\otimes am}\to \sG^{\otimes a}$ and
 $f_*(\omega_{\sX/\cE_{I'}}^{\otimes am})\to f_*(\sG^{\otimes a})$. Moreover, from the definition of $\sY$, for $a$ divisible enough $f_*(\sG^{\otimes a})$ is a line bundle.
 
  Now, from the explicit description of $\sX_p
 \to \sY_p$ for every $p$, there is an open subset $U \subseteq \sY$ of codimension 2 such that $f^{-1}(U) \xrightarrow{f} U$ is an isomorphism and $U \to \cE_I$ is Gorenstein. Let $j:f^{-1}(U) \to \sX$ be the inclusion. Then the restriction morphism $\omega_{\sX/\cE_{I'}}^{\otimes am} \to j_*(\omega_{f^{-1}(U)/\cE_{I'}}^{\otimes am})$ can be pushed forward to get a morphism $f_*(\omega_{\sX/\cE_{I'}}^{\otimes am}) \to (f\circ j)_*(\omega_{f^{-1}(U)/\cE_{I'}}^{\otimes am}) \cong \omega_{U/\cE_{I'}}^{\otimes am} \cong (\omega_{\sY/\cE_{I'}}^{\otimes am})_{|U}$. Observe that the sheaves $f_*(\omega_{\sX/\cE_{I'}}^{\otimes am})$ and $\omega_{\sY/\cE_{I'}}^{\otimes am}$ agree in codimension 2, therefore
$$f_*(\omega_{\sX/\cE_{I'}}^{\otimes am})^{**} \cong \omega_{\sY/\cE_{I'}}^{[am]}.$$But $f_*(\sG^{\otimes a})$ is a line bundle, so the morphism $f_*(\omega_{\sX/\cE_{I'}}^{\otimes am})\to f_*(\sG^{\otimes a})$ factors as
 $f_*(\omega_{\sX/\cE_{I'}}^{\otimes am})\to
 f_*(\omega_{\sX/\cE_{I'}}^{\otimes am})^{**} \to f_*(\sG^{\otimes a})$. Thus, composing the canonical morphism $\omega_{\sY/\cE_{I'}}^{\otimes am} \to \omega_{\sY/\cE_{I'}}^{[am]}$ with the isomorphism above, we get
$\omega_{\sY/\cE_{I'}}^{\otimes am} \to f_*(\sG^{\otimes a})$.
 Now, $f_*(\sG^{\otimes a})$ is an ample line bundle, and if we choose $b$ divisible enough, $f_*(\sG^{\otimes a})^{\otimes b}$ is
 very ample: we can take $\sL:=f_*(\sG^{\otimes a})^{\otimes b}$.

 Let $r:=amb$ and let $\alpha:\omega_{\sY/\cE_{I'}}^{\otimes r} \to \sL$ be the morphism we just constructed.
To check that for each $p \in \cE_{I'}(\spec(k))$, the morphism $\alpha_{|\sY_p}:\omega_{\sY_p}^{\otimes r} \to \sL_y$ satisfies the required properties of Definition \ref{Def:KP:stable:pairs}, we first
choose a DVR $R$, with generic point (resp. closed point) $\eta$ (resp. $p$), and with a morphism $\spec(R) \to \cE_{I'}$. We require that $\eta \mapsto \cE_{I'}^\circ$ and $q \mapsto p$. Then we first pull back $\alpha$ to $\spec(R)$, and then to $p$. But now let $\sY_R:=\sY \times_{\cE_{I'}}\spec(R)$ and similarly
$\sX_R:=\sX \times_{\cE_{I'}}\spec(R)$; let $f_R:\sX_R \to \sY_R$ be the induced morphism and let
$\alpha_R:\omega_{\sY_R/\spec(R)}^{\otimes r} \to \sL_{\sY_R}$ be the morphism induced by the pull back of $\alpha$ (induced as in \cite{KP}*{Definition 5.6}). Notice that the construction of Section \ref{section:singularities:threefold} give us a particular choice of $\beta:\omega_{\sY_R/\spec(R)}^{\otimes r} \to \sL_{\sY_R}$, which satisfies the assumptions of Definition \ref{Def:KP:stable:pairs}.
Moreover, since $f_R$ is an isomorphism in codimension 2, $\alpha_R$ and $\beta$ agree on an open subset of codimension 2. Finally, $\alpha_R$ and $\beta$ are uniquely determined by their induced morphisms $\alpha':\omega_{\sY_R/\spec(R)}^{[r]} \to \sL_{\sY_R}$ and $\beta':\omega_{\sY_R/\spec(R)}^{[r]} \to \sL_{\sY_R}$. But now all the sheaves are reflexive, and since $\alpha'$ and $\beta'$ agree in codimension 2 they agree everywhere.
Therefore the morphism $\omega_{\sY/\cE_{I'}}^{\otimes r} \to \sL$ satisfies the requirements of Definition \ref{Def:KP:stable:pairs}.

Recall finally that, to distinguish the fibers in $\sF$, we added $n$ sections $\sigma_i:\cE_{I'} \to \sX$. Composing these with $f$ gives $n$ sections $\cE_{I'} \to \sY$. This data induces we a morphism $\psi:\cE_{I'} \to \cE_{I}^{sn}$, it factors through the normalization $\cE_{I} \to \cE_I^{sn}$ (see \cite{AB3}*{Lemma A.5 (3)}), and gives $\Phi:\cE_{I'} \to \cE_{I}$.
 
 \begin{bf} Step 2: $\Phi$ is an isomorphism. \end{bf} We check that Proposition \ref{proposition:iso:dm:stacks} applies. We need to check that:
 (1) $\Phi$ is proper; (2) $\Phi$ is an isomorphism when restricted to an appropriate open substack of
 $\cE_{I}$; (3) $\Phi$ is representable, and (4) $\Phi(\spec(k))$ is surjective with finite fibers. Let $p \in \cE_{I'}(\spec(k))$ be a point corresponding to
 $(Y',(s-\epsilon)S'+\vec{a}F')$, and assume $\Phi(p)$ corresponds to $(Y,sS+\vec{a}F)$.

 $(1)$ follows since $\cE_{I'}$ is proper, whereas for $(2)$ we can take $\cE_{I}^\circ$.
To check $(3)$, it is enough to show that
the morphism $\Phi_p:\Aut_{\cE_{I'}}(p) \xrightarrow{ }
\Aut_{\cE_{I}}(\Phi(p))$ is injective. This follows from Observation \ref{Oss:automorphism}.
$(4)$: since $\cE_{I'}$ is proper, $\Phi$ is closed. Since $\Phi_{|\cE_{I'}^\circ}:\cE_{I'}^\circ \to \cE_{I}^\circ$
is an isomorphism, and $\cE^\circ_{I}$ is dense,
$\Phi$ is dominant. Thus
$\Phi$ is surjective, we need to check that it has finite fibers.

Since the auxiliary sections we introduced to define $\cE_I$ are a finite set of points in $Y$ and $Y'$ supported on the finite set of points $\supp(S) \cap \supp(\vec{a}F)$ and
$\supp(S') \cap \supp(\vec{a}F')$, to show that $\Phi$ has \emph{finite} fibers, we can ignore them.

To get $(Y,sS+\vec{a}F)$
it is enough to contract some components of $S'$,
without contracting any irreducible component of $Y'$. Therefore to get $p$ from $\Phi(p)$ it is enough
to perform a sequence of blow-ups to reintroduce the section-components contracted. Our goal is to show that the ideal sheaves we blow-up are uniquely
determined. This follows from Observation \ref{Oss:section:uniquely:determined:on:psudo}. 

The case $I_j$ is similar as above, except for the proof of $(4)$. For the proof of $(4)$, we need to show that
if $Y' \to Y$ is the contraction of some intermediate components of some intermediate fibers or pseudofibers, then we can
perform a sequence of blow-ups to reintroduce the intermediate components contracted.
The blow-ups we perform are along points on which $S$ is not $\mathbb{Q}$-Cartier.  This can be done as follows. First, proceeding as above, we can reintroduce the sections on each pseudoelliptic component on which $Y' \to Y$ is
not an isomorphism, to get a surface $Z$. Then from \cite{LaNave}*{Lemma 7.1.6}, if we perform a flip of La Nave on a degeneration having $Z$ as closed fiber, the self intersection of the intermediate component introduced by the flip is uniquely determined by the self intersection of the contracted section component. From Proposition \ref{Prop:recognize:the:intermediate:from:int:pairing}, this determines uniquely the intermediate component.\end{proof}

Before proving Theorem \ref{Teo:intro:extension:morphism}, it is convenient to adopt the following \begin{Notation}\label{Notation:ai:zero}It is convenient to generalize Definition \ref{Def:minimal:W:fib} allowing some of the $a_i$ to be 0. In this case, we do \emph{not} consider the corresponding fibers $F_{i}$ part of the data. For example, if $a_{i}=0$ for $r<i\le n$, we consider the pairs
$(X,sS+a_1F_1+...+a_nF_n)$ and $(X,sS+a_1F_1+...+a_rF_r)$ to be the same. Similarly, we consider the moduli spaces $\cE_{(s,(a_{1},...,a_n),\beta)}$ and $\cE_{(s,(a_1,...,a_r),\beta)}$ to be the same.
\end{Notation}

Now, given $I_1:=(s_1,\vec{a}_1,\beta) \le I_2:=(s_2,\vec{a}_2,\beta)$
two admissible weights, there are morphisms $r_{I_2,I_1}:\cE_{I_2}^\circ \to \cE_{I_1}^\circ$,
which on closed points can be described sending $(X,s_2S+\vec{a}_2A)\mapsto (X,s_1S+\vec{a}_1A)$. 
The main result of the section is the following (see Theorem \ref{Teo:intro:extension:morphism} and \cite{AB3}*{Theorem 1.5} if $s=1$):
\begin{Teo}\label{Teo:wall:crossings:morphisms}
There are morphisms $R_{I_2,I_1}:\cE_{I_2}\to \cE_{I_1}$ which extend $r_{I_2,I_1}$.
\end{Teo}
In particular, using Notation \ref{Notation:ai:zero}, there is a forgetful morphism
$\cE_{(s,(a_1,...,a_n),\beta)} \to \cE_{(s,(0),\beta)}$ when the weight vector $(s,(0),\beta)$ is admissible.
\begin{proof}
The proof follows closely \cite{Hassett}*{Theorem 4.1}.
Let $I_3:=(s_1,\vec{a}_2,\beta)$.
It suffices to prove that there are morphisms $\cE_{I_2} \to \cE_{I_3}$ and $\cE_{I_3} \to \cE_{I_1}$, which
extend $r_{I_2,I_3}$ and $r_{I_3,I_1}$. Namely, if we can prove the result in the cases $\vec{a}_1=\vec{a}_2$, and $s_1=s_2$; we can prove the result
in general.
We tackle the case $\vec{a}:=\vec{a}_1=\vec{a}_2$, the other case is analogous.

Consider then $I(t):=((1-t)s_2 + ts_1,\vec{a},\beta)$.
Up to replacing $s_2$ with $s_2-\epsilon$ for $\epsilon <w(I)$,
and using Proposition \ref{Prop:iso:se:riduco:s} and Corollary \ref{Corollary:Qcartier:threshold:guarantees:Qcartier}, we can assume that $\sS_{I_2}$
is $\mathbb{Q}$-Cartier. From
Theorem \ref{Teo:wall:and:chamb:dec:stable:models}, there are finitely many $t$ such that $I(t)$ is on a wall. In particular, since
$w(I)>0$, there is a positive $t_1$ such that for
$0\le t<t_1$, the divisor $K_{\sX_{I_2}/\cE_{I_2}}+((1-t)s_2+ts_1)\sS_{I_2}+\vec{a}\sF_{I_2}$ is ample relatively to $\cE_{I_2}$, but when
$t=t_1$ it is only nef. Consider then, for $d$ divisible enough,

$$\sY:=\operatorname{Proj}(\bigoplus _{m=1}^\infty \pi_*\cO_{\sX_{I_2}}( md(((1-t_1)s_2+t_1s_2)\sS_{I_2}+\vec{a}\sF_{I_2}))) \to \cE_{I_2}$$
From Theorem \ref{Teo:cohom:vanishing} and the theorems on cohomology and base change, $\sY \to \cE_{I_2}$ is a family
of surfaces, and there is a morphism $f:\sX_{I_2} \to \sY$. We proceed as in step 1 of the proof of Proposition \ref{Prop:iso:se:riduco:s},
to produce a line bundle $\sL$ on $\sY$, and a morphism
$\omega_{\sY/\cE_{I_2}}^{\otimes r} \to \sL$ to get a morphism $\cE_{I_2} \to \cE_{I(t_1)}$. Up to replacing $t_1$ with $t_1+\epsilon$ and from Proposition
\ref{Prop:iso:se:riduco:s}, we can assume that $\sS_{I(t_1)}$ is $\mathbb{Q}$-Cartier.

Then we repeat the procedure above, replacing $s_2$ with $(1-t_1)s_2+t_1s_1$. If we keep iterating, since
there are finitely many walls from Theorem \ref{Teo:wall:and:chamb:dec:stable:models}, in finitely many steps we get to $s_1$.
\end{proof}
\begin{Cor}\label{Cor:image:determided:by:rnd}
Let $I_1:=(s_1,\vec{a}_1,\beta)$ and $I_2:=(s_2,\vec{a}_2,\beta)$ be two admissible weight vectors, and assume that $I_1 \le I_2$. Assume also that $I_2$ is in an open chamber. Let $p \in \cE_{I_2}(\spec(k))$ corresponding to $(X,s_2S+\vec{a}_2F)$. Then $R_{I_2,I_1}(p)$ is uniquely determined by the refined numerical data of $(X,s_2S+\vec{a}_2F)$.
\end{Cor}
\begin{proof}
Choose a degeneration $(\cX,s_2\cS+\vec{a}_2\cF) \to \spec(R)$  with special fiber $(X,s_2S+\vec{a}_2F) \to \spec(k)$. From Theorem \ref{Teo:stable:model:depends:only:on:num:data}, the stable model of
$(\cX,s_1\cS+\vec{a}_1\cF)$ depends only on the refined numerical data of $(\cX,s_2\cS+\vec{a}_2\cF)$ (which is the refined numerical data of $(X,s_2S+\vec{a}_2F)$). To prove the desired result it suffices to notice that the following square commutes:
$$\xymatrix{ \cE_{I_2}(\spec(R))  \ar[d]_-{R_{I_2,I_1}(\spec(R))} \ar[r]& \cE_{I_2}(\spec(k)) \ar[d]^{R_{I_2,I_1}(\spec(k))} \\ \cE_{I_1}(\spec(R)) \ar[r]& \cE_{I_1}(\spec(k))}$$\end{proof}
\section{Universal curve and remarkable chambers}\label{Section:universal:curve}
By definition, a Weierstrass fibration $X$ comes with a surjective morphism $X \to C$ to a curve.
In particular, any surface pair $(X,sS+\vec{a}F)$ corresponding to a closed point of $\cE_I^\circ$ comes with a morphism $X \to C$, and $S$ is a section of it.
In fact, in \cite{AB3} the surfaces
the two authors parametrize admit a morphism to a curve, and there is a universal curve over the moduli space they construct.
The first goal of this section is to show that, also if $s <1$, there is such an universal curve. We first construct an auxiliary parameter space $\widetilde{\cE_I}$ as in \cite{AB3}, which comes with an universal curve by definition. Then
we use Proposition \ref{proposition:iso:dm:stacks} to show that $\widetilde{\cE_I}\cong \cE_I$.

After that, we prove Theorem \ref{Teo:intro:have:flat:fibration}. We show that, given any admissible weight
vector $I=(s,\vec{a},\beta)$, we can find $s'$ satisfying the following condition. For any point $p \in \cE_{(s',\vec{a},\beta)}(\spec(k))$, the corresponding surface pair has no pseudoelliptic components.

\subsection{The universal curve $\sC \to \cE_I$}

Let $\sX_I^{sn} \to \cE_I^{sn}$ the universal surface, and let $\sS_I^{sn} \to \cE_I^{sn}$ be the universal section (see Notation \ref{Notation:def:SI}).
Let $\mathfrak{M}_g$ be the algebraic stack of prestable curves of genus $g$, with universal family $\mathfrak{C} \to \mathfrak{M}_g$.
Let us denote with $\sH$ the following stack:
$$\sH om_{\cE_I^{sn} \times \kM_g}(\sX_I \times \kM_g, \cE_I^{sn} \times \kC) \times \sH om_{\cE_I^{sn}\times \kM_g}(\cE_I^{sn} \times \kC,\sS_I^{sn} \times \kM_g)$$
where for the properties of the Hom-stacks we refer to \cite{Homstack}.
Now, recall that in Subsection
\ref{subsection:construction:family:of:surfaces} (see Notation \ref{Def:sY}) we constructed a family of Weierstrass fibrations $(\sY,s\sS+\vec{a}\sF) \to \Kg^\circ$, with
universal curve $\sC_{\sY} \to \Kg^\circ$. 
Moreover, over $\Kg^\circ$ we have the morphism $\sY \to \sC_{\sY}$ and the section $\sC_{\sY} \to \sS$. Finally, recall that to keep track of the irreducible components of $\sF$, we put $n$ auxiliary sections $\sigma_i:\Kg^\circ \to \sY$.

This data induces a morphism
 $\Psi:\Kg^\circ \to \sH$.
Proceeding as in Subsection \ref{Subsection:construction:parameter:space:EI} we
define $\widehat{\cE_I}$ to be the closure of the image of $\Psi$.
Over $\widehat{\mathcal{W}_{I}}$, there are the following universal objects:
\begin{itemize}
 \item The pull back of $(\sX_I,\omega_{\sX_I/\cE_I^{sn}}^{\otimes m} \to \sL)$ that gives $(\cX,\omega_{\cX/\widehat{\cE_I}}^{\otimes m} \to \sG)$ (see Definition \ref{Def:KP:stable:pairs});
 \item The $n$ sections of $\cX \to \widehat{\cE_I}$, coming from those of $\sX_I^{sn} \to \cE_I^{sn}$ (see Subsection \ref{Subsection:construction:parameter:space:EI}); 
 \item The pull back of $\kC \to \kM_g$ that gives $\cC \to \widehat{\cE_I}$;
 \item Two universal morphisms, $\alpha:\cX \to \cC$ and $\beta:\cC \to \sS_I^{sn} \times_{\cE_I^{sn}}\widehat{\cE_I}$. 
\end{itemize}
\begin{Notation}\label{Def:the:section:is:a:section}
 Let $\widetilde{\cE_I}' \subseteq \widehat{\cE_I}$ be the locally closed substack where
 $\alpha \circ \beta$ is an isomorphism and $\beta$ is surjective. Let
 $\widetilde{\cE_I}$ be the seminormalization of $\widetilde{\cE_I}'$, and let
$\cS:=\sS_I^{sn} \times_{\cE_I^{sn}}\widetilde{\cE_I}$.
\end{Notation}
From the universal property of $\cE_I^{sn}$, there is a morphism $f:\widetilde{\cE_I} \to \cE_I^{sn}$.
We want to show that $f$ is an isomorphism.
\begin{Oss}In Notation \ref{Def:the:section:is:a:section}, it is tempting to look instead at the locus where $\alpha \circ \beta$ and $\beta$ are both isomorphisms. However, we do not know if $\sS_I^{sn} \to \cE_I^{sn}$ is a family of nodal curves, since a priori there might be non-reduced fibers (however, we know it for $\sS_I \to \cE_I$ from Corollary \ref{Cor:SI:is:a:fam:of:nodal:curves}). 
\end{Oss} 
\begin{Oss}
 From \cite{Homstack}, $\sH$ is locally of finite type over $\spec(k)$. Therefore also $\widetilde{\cE_I}$ is
 locally of finite type.
\end{Oss}

We start by describing the objects of $\widetilde{\cE_I}(\spec(k))$. One can deduce in a similar way the case
$\widetilde{\cE_I}(\spec(R))$ for every DVR $R$.
The groupoid $\widetilde{\cE_I}(\spec(k))$ has as object the quadruplets $((X,sS+\vec{a}F); C; \pi:X\to C; \sigma:C \to S)$
consisting of:
\begin{itemize}
 \item $(X,sS+\vec{a}F)$, an object of $\cE_I^{sn}(\spec(k))$;
 \item $C$, and object of $\mathfrak{M}_g(\spec(k))$;
 \item Two morphisms $\pi:X \to C$ and $\sigma:C \to S$ such that $\pi_{|S}\circ \sigma$ is an isomorphism and $\sigma$ is surjective.
\end{itemize}
We will not explicitly write the auxiliary sections, and we consider them as part of the data when we write $(X,sS+\vec{a}F)$. This should cause no confusion, since they will not play any significant role.

We also require an extra condition, since we are taking the closure of the image of $\Psi$.
We require that there is a DVR $R$, a
threefold pair $(\cX,s\cS+\vec{a}\mathcal{A}) $ which is an object
of $\cE_I^{sn}(\spec(R))$, and a family of prestable genus $g$ curves $\cC \to \spec(R)$, satisfying the following two conditions:
\begin{itemize}
 \item If $p$ (resp. $\eta$) is the closed (resp. open) point of $\spec(R)$, we require that there is a morphism $\cX \to \cC \to\spec(R)$
 which has a section $\cC \to \cS$; and there are isomorphisms
$(\cX,s\cS+\vec{a}\mathcal{F})_p \cong  (X,sS+\vec{a}F)$ and $\cC_p \cong C$ which make the obvious diagrams commutative;
\item $(\cX,s\cS+\vec{a}\mathcal{F})_\eta \to \cC_\eta \to \spec(k(\eta))$ and the section $\cC_\eta \to \cS_\eta$ are in the
image of $\Psi$.
\end{itemize}
A morphism between $((X,sS+\vec{a}F); C; \pi:X\to C; \sigma:C \to S)$ and
$((X',sS'+\vec{a}F'); C'; \pi':X'\to C'; \sigma':C' \to S')$ is the data of two isomorphisms $(f_1,f_2)$, with
$f_1:(X,sS+\vec{a}F) \to (X',sS'+\vec{a}F')$ and $f_2:C \to C'$, such that the obvious diagrams commute.
\begin{Prop}\label{Prop:the:auxiliary:stack:is:separated}
 The stack $\widetilde{\cE_I}$ is separated.
\end{Prop}
\begin{proof}
We use the valuative criterion for separatedness, \cite{LMB}*{Proposition 7.8}.
Let $R$ be a DVR, and let $\eta$ (resp. $p$) the generic (resp. closed) point of $\spec(R)$. Consider
two families $((X,sS+\vec{a}F)),C, \pi:X \to C, \sigma:C \to S)$ and
$((X',sS'+\vec{a}F')),C', \pi':X' \to C', \sigma':C' \to S')$ in $\widetilde{\cE_I}(\spec(R))$.
Assume there are two isomorphisms $h:(X,sS+\vec{a}F)_\eta \to (X',sS'+\vec{a}F')_\eta$ and $g:C_\eta \to C_\eta'$ such that the following two squares
commute:
$$\xymatrix{(X,sS+\vec{a}F)_\eta \ar[r]^{h} \ar[d] & (X',sS'+\vec{a}F')_\eta\ar[d] \\ C_\eta \ar[r]_{g} & C'_\eta} \text{ } \text{ }
\xymatrix{S_\eta \ar[r]^{h_{|S_\eta}} & S'_\eta \\ C_\eta \ar[r]_{g} \ar[u]^{\sigma_\eta} & C'_\eta \ar[u]_{\sigma'_\eta}}$$

where $\sigma_\eta$ and $\sigma'_\eta$ are the two sections. We need to find two isomorphisms $H$
and $G$ which extend $h$ and $g$ respectively, and witch make two corresponding diagrams commutative.

The moduli of stable pairs is separated, so we can find an isomorphism $H:(X,sS+\vec{a}F) \to (X',sS'+\vec{a}F')$.
Moreover, $H(S)= \overline{h(S_\eta)}=S'$, so there is an isomorphism $S \cong S'$.
Therefore we have an isomorphism $G:=\pi'\circ H \circ \sigma :C \to C'$.
To show that $H$ and $G$ induce
a morphism of $\widetilde{\cE_I}^{sn}(\spec(R))$ we just need to check the commutativity condition. Namely, we need to check that $G \circ \pi=\pi' \circ H$ and $(H_{|S})\circ \sigma = \sigma' \circ G$. But these are morphisms of separated and reduced schemes, and they agree when we restrict them to the generic fiber. Therefore they agree everywhere.
\end{proof}
Observe that coupling Proposition \ref{Prop:the:auxiliary:stack:is:separated} with Corollary \ref{Cor:description:lc:degeneration}, we have a description of the objects on the boundary of $\widetilde{\cE_I}$. We will use the following lemma in the proof of Proposition \ref{Prop:we:have:the:universal:curve}:
\begin{Lemma}\label{Lemma:extend:iso:last:section}
 Let $\alpha:=((X,sS+\vec{a}F); C; \pi:X\to C; \sigma:C \to S)$ and
$\beta:=((X',sS'+\vec{a}F'); C'; \pi':X'\to C'; \sigma':C' \to S')$ be two objects of $\widetilde{\cE_I}(\spec(k))$. Assume
that there is an isomorphism $f_1: (X,sS+\vec{a}F) \to (X',sS'+\vec{a}F')$. Then there is a unique $f_2:C \to C'$ such that
$(f_1,f_2)$ is an isomorphism $\alpha \to \beta$.
\end{Lemma}
\begin{proof}
  We need to find a morphism $f_2:C \to C'$ which makes these two diagrams
commutative:
$$\xymatrix{(X,sS+\vec{a}F) \ar[r]^{f_1} \ar[d]_{\pi} & (X',sS'+\vec{a}F')\ar[d]^{\pi'} \\ C \ar[r]_{f_2} & C'} \text{ } \text{ }
\xymatrix{S \ar[r]^{f_1} & S'\\ C \ar[r]_{f_2} \ar[u]^{\sigma} & C' \ar[u]_{\sigma'}}$$
Since $\sigma$ and $\sigma'$ are isomorphism, using the diagram one right, we need to show that
$f_2:=(\sigma')^{-1} \circ f_1 \circ \sigma$ makes the
diagram on the left commutative, i.e. we need to show that $\pi'\circ f_1=\pi' \circ f_1 \circ \sigma \circ \pi$. It is enough to show it on sets, since we are dealing with reduced separated schemes. For the same reason, it is enough to show that it commutes when restricted to a dense open subset.

We first show that if $G$ is an irreducible curve whose support is a fiber of $\pi$ which does not intersect the double locus and $\supp(\vec{a}F)$, then $f(G)$ is supported on a fiber of $\pi'$.
From Corollary \ref{Cor:description:lc:degeneration},
for every irreducible component $D$ of $C$, there is an irreducible elliptic component $Y$ of $X$ such that $\sigma(D)\subseteq Y$.
Moreover, let $S_Y:=S_{|Y}$, $F_Y:=F_{|Y}$ and let $E:=Y\cap \overline{(X\smallsetminus Y)}$ be the double locus.
Then $(Y,sS_Y+\vec{a}F_Y+E)$ is stable.

We can characterize the irreducible curves $G$ in $Y$ as above which are fibers as follows. We need to have $G^2=0$, $G.K_Y=0$ and $G\cap \supp(\vec{a}F_Y+E) = \emptyset$. Indeed, an irreducible fiber satisfies these requirements ($K_Y$ is supported on some fiber components). Moreover, if an irreducible multisection $M$ satisfies these requirements, then $\vec{a}F_Y=0$. Furthermore, since $M \cap \supp(E)=\emptyset$, $M$ passes through the intermediate components of the intermediate fibers. But then the fiber components intersected by $M$ are the fiber components intersected by $S_Y$, and since $(K_Y+\vec{a}F_Y+E).M=0$ we have $(K_Y+\vec{a}F_Y+E).S_Y=0$. Since $S_Y^2\le 0$ from \cite{AB3}*{Lemma B.1}, we have $(sS_Y+K_Y+\vec{a}F_Y+E).S_Y\le0$ which contradicts the stability assumption.
Then the irreducible fibers are determined by the surface pair $(X,sS+\vec{a}F)$, so $\pi$ and $\pi'$, generically, have the same fibers.

Now, for every $p \in X$ supported on a generic irreducible fiber for both $\pi$ and $\pi'$, we show that $(\pi'\circ f_1)(p)=(\pi' \circ  f_1  \circ \sigma  \circ  \pi) (p)$. This boils down to proving that $f_1(p)$ and $(f_1 \circ \sigma \circ \pi )(p)$ are in the same fiber for $\pi'$. But then it is enough to show that
$p$ and $(\sigma \circ \pi) (p)$ are in the same fiber for $\pi$, which follows since $\pi \circ \sigma = \Id$.
\end{proof}

\begin{Prop}\label{Prop:we:have:the:universal:curve}
 The morphism $f:\widetilde{\cE_I} \to \cE_I^{sn}$ is an isomorphism. 
\end{Prop}
\begin{Cor}\label{Cor:we:have:the:curve}
There is a family of curves $\sC_I \to \cE_I$, and a morphism $\sX_I \to \sC_I \to \cE_I$ satisfying the following condition. For every $\spec(k) \to \cE_I^\circ$, the corresponding morphism $(\sX_I)_p \to (\sC_I)_p$ is the morphism to a curve in the definition of a Weierstrass fibration.
\end{Cor}
\begin{proof}[Proof of Proposition \ref{Prop:we:have:the:universal:curve}]
The strategy is to apply Proposition \ref{proposition:iso:dm:stacks}. 
Let $\alpha:=((X,sS+\vec{a}F); C; \pi:X\to C; \sigma:C \to S)$ be an object of $\widetilde{\cE_I}(\spec(k))$.

\underline{$f$ is surjective on $k$-points}: This follows from Corollary \ref{Cor:description:lc:degeneration}.

\underline{$f$ is injective on $k$-points}: Given $\beta:=((X,sS+\vec{a}F); C'; \pi':X'\to C'; \sigma':C' \to S')$,
we need to show that $\alpha \cong \beta$. This follows from
Lemma \ref{Lemma:extend:iso:last:section}.

\underline{$f_\alpha:\Aut(\alpha) \to \Aut(f(\alpha))$ is bijective}: This follows again from Lemma \ref{Lemma:extend:iso:last:section}.

So $\widetilde{\cE_I}(\spec(k)) \to \cE_I^{sn}(\spec(k))$ is an equivalence.
We show that $\widetilde{\cE_I}$ is proper.
We already know it is separated (Proposition \ref{Prop:the:auxiliary:stack:is:separated}), we show now that $\widetilde{\cE_I}$ is quasi-compact.

\underline{$\widetilde{\cE_I}$ is quasi-compact:} Consider
$B \to \cE_I$ an atlas which is a scheme, and let $\sS_B:= \sS_I \times_{\cE_I}B$. From Corollary \ref{Cor:SI:is:a:fam:of:nodal:curves} the second projection $\sS_B \to B$ is a family of nodal curves.

Recall that the Hom-scheme $\operatorname{Hom}_B(\sX_B,\sS_B)$ is an open subscheme of $\amalg _{p(t)}\operatorname{Hilb}_{\sX_B \times_B \sS_B/B}^{p(t)}$ (see \cite{arbarello-cornalba-griffiths}*{IX.7}). In particular, it is a disjoint union of schemes of finite type.
Moreover, over $\operatorname{Hom}_B(\sX_B,\sS_B)$ there are the following universal objects:\begin{itemize}
\item A family of surfaces $\sX_H \to \operatorname{Hom}_B(\sX_B,\sS_B)$;

\item A family of curves $\sS_H \to \operatorname{Hom}_B(\sX_B,\sS_B)$ which has a closed embedding $\sS_H \to \sX_H$, and
\item A universal morphism $\phi:\sX_H \to \sS_H$.
\end{itemize}
Composing the closed embedding $\sS_H\to \sX_H$ and $\phi$, gives a map $g:\sS_H \to \sS_H$.
Consider the open subscheme $\operatorname{Hom}_B(\sX_B,\sS_B)^\circ \subseteq \operatorname{Hom}_B(\sX_B,\sS_B)$ where $g$ is an isomorphism.

Now, over $\Kg^\circ$ we have the family of surface pairs $(\sY,s\sS+\vec{a}\sF)$ which are stable Weierstrass fibrations; with a morphism $\sY \to \sC_\sY$ which is an isomorphism when restricted to $\sS$. From Observation \ref{Oss:definition:psi:KtoW} there is a morphism $\Psi:\Kg^\circ \to \cE_I^{sn}$, let $F:= \Kg^\circ \times_{\cE_I}B$. Observe that $F$ is of finite type. Let $\sY_F:=\sY \times_{\Kg^\circ} F$ and let $\sC_F:=\sC_\sY \times_{\Kg^\circ}F$. The morphism $\sY \to \sC_\sY$ induces $\sY_F \to \sC_F$ which in turn induces a morphism $\chi:F \to \operatorname{Hom}_B(\sX_B,\sS_B)^\circ$:

$$\xymatrix{\sY_F \ar[r]& \sC_F \ar[r] & F \ar[r] \ar[d] \ar@{.>}[ld]_-{\chi} & \Kg^\circ \ar[d] \\ &\operatorname{Hom}_B(\sX_B,\sS_B)^\circ \ar[r]& B \ar[r] & \cE_I}$$

Let $H$ be the closure of the image of $\chi$. Since $F$ is of finite type and $\operatorname{Hom}_B(\sX_B,\sS_B)^\circ$ is a disjoint union of schemes of finite type, the image of $\chi$ is contained in a closed subscheme of finite type of $\operatorname{Hom}_B(\sX_B,\sS_B)^\circ$. But then also $H$ is of finite type. From Corollary \ref{Cor:description:lc:degeneration} the morphism $H \to \cE_I$ is surjective, so also the composition $H \to \cE_I \to \cE_I^{sn}$ is surjective. But $H \to \cE_I^{sn}$ factors through $H \to \widetilde{\cE_I}$, and $f$ is an equivalence on points. Then
$H \to \widetilde{\cE_I}$ is surjective as well, and since $H$ is quasi-compact, also $\widetilde{\cE_I}$ is quasi-compact.

\underline{End of the proof:} Finally we have that $\widetilde{\cE_I}$ is of finite type. Then from \cite{LMB}*{Proposition 7.12, Remark 7.12.3}
and Corollary \ref{Cor:description:lc:degeneration}, the moduli space $\widetilde{\cE_I}$
is proper. So also $\widetilde{\cE_I} \to \cE_{I}^{sn}$ is proper, and from
Proposition \ref{proposition:iso:dm:stacks} the map $\widetilde{\cE_I}\to \cE_I^{sn}$ is an isomorphism.
\end{proof}

\subsection{Chambers with no pseudoelliptics} 
In this subsection we show that there are chambers such that, if $I$ belongs to such a chamber, the surface pairs parametrized by $\cE_I$ do not have any pseudoelliptic component (see Theorem \ref{Teo:intro:have:flat:fibration}). We start with the particular
case of a fixed tsm limit:
\begin{Prop}\label{Prop:no:psdo:fixed:tsm:limit}
 Let $(X',sS'+\vec{a}F') \to \spec(R)$ be a tsm limit. Then, given the vector $\vec{a}$, there is a positive $\widetilde{s}$ such that for every $s \le \widetilde{s}$,
 there are no pseudoelliptic components in the special fiber of the stable model of $(X',sS'+\vec{a}F')$.
\end{Prop}
\begin{proof}
 We show that we can choose $s$ small enough such that, if taking the stable model of $(X',sS'+\vec{a}F')$ we do a flip of
 La Nave, then the pseudoelliptic component generated by the flip gets contracted; and we perform no small contraction (see Theorem \ref{Teo:descripton:bir:trans:to:get:lcdeg}).  
 We proceed as in Subsection \ref{Subsection:alg:nature:stable:reduction}:
 let $(X^{(1)},sS^{(1)}+\vec{a}F^{(1)}):=(X',sS'+\vec{a}F')$, let $\eta$ (resp. $p$) be the generic (resp. closed) point
 of $\spec(R)$. Let $L^{(1)}$ be the lc divisor of
 $(X^{(1)},sS^{(1)}+\vec{a}F^{(1)})$, let $C^{(1)}_j$ be the irreducible components of $S^{(1)}_p$ and let $Y^{(1)}_j$ be the irreducible component
 of $X'$ that contains $C^{(1)}_j$. Observe that \emph{all the fibers of} $X^{(1)}$ \emph{are irreducible}. 
 
 We start by computing $(L^{(1)},C^{(1)}_j)$ for every $j$. If all these intersection numbers are positive, then
 $(X^{(1)},sS^{(1)}+\vec{a}F^{(1)})$ is a stable pair and the algorithm ends. Otherwise, say that
 $(L^{(1)},C^{(1)}_1)\le 0$. Then from Theorem \ref{Teo:descripton:bir:trans:to:get:lcdeg}
 we can add a $\mathbb{Q}$-Cartier $\mathbb{Q}$-divisor $G^{(1)}$, supported on some marked fibers, in order to make $(X^{(1)},sS^{(1)}+\vec{a}F^{(1)}+G^{(1)})$ a stable pair. We then replace the coefficients on the irreducible components of $G^{(1)}$ that intersect $C_1^{(1)}$ with 0, to get a new (unstable) pair $(X^{(1)},sS^{(1)}+\vec{a}F^{(1)}+\Gamma^{(1)})$. 
 So taking the stable model
 of $(X^{(1)},sS^{(1)}+\vec{a}F^{(1)}+\Gamma^{(1)})$, we need to contract $C^{(1)}$, and we will not contract any other $C^{(1)}_j$.
 Let $(Z,D)$ be such a stable model.
 From Theorem \ref{Teo:descripton:bir:trans:to:get:lcdeg}, either $Y^{(1)}$ will contract, or it becomes a pseudoelliptic surface.
 
 Assume $Y^{(1)}$ becomes a pseudoelliptic component $W$. Take $M \subseteq W$ an irreducible pseudomultisecion
 that does not meet the point to which $C^{(1)}_1$ contracts, and let $M' \subseteq Y^{(1)}$ be its proper transform.
 Since $(Z,D)$ is stable, we have $(K_Z+D).M>0$.
 Since $M$ is contained in the locus where $(Y^{(1)},sS^{(1)}+\vec{a}F^{(1)}+\Gamma^{(1)}) \dashrightarrow (Z,D)$ is an isomorphism,
 also $(L^{(1)}).M'>0$. 
 But all the fibers of $Y^{(1)}$ are irreducible,
 so from Observation \ref{Oss:all:fib:irred:implies:pulback}, there is a positive constant $c$ such that
 $L^{(1)}.M'=c(L^{(1)}-sS^{(1)}).C_1^{(1)}$. 
 Therefore we have that $L^{(1)}.C_1^{(1)}<0\text{ , but }(L^{(1)}-sS^{(1)}).C_1^{(1)}>0$. So there is a $s_0$ small enough such that
 for $t\le s_0$ we have $(L^{(1)}-(s-t)S^{(1)},C_1^{(1)})>0$. Namely, for any such $t$, we see that $C_1^{(1)}$ does not contract.
 Then we take
 $(X^{(2)},s_0S^{(2)}+\vec{a}F^{(2)})$, and we start this procedure again.
 
 We are left with the case in which $Y^{(1)}$ contracts. In that case, it either contracts with a divisorial contraction, or after a flip of La Nave.
 In either case, we define $(X^{(2)},sS^{(2)}+\vec{a}F^{(2)})$ to be the stable model of $(X^{(1)},sS^{(1)}+\vec{a}F^{(1)}+\Gamma^{(1)})$, without the
 markings on $\Gamma^{(1)}$. We see that $(X^{(2)},sS^{(2)}+\vec{a}F^{(2)})$ has all the components which are elliptic, with all the fibers irreducible. Then we start this procedure again, replacing $(X^{(1)},sS^{(1)}+\vec{a}F^{(1)})$ with $(X^{(2)},sS^{(2)}+\vec{a}F^{(2)})$.
 
 This procedure terminates in a finite number of steps, since there are finitely many irreducible components on $X^{(1)}$.
\end{proof}
\begin{Teo}\label{Teo:no:psudo}
 Let $(s,\vec{a},\beta)$ be a weight vector. Then we can choose a positive $\widetilde{s}$ such that for every $s \le \widetilde{s}$ and every point of $ \cE_{(s',\vec{a},\beta)}$,
the corresponding surface pair
has no pseudoelliptic components.
\end{Teo}
\begin{proof} We need to show that for \emph{every} tsm limit $(X',s'S'+\vec{a}F') \to \spec(R)$, we can choose $s'$ small enough such that to take the stable model of $(X',s'S'+\vec{a}F')$ we perform no small contraction, and if we need to perform $m$ flips of La Nave, we also need to contract $m$ pseudoelliptic components (see Theorem \ref{Teo:descripton:bir:trans:to:get:lcdeg}).
 From Proposition \ref{Prop:no:psdo:fixed:tsm:limit}, we can pick such an $s$ for a chosen tsm limit. Thus also for a finite set of tsm limits.
Proposition \ref{Prop:stratification:exists} and
Theorem \ref{Teo:stable:model:depends:only:on:num:data} prove the result.
\end{proof}

\end{document}